\documentclass[11pt]{article}
\usepackage{csquotes}
\usepackage{amsmath,amsthm,mathtools}
\usepackage{amsfonts}
\usepackage{amssymb}
\usepackage{graphicx}
\theoremstyle{plain} 
\newtheorem{theorem}{Theorem}[section]
\newtheorem{lemma}{Lemma}[section]
\newtheorem{coro}{Corollary}[section]
\newtheorem{prop}{Proposition}[section]

\theoremstyle{definition} 
\newtheorem{definition}{Definition}[section]
\newtheorem{assumption}{Assumption}[section]

\theoremstyle{remark} 
\newtheorem{remark}{Remark}[section]
\newtheorem{example}{Example}[section]
\numberwithin{equation}{section}

\usepackage{tikz}
\usepackage{tikz-3dplot}
\usepackage{amsmath}
\usepackage{graphicx}
\usepackage{caption}
\usepackage{subfig}

\usetikzlibrary{calc}

\usepackage{algorithm}
\usepackage{algorithmicx}
\usepackage{algpseudocode}


\usepackage{booktabs, float, multicol, multirow, diagbox} 
\usepackage{color} 

\numberwithin{table}{section}
\numberwithin{figure}{section}
\date{}

\begin{document}
	\makeatother
	\title{The Geometry of Admissible Short Selling in Discrete-Time Stochastic Portfolio Theory
		\thanks{Corresponding author.Email: xcui@nju.edu.cn\\
			Xiaojun Cui and Jilong Xu are supported by the National Natural Science Foundation
			of China (Grant No. 12171234), the Project Funded by the Priority Academic Program
			Development of Jiangsu Higher Education Institutions (PAPD) and the Fundamental
			Research Funds for the Central Universities.}
		\author{Jilong Xu$^a$, Xiaojun Cui$^{b,*}$\\
			\small\em $^a$ School of Mathematics, Nanjing University, Nanjing {\rm 210093}, P. R. China\\
			\small\em $^b$ School of Mathematics, Nanjing University, Nanjing {\rm 210093}, P. R. China\\
			}}
	\maketitle
	\begin{abstract}
		While discrete-time Stochastic Portfolio Theory (SPT) provides a robust framework for market analysis, existing work on functional generation has predominantly focused on long-only portfolios defined on the entire unit simplex. This paper extends the geometric framework of functional generation to the broader class of \emph{bankruptcy-proof long-short portfolios} defined on local market state spaces. We establish that, within this admissible setting, pseudo-arbitrage is fully characterized by the concavity of the generating function on the market state space, thereby relaxing the usual global domain requirement.
		
		A central contribution of this work is a geometric characterization of the short-selling mechanism. We prove that the presence of short selling is equivalent to the negativity of the \emph{maximal concave extension} of the generating potential. This phenomenon is linked to the steepness of the logarithmic gradient as the market approaches a zero boundary nested inside the simplex. To systematically exploit this mechanism, we introduce the \emph{barycentric scaling transformation}, a constructive methodology that maps classical long-only generating functions onto restricted domains to engineer admissible strategies with controlled short-selling exposure.
		
		Finally, through the analysis of specific shrunken portfolios, we identify a \emph{geometric phase transition}: under suitable boundary conditions, admissible strategies exhibit a long-only core and a short-selling region in a qualitative sense (without asserting an exact partition of the state space). This provides a unified geometric perspective on relative arbitrage beyond the long-only constraint.
		
	\end{abstract} 
	
	\bigskip
	\section{Introduction}
	
	\subsection{Background}
	In an equity market with $n\ge 2$ stocks, we consider self-financing trading strategies that are fully invested in the sense that the portfolio weights sum to one at all times. Short selling is permitted, while borrowing from or lending to a money market account is prohibited. The natural state space for (long-only) market weights is the open unit simplex in $\mathbb{R}^n$,
	\[
	\Delta^{(n)}=\left\{p=(p_1,\ldots,p_n)\in\mathbb{R}^n:\ p_i>0,\ \sum_{i=1}^n p_i=1\right\}.
	\]
	Let $H=\left\{p\in\mathbb{R}^n:\ \sum_{i=1}^n p_i=1\right\}$ be the affine hyperplane containing $\Delta^{(n)}$.
	
	\paragraph{Convention.}
	Throughout, all topological and convex-analytic notions on $\Delta^{(n)}$ (e.g., open/closed sets, closure, boundary, continuity, concavity, directional derivatives, and superdifferentials) are understood \emph{relative to the affine space} $H$. Equivalently, we work in the translated linear space
	\[
	T:=\left\{x\in\mathbb{R}^n:\ \sum_{i=1}^n x_i=0\right\},
	\]
	and whenever a supergradient is used, we identify it with its representative in $T$ (i.e., we impose the normalization $\sum_{i=1}^n x_i=0$).
	
	In this context, a \emph{long-short portfolio} is represented by a vector in $H$. If the portfolio vector lies in the closure  $\overline{\Delta^{(n)}}$ of $\Delta^{(n)}$, then it is \emph{long-only}. A portfolio with at least one negative weight (that is, a vector in $H\setminus\overline{\Delta^{(n)}}$) is said to involve \emph{short selling}. The main purpose of this paper is to investigate the geometric structure and the relative arbitrage properties of admissible strategies that allow short positions.
	
	Following the discrete-time framework of \cite{Wong2015OptimizationofRelativeArbitrage}, we model the market by a sequence of market weight vectors $\{\mu(t)\}_{t=0}^{\infty}\subset \Delta^{(n)}$. For each $t\in\mathbb{N}$ and each stock $i\in\{1,\ldots,n\}$, the market weight $\mu_i(t)$ is defined in terms of the capitalization process $X_i(t)>0$ as
	\begin{equation}
		\label{eq:MarketPortfolio}
		\mu_i(t)=\frac{X_i(t)}{\sum_{j=1}^n X_j(t)}.
	\end{equation}
	The vector $\mu(t)=(\mu_1(t),\ldots,\mu_n(t))$ corresponds to the \emph{market portfolio}, which serves as the benchmark strategy.
	
	Stochastic Portfolio Theory (SPT), as consolidated in the monograph of Fernholz \cite{Fernholz1998PortfoliogeneratingFunctions,Fernholz2002SPT}, was originally developed as a descriptive framework for studying the structure and macroscopic properties of equity markets; see, for example, \cite{Campbell2025MacroscopicProperties}. Much of the early literature therefore focused on long-only portfolios, reflecting the capitalization-weighted construction of standard equity indices. However, long-only constraints are not intrinsic to the functional generation methodology. In particular, the ``decomposition equation'', which links relative portfolio performance to a drift term associated with a generating function, holds under more general conditions.
	
	Karatzas and Ruf \cite{Karatzas2017Lyapunov} introduced \emph{additive functional generation} of trading strategies by interpreting generating functions as Lyapunov functions of the market weight process. This approach accommodates strategies with short positions; to preclude bankruptcy, one typically imposes a positivity constraint on the (relative) value process . 
	Building on this Lyapunov framework, Fernholz, Karatzas and Ruf \cite{Fernholz2018VolatilityandArbitrage} characterized the relationship between market volatility and relative arbitrage.
	In conjunction with the additive generation framework of Karatzas and Ruf \cite{Karatzas2017Lyapunov}, which allows trading strategies with short positions, this provides a natural setting for understanding how volatility can be converted into relative performance beyond the long-only class.
	More recently, functional generation has been extended to certain path-dependent functionals via signature methods \cite{Cuchiero2023SignatureMethods}, further illustrating the breadth of the framework.
	
	In discrete time, Wong \cite{Wong2019InfromationGeometryinPortfolio} generalized the additive generation framework of Karatzas and Ruf to a geometric setting that permits both long and short positions. Without additional constraints, the corresponding portfolio value process need not remain strictly positive. Despite this general formulation, much of the subsequent literature continues to impose long-only constraints when studying optimization and relative arbitrage. 
	For instance, in their work on functional portfolio optimization, Campbell and Wong \cite{Wong2023FunctionalPortfolioOptimization} focus on multiplicatively generated long-only portfolios. This formulation leaves open how the functional portfolio optimization framework might be extended to allow short positions. In this regard, long-only frameworks, such as the geometric characterization of relative arbitrage by Pal and Wong \cite{Wong2016TheGeometryofRelativeArbitrage}, may provide useful starting points for developing admissible long-short portfolios.
	
	In the long-only setting, the optimal transport formulation in \cite{Wong2016TheGeometryofRelativeArbitrage} can be viewed as inducing a dual (primal/dual) coordinate system on the simplex. Pal and Wong \cite{Pal2018ExponentiallyCF} developed the associated geometric framework; see also the \emph{$L$-divergence} \cite{Wong2021ProjectionswithLogarithmicDivergences,Wong2019LogarithmicDivergencesCurvature}, which together lead to a geometry in which portfolio rebalancing can be described via geodesic structures and their geometric relationships. This perspective also underpins applications to information geometry \cite{Wong2018LogarithmicDivergencesOptimalTransport} and the development of conformal mirror descent for continuous-time optimization \cite{Kainth2024ConformalMirrorDescentwithLogarithmicDivergences}. Extending this line of work, Kainth et al.\ \cite{Kainth2025BregmanWassersteinDivergence} established the Bregman--Wasserstein geometry, which lifts finite-dimensional Bregman divergence geometry to the space of probability distributions, thereby connecting finite-dimensional divergence geometry with infinite-dimensional optimal transport.
	
	While more advanced differential geometric aspects (such as dual connections and infinite-dimensional manifold structures) are beyond the scope of this paper, extending the basic \emph{functional generation} mechanism to admissible strategies that allow short selling remains an important open direction. The goal of this paper is to generalize the \emph{portfolio construction} mechanism and the \emph{relative arbitrage characterization} \cite[Theorem~1]{Wong2016TheGeometryofRelativeArbitrage} from the long-only setting to a broader class of admissible long-short portfolios.
	
	\subsection{Diversity Assumption}
	To preclude degenerate market behaviors, we adopt the standard concept of \emph{market diversity} from Stochastic Portfolio Theory \cite{Fernholz2002SPT}. Specifically, we assume that the market weight process stays uniformly away from the boundary of the simplex.
	\begin{assumption}[Diversity]
		\label{ass:Diversity}
		There exists a nonempty open convex subset $K \subset \Delta^{(n)}$, referred to as the \emph{market state space}, with  $\partial K\cap \partial \Delta^{(n)}=\varnothing$, such that $\mu(t)\in K$ for all $t\ge 0$.
	\end{assumption}
	\begin{remark}
		Assumption \ref{ass:Diversity} holds, for instance, if the set of market weights has compact closure in the simplex, that is, if $\overline{\{\mu(t): t\ge 0\}}\subset \Delta^{(n)}$. In financial terms, this excludes the possibility that a stock's market weight vanishes, and it also rules out market concentration where a single stock captures the entire market.
	\end{remark}
	
	This boundedness condition is structurally essential for the existence of admissible short-selling strategies, which we formalize as \emph{bankruptcy-proof portfolios} in Definition \ref{def:Bankruptcy-ProofPortfolio}. As we will establish in Proposition \ref{prop:ConvergenceToLongOnly}, if the market state space is allowed to approach the entire open simplex, the set of admissible long-short strategies reduces to the set of long-only portfolios, which has been studied comprehensively.
	
	Formally, a \emph{portfolio function} is a map $\pi : D \rightarrow H$, where $D$ is a convex subset of $\Delta^{(n)}$. Given a current market weight $p\in D$, the investor holds the portfolio vector $\pi(p)\in H$. Under Assumption \ref{ass:Diversity}, we focus on the bankruptcy-proof portfolios $\pi : K \rightarrow H$ in this paper. A fundamental class of such strategies is the \emph{constant-weighted portfolio} (CWP), generated by a fixed vector $p \in H$.
	
	When $p \in \overline{\Delta^{(n)}}$, this corresponds to the well-studied long-only CWP. Following Cover and Ordentlich's seminal work on universal portfolios \cite{Cover1996UniversalPortfolios}, which leveraged the asymptotic properties of CWPs to achieve optimal growth, this class of strategies has inspired extensive research extensions \cite{He2022UniversalPortfolio, Walk2013Cover'sAlgorithmModified}. Moreover, CWPs continue to play a central role in modern theoretical developments, serving as a bridge between SPT, model-free finance \cite{Cuchiero2019Cover'sUniversalPortfolioSPT}, and information geometry \cite{Campbell2025MathematicalStudyExcessGrowth}. Given their foundational status in long-only theory, it is essential to rigorously examine the behavior of CWPs within our broader long-short framework. We will address this investigation in detail in Section \ref{subsec:CWP_MCM}.
	
	\subsection{Geometric Motivation: The Long-Only Case}
	Before presenting our main results, we briefly recall the functional generation framework. A functionally generated portfolio is specified by a positive concave generating function. Let $\Phi: \Delta^{(n)} \rightarrow (0, \infty)$ be a positive concave function. We define the associated logarithmic potential as $\varphi \coloneqq \log \Phi$. If $\Phi$ is differentiable, the portfolio weights $\pi$ generated by $\Phi$ are given by the gradient map:
	\begin{equation}
		\pi_i(p) = p_i \left( 1 + \langle \nabla \varphi(p), e_i - p \rangle \right), \quad i = 1, \ldots, n,
		\label{eq:FGDifferential}
	\end{equation}
	The mapping $p\mapsto \pi(p)$ may be viewed as an affine adjustment of the gradient field that enforces the budget constraint $\sum_{i=1}^n \pi_i(p)=1$, that is, $\pi(p)\in H$.
	In the information geometry literature \cite{Pal2018ExponentiallyCF, Wong2021ProjectionswithLogarithmicDivergences, Wong2019LogarithmicDivergencesCurvature}, authors work directly with the logarithmic potential $\varphi$. Such functions are termed \emph{exponentially concave}, a concept intimately linked to the \emph{$L$-divergence} (formally defined in Definition \ref{def:L_divergence}) which generates a dually flat geometry on the simplex. However, since the finer differential geometric structures (for example, dual connections) are beyond the scope of this paper, we primarily work with $\Phi$ and only invoke $\varphi$ when discussing divergence properties.
	
	Equation \eqref{eq:FGDifferential} serves as the primary motivation for our work. It reveals that the construction of a portfolio $\pi$ does not structurally require the generating function to be defined on the entire simplex. It suffices that $\Phi$ is defined on a convex set containing the practical market weights.
	In the classical setting of \cite{Wong2016TheGeometryofRelativeArbitrage}, Pal and Wong specifically study positive concave functions $\Phi$ that are defined on the \emph{entire} simplex $\Delta^{(n)}$. Under this global assumption, it is a known result (Proposition 5 in \cite{Wong2016TheGeometryofRelativeArbitrage}) that the portfolio $\pi$ defined by equation \eqref{eq:FGDifferential} is guaranteed to be long-only.
	
	Consequently, the restriction to long-only strategies is not an intrinsic feature of the functional generation framework; rather, it arises naturally from the requirement that the generating function be defined globally on the entire simplex $\Delta^{(n)}$. Relaxing the global assumption implies that the market must satisfy local constraints, but simultaneously opens the possibility for the strategy to exhibit negative weights (short positions), which is the focus of this paper.
	
	The performance of such portfolios is governed by the \emph{$L$-divergence} (Definition \ref{def:L_divergence}), $T_{\varphi}(q\mid p)$, which measures the convexity of the generating potential. This leads to the canonical pathwise decomposition of the \emph{relative wealth process} $V_\pi(t)$:
	\begin{equation}
		\log V_\pi(t) = \log \frac{\Phi(\mu(t))}{\Phi(\mu(0))} + \sum_{k=0}^{t-1} T_\varphi(\mu(k+1) \mid \mu(k)).
		\label{eq:Decomposition}
	\end{equation}
	Here $T_{\varphi}$ is nonnegative and captures the effect often referred to as ``volatility harvesting''.
	To characterize the growth capability of these strategies, we adopt the standard notion of \emph{pseudo-arbitrage} (see Definition \ref{def:pseudo-arbitrage} in Section \ref{subsec:Bankruptcy-proofPortfolios} for the formal definition), which requires bounded downside and unbounded upside potential.
	
	Connecting the decomposition \eqref{eq:Decomposition} with this definition yields the fundamental characterization for the long-only case. We restate Pal and Wong's theorem in our notation to highlight the role played by the domain of the generating function.
	
	\begin{theorem}[Characterization of Long-Only Pseudo-Arbitrage, \cite{Wong2016TheGeometryofRelativeArbitrage}]
		\label{thm:LongOnlyArbitrage}
		A long-only portfolio map $\pi:\Delta^{(n)} \rightarrow \overline{\Delta^{(n)}}$ constitutes a pseudo-arbitrage on a nonempty open convex subset $K \subset \Delta^{(n)}$ if and only if it is generated by a positive concave function $\Phi$ such that $\Phi$ extends to a non-negative concave function on the \emph{entire} simplex $\Delta^{(n)}$.
	\end{theorem}

	\subsection{Summary of the main results}
	This paper extends the geometric framework of functional generation to the broader class of \emph{bankruptcy-proof} portfolios (Definition \ref{def:Bankruptcy-ProofPortfolio}). Our main contributions are summarized as follows.
	
	\begin{enumerate}
		\item \textbf{Characterization of Admissible Arbitrage.}
		We show that, for bankruptcy-proof portfolios, pseudo-arbitrage on $K$ is characterized by concavity of the generating function on $K$.
		
		\begin{theorem}[Characterization of Admissible Pseudo-Arbitrage]
				\label{thm:AdmissibleArbitrage}
				Let $K$ be the market state space in Assumption~\ref{ass:Diversity}, and let $D\subset\Delta^{(n)}$ be a nonempty open convex set with $K\subset D$. A bankruptcy-proof portfolio $\pi:D\to H$ is a pseudo-arbitrage on $K$ if and only if there exists a function $\Phi:D\to[0,\infty)$ such that:
				\begin{enumerate}
					\item[(i)] $\Phi|_K$ is bounded away from zero on $K$;
					\item[(ii)] $\Phi|_K$ is not affine on $K$;
					\item[(iii)] $\Phi|_K$ is concave on $K$;
					\item[(iv)] $\pi|_K$ is generated by $\Phi|_K$ in the sense of Definition~\ref{def:FGP_Admissible}.
				\end{enumerate}
		\end{theorem}
		
		\begin{remark}
			If one takes $D=\Delta^{(n)}$, then Proposition \ref{prop:ConvergenceToLongOnly} implies that any bankruptcy-proof portfolio $\pi:\Delta^{(n)}\to H$ must in fact be long-only. In this case, Theorem \ref{thm:AdmissibleArbitrage} reduces to the long-only characterization of pseudo-arbitrage, and its statement agrees with \cite[Theorem~1]{Wong2016TheGeometryofRelativeArbitrage}; see Theorem \ref{thm:LongOnlyArbitrage}.
		\end{remark}

		\item \textbf{Geometric Criterion for Short Selling.}
		We solve the identification problem for short-selling strategies by introducing the \emph{maximal concave extension} $\hat{\Phi}$ of the generating function from $K$ to $\Delta^{(n)}$. We show that $\pi$ involves short selling on $K$ if and only if $\hat{\Phi}$ takes a negative value somewhere on $\Delta^{(n)}$.
		
		\begin{theorem}[Equivalence of Short Selling and Negative Extension]
				\label{thm:ShortSellingEquivalence}
				Let $K$ be the market state space in Assumption~\ref{ass:Diversity}$,$ and let $\Phi:K\to(0,\infty)$ be a differentiable concave function generating a portfolio $\pi:K\to H$. Then $\pi$ involves short selling on $K$ if and only if its maximal concave extension $\hat{\Phi}$ satisfies $\hat{\Phi}(x)<0$
				for some $x\in\Delta^{(n)}$.
		\end{theorem}
		
		\item \textbf{Geometric Phase Transition of Shrunken Portfolios.}
		We provide a rigorous description of the spatial structure of functionally generated portfolios. By analyzing the generating potential's boundary behavior, we identify a long-only core and, under a boundary condition, a region near the boundary where short selling occurs. These regions are described qualitatively and are not claimed to form a partition of the domain.
		
		\begin{theorem}[Geometric Phase Transition of Shrunken Portfolios]
			\label{thm:GeometricPhaseTransition}
			Let $D\subset\Delta^{(n)}$ be a symmetric open convex set, and let $\Phi:\overline{D}\to[0,\infty)$ be a symmetric continuous concave function that is strictly positive and continuously differentiable on $D$. Assume that there exists a boundary point $q\in\partial D\cap\Delta^{(n)}$ such that $\Phi(q)=0$. Let $\pi:D\rightarrow H$ be the portfolio generated by $\Phi$, equivalently represented by \eqref{eq:FGDifferential_Local}. Then:
			\begin{enumerate}
				\item[(i)] There exists an open neighborhood $U\subset D$ of the barycenter $\bar e$ such that $\pi_i(p)>0$ for all $i=1,\ldots,n$ and all $p\in U$.
				\item[(ii)] If, for some $i$, the point $q$ lies in the relative interior of the segment joining a point $p_0\in D$ and the vertex $e_i$, then there exists a point $p^\ast\in[p_0,q)$ such that
				\[
				\pi_i(p)<0,\quad p\in[p^\ast,q)\cap D.
				\]
			\end{enumerate}
		\end{theorem}

		\item \textbf{Applications: The SEWP and SEP.}
		We bridge theory and practice by introducing the \emph{Barycentric Scaling Transformation} (Definition \ref{def:ScalingTransformation}). Through the construction of the Shrunken Equal-Weighted Portfolio (SEWP) and Shrunken Entropy Portfolio (SEP), we demonstrate how this methodology allows for the systematic design of admissible strategies with controlled short selling exposure. In particular, for the SEWP we further apply the analytical framework of Theorem \ref{thm:GeometricPhaseTransition} to its specific domain and provide qualitative characterizations of the long-only core and the shell where short selling occurs; see subsection \ref{subsubsec:SEWP}, in particular Propositions \ref{prop:EWPAllLongArea} and \ref{prop:ShortCondition}.
	\end{enumerate}
	
	\section{Admissible Portfolios and Market Stability} 
	\label{sec:AdmissiblePortfolios}
	
	The generalization of the geometric framework to the entire hyperplane $H$ aligns with fundamental economic insights regarding market efficiency. Miller \cite{Miller1997RiskUncertaintyDivergence} famously argued that when short selling is restricted, asset prices tend to reflect only the views of optimistic investors, leading to potential overvaluation. Furthermore, Diamond and Verrecchia \cite{Diamond1987ConstraintsShort-selling} demonstrated that short-sale constraints reduce the speed at which prices adjust to private information. Therefore, a comprehensive theory of arbitrage must explicitly account for strategies that permit short positions in order to correct such mispricing and exploit a broader range of market dynamics.
	
	However, extending the geometric framework to include short positions introduces significant risks. Unlike long-only strategies, a portfolio with negative weights faces the possibility of unlimited liability if the shorted assets experience sharp appreciation. To maintain the viability of the strategy, we restrict attention to \emph{admissible strategies}, a concept rooted in the work of Harrison and Pliska \cite{Harrison1981MartingalesandStochasticIntegralsContinuousTrading}. In our setting, admissibility means that the (relative) portfolio value remains strictly positive at all times. This implies that any short position must be sufficiently collateralized by the long positions. The imposition of market assumptions, such as Assumption \ref{ass:Diversity}, provides the necessary regularity to analyze these admissible strategies without relying on specific probabilistic models.
	
	\subsection{Bankruptcy-proof Long-Short Portfolios} 
	\label{subsec:Bankruptcy-proofPortfolios}
	
	To formalize the set of admissible portfolios, we first recall the dynamics of the relative wealth process. Let $V_\pi(t)$ denote the ratio of the wealth of a portfolio $\pi$ to that of the market portfolio $\mu$ at time $t$. In a discrete-time setting, the evolution of this relative value is governed by the portfolio's one-period relative return. Specifically, as noted in \cite[Lemma 2.1]{Pal2016energyentropyarbitrage}, the dynamics of a long-only portfolio $\pi$ are given by:
	\begin{equation}
		V_\pi(0) = 1, \quad \frac{V_\pi(t+1)}{V_\pi(t)} = \left\langle \pi(\mu(t)), \frac{\mu(t+1)}{\mu(t)} \right\rangle,
		\label{eq:RelativeValue}
	\end{equation}
	where $\frac{\mu(t+1)}{\mu(t)}$ denotes the element-wise ratio vector of market weights.
	
	Unlike long-only strategies, which inherently satisfy $\langle \pi, q/p \rangle > 0$ due to the positivity of weights, a long-short portfolio faces the risk of bankruptcy. If the inner product in \eqref{eq:RelativeValue} becomes non-positive, the relative wealth drops to zero or becomes negative, signifying ruin. Therefore, to ensure the strategy remains solvent ($V_\pi(t) > 0$) for all possible market trajectories allowed by the diversity assumption, we impose the following strict positivity condition.
	
	\begin{definition}[Bankruptcy-proof Portfolio]
		\label{def:Bankruptcy-ProofPortfolio}
		Let $D$ be a nonempty open convex subset of $\Delta^{(n)}$. A portfolio map $\pi: D \rightarrow H$ is termed \emph{bankruptcy-proof} if, for all current market weights $p \in D$ and all possible next-period weights $q \in D$, it satisfies
		\begin{equation}
			\left\langle \pi(p), \frac{q}{p} \right\rangle > 0.
			\label{eq:NoBankruptcyCondition}
		\end{equation}
		The set of all such bankruptcy-proof portfolios is denoted by
		\begin{equation}
			\mathcal{BP}(D) = \left\{ \pi: D \rightarrow H \ \big| \ \forall p, q \in D, \ \left\langle \pi(p), \frac{q}{p} \right\rangle > 0 \right\}.
		\end{equation}
		
		Furthermore, for a fixed asset $i$, we define the class of \emph{$i$-bankruptcy-proof} portfolios, denoted $\mathcal{BP}_i(D)$, as the subset of $\mathcal{BP}(D)$ consisting of strategies that maintain long positions in all assets except possibly asset $i$:
		\begin{equation}
			\mathcal{BP}_i(D) = \left\{ \pi \in \mathcal{BP}(D) : \pi_j(p) \ge 0, \ \forall j \neq i, \ \forall p \in D \right\}.
		\end{equation}
	\end{definition}
	The condition \eqref{eq:NoBankruptcyCondition} is the geometric guarantee of solvency. In this paper, we use the terms \emph{bankruptcy-proof} and \emph{admissible} interchangeably.
	
	Under the bankruptcy-proof condition, equation \eqref{eq:RelativeValue} holds for long-short portfolios, and the definition of \emph{pseudo-arbitrage} \cite[Definition 1]{Wong2016TheGeometryofRelativeArbitrage} is naturally inherited.
	
	\begin{definition}[\emph{Pseudo-arbitrage}]
		\label{def:pseudo-arbitrage}
		Let $K$ be a market state space satisfying Assumption \ref{ass:Diversity}. Let $D\subset \Delta^{(n)}$ be an open convex set with $K\subset D$. A bankruptcy-proof portfolio $\pi: D \rightarrow H$ is called a pseudo-arbitrage on $K$ if:
		\begin{enumerate}
			\item There exists a non-negative constant $C = C(K, \pi)$ such that for every sequence of market weights $\{\mu(s)\}_{s=0}^{\infty} \subset K$, we have $\log V_\pi(t) \ge -C$ for all $t \ge 0$.
			\item There exists at least one market weight sequence $\{\mu(t)\}_{t=0}^{\infty}\subset K$ such that $\lim\limits_{t\to\infty}V_{\pi}(t)=\infty$.
		\end{enumerate}
	\end{definition}
	
	\subsection{Properties of Bankruptcy-proof Portfolios}
	\label{subsec:PropertiesofBankruptcy-proofPortfolios}
	The definition of bankruptcy-proof portfolios is not merely a technical constraint; it induces a rich geometric structure. The following three propositions characterize this structure, highlighting the trade-off between market uncertainty and the ability to maintain short positions.
	
	Proposition \ref{prop:Existence} establishes that the set of admissible strategies is convex and, crucially, strictly larger than the long-only set under Assumption \ref{ass:Diversity}. Proposition \ref{prop:Monotonicity} reveals that as market uncertainty (the size of the domain $D$) increases, the set of admissible strategies shrinks. Finally, Proposition \ref{prop:ConvergenceToLongOnly} provides the theoretical justification for Assumption \ref{ass:Diversity}: it proves that in the absence of a restricted state space (i.e., as uncertainty approaches the entire simplex), short selling becomes mathematically impossible without risking bankruptcy.
	
	\begin{prop}[Convexity and Existence]
		\label{prop:Existence}
		For any nonempty open convex subset $D$ of $\Delta^{(n)}$ and market state space $K$ satisfying Assumption \ref{ass:Diversity}:
		\begin{enumerate}
			\item[(i)] The sets $\mathcal{BP}(D)$ and $\mathcal{BP}_i(D)$ are nonempty convex sets.
			\item[(ii)] $\mathcal{BP}_i(K)$ contains portfolios that involve short selling in asset $i$.
		\end{enumerate}
	\end{prop}
	\begin{proof}
		\begin{enumerate}
			\item[(i)] Convexity follows immediately from the linearity of the inner product. It is immediate that the set of long-only portfolios is contained in $\mathcal{BP}_i(D)$, as the inner product of non-negative vectors with strictly positive relative prices is strictly positive.
			\item[(ii)] The bankruptcy-proof condition $\langle \pi(p), q/p \rangle > 0$ imposes a strict linear inequality on $\pi(p)$ for each pair $(p,q)$. Consequently, a long-only portfolio with $\pi_i(p)=0$ can be perturbed, for sufficiently small $\varepsilon>0$, into a portfolio with $\pi_i(p)=-\varepsilon$ while preserving strict positivity. This is possible because Assumption \ref{ass:Diversity} implies that the family of market weight ratios $\{q/p:\ p,q\in K\}$ is bounded in $\mathbb{R}^n$.
		\end{enumerate}
	\end{proof}
	
	\begin{example}[Construction of a Short-Selling Strategy]
		\label{ex:ShortSellingConstruction}
		To explicitly demonstrate the existence of short-selling strategies in $\mathcal{BP}_1(K)$ (as guaranteed by Proposition \ref{prop:Existence}), we utilize the bounds on market transitions. By Assumption \ref{ass:Diversity}, there exists a constant $M > 1$ such that the market weight ratios are bounded by $1/M \le q_j/p_j \le M$ for all $1\le j\le n$.
		
		Consider a constant portfolio $\pi$ defined by:
		\[
		\pi_1 = -\frac{1}{2M(M - 1/M)}, \quad \text{and} \quad \pi_j = \frac{1 - \pi_1}{n - 1} \quad \forall j \ge 2.
		\]
		Here, asset 1 is shorted ($\pi_1 < 0$) and the proceeds are distributed among the remaining assets. To verify this is bankruptcy-proof, we analyze the worst-case scenario where the shorted asset appreciates maximally ($q_1/p_1 = M$) and the long assets depreciate maximally ($q_j/p_j = 1/M$):
		\[
		\begin{aligned}
			\left\langle \pi, \frac{q}{p} \right\rangle &= \pi_1 \frac{q_1}{p_1} + \sum_{j=2}^{n} \pi_j \frac{q_j}{p_j} \\
			&\ge \pi_1 M + (1-\pi_1)\frac{1}{M} \\
			&= \pi_1 \left( M - \frac{1}{M} \right) + \frac{1}{M}.
		\end{aligned}
		\]
		Substituting the value of $\pi_1$:
		\[
		\left\langle \pi, \frac{q}{p} \right\rangle \ge -\frac{1}{2M(M - 1/M)} \cdot \left( M - \frac{1}{M} \right) + \frac{1}{M} = -\frac{1}{2M} + \frac{1}{M} = \frac{1}{2M} > 0.
		\]
		Thus, $\left\langle \pi, \frac{q}{p}\right\rangle$ remains strictly positive for all valid market weight ratios, confirming $\pi\in\mathcal{BP}_1(K)$.
	\end{example}
	
	\begin{prop}[Monotonicity with respect to Market Uncertainty]
		\label{prop:Monotonicity}
		Let $D_1$, $D_2$ be two nonempty open convex subsets of $\Delta^{(n)}$. If $D_1 \subset D_{2}$, then the corresponding sets of bankruptcy-proof portfolios are decreasing:
		\begin{equation}
			\mathcal{BP}(D_2) \subset \mathcal{BP}(D_1).
		\end{equation}
	\end{prop}
	
	\begin{proof}
		Let $\pi \in \mathcal{BP}(D_2)$. By definition, $\langle \pi(p), q/p \rangle > 0$ for all pairs $p, q \in D_2$. Since $D_1 \subset D_2$, the inequality holds in particular for all $p,q\in D_1$. Thus, $\pi \in \mathcal{BP}(D_1)$.
	\end{proof}
	
	\begin{prop}[Convergence to Long-Only]
		\label{prop:ConvergenceToLongOnly}
		Let $\{D_m\}_{m=1}^{\infty}$ be a sequence of strictly increasing nonempty open convex subsets of $\Delta^{(n)}$, such that $\lim_{m\to\infty}D_m=\cup_{m=1}^{\infty}D_m=\Delta^{(n)}$. Then, a portfolio map $\pi:\Delta^{(n)}\to H$ is bankruptcy-proof on $\Delta^{(n)}$ if and only if it is long-only. Equivalently,
		the set of bankruptcy-proof portfolios on $\Delta^{(n)}$ coincides exactly with the set of long-only portfolios on the simplex, denoted by
		\[
		\mathcal{LO}(\Delta^{(n)}) \coloneqq \left\{ \pi : \Delta^{(n)} \rightarrow \overline{\Delta^{(n)}} \right\}.
		\]
		That is,
		\begin{equation}
			\mathcal{BP}(\Delta^{(n)}) = \bigcap_{m=1}^{\infty} \mathcal{BP}(D_m)= \mathcal{LO}(\Delta^{(n)}).
		\end{equation}
	\end{prop}
	
	\begin{proof}
		We aim to show the equality between the limit set and the long-only set.
		
		(i) $\mathcal{BP}(\Delta^{(n)}) \subset \bigcap_{m=1}^{\infty} \mathcal{BP}(D_m)$: Proposition \ref{prop:Monotonicity} implies $\mathcal{BP}(\Delta^{(n)}) \subset \mathcal{BP}(D_m)$ for any $m\in \mathbb{N}$. Hence $\mathcal{BP}(\Delta^{(n)}) \subset \bigcap_{m=1}^{\infty} \mathcal{BP}(D_m)$.
		
		(ii) $\bigcap_{m=1}^{\infty} \mathcal{BP}(D_m) \subset \mathcal{LO}(\Delta^{(n)})$: Since $\pi\in \mathcal{BP}(D_m)$ means $\pi$ is defined on $D_m$, combined with $\cup_{m=1}^{\infty}D_m=\Delta^{(n)}$, $\pi\in\bigcap_{m=1}^{\infty} \mathcal{BP}(D_m)$ implies $\pi$ is defined on the entire simplex. Now we complete the proof by contradiction. Assume $\pi \in \bigcap_{m=1}^{\infty} \mathcal{BP}(D_m)$ but $\pi \notin \mathcal{LO}(\Delta^{(n)})$. Then there exist a market weight $p^* \in \Delta^{(n)}$ and an index $k$ such that $\pi_k(p^*) < 0$.
		
		Consider a sequence $\{q^{(l)}\}_{l=1}^{\infty} \subset \Delta^{(n)}$ converging to the vertex $e_k$. As $l \to \infty$, the ratio vector components satisfy $q^{(l)}_k/p^*_k \to 1/p^*_k$ and $q^{(l)}_j/p^*_j \to 0$ for $j \neq k$. Consequently,
		\[
		\lim_{l \to \infty} \left\langle \pi(p^*), \frac{q^{(l)}}{p^*} \right\rangle = \frac{\pi_k(p^*)}{p^*_k} < 0.
		\]
		Therefore, there exists an index $L$ such that $\langle \pi(p^*), q^{(L)}/p^* \rangle < 0$.
		
		Since $\{D_m\}$ covers $\Delta^{(n)}$, there exists an integer $M$ such that both $p^*$ and $q^{(L)}$ belong to $D_M$. This implies that $\pi \notin \mathcal{BP}(D_M)$, contradicting the assumption that $\pi \in \bigcap_{m=1}^{\infty} \mathcal{BP}(D_m)$.
		
		Thus, we must have $\pi_k(p) \ge 0$ for all $k$ and all $p\in \Delta^{(n)}$, proving $\pi \in \mathcal{LO}(\Delta^{(n)})$.
		
		(iii) $\mathcal{LO}(\Delta^{(n)}) \subset \mathcal{BP}(\Delta^{(n)})$: This follows immediately from \eqref{eq:NoBankruptcyCondition}, since $\pi(p)\in\overline{\Delta^{(n)}}$ and $q/p$ has strictly positive components for all $p,q\in \Delta^{(n)}$.
	\end{proof}
	\subsection{The Necessity of Dynamic Rebalancing: Failure of Static Weights}
	\label{subsec:CWP_MCM}
	
	Having established the structural properties of bankruptcy-proof portfolios, we now address the conditions required for arbitrage. Before analyzing complex strategies, it is instructive to examine the simplest class of portfolios: the \emph{constant-weighted portfolios} (CWPs). To rigorously analyze their potential for arbitrage in a market with short selling, we first recall a fundamental geometric condition formalized by Pal and Wong \cite{Wong2016TheGeometryofRelativeArbitrage}: \emph{multiplicative cyclical monotonicity} (MCM).
	
	\subsubsection{Multiplicative Cyclical Monotonicity as a Necessary Condition}
	
	In the context of relative arbitrage, a robust strategy should not systematically lose value over market cycles. This intuition is mathematically captured by the MCM property. We adapt the definition from \cite{Wong2016TheGeometryofRelativeArbitrage} to our admissible framework.
	
	\begin{definition}[Multiplicative cyclical monotonicity]
			\label{def:MCM}
			Let $K\subset\Delta^{(n)}$ be a nonempty set, and let $\pi:K\to H$ be a bankruptcy-proof portfolio. We say that $\pi$ satisfies \emph{multiplicative cyclical monotonicity} (MCM) on $K$ if for every cycle $\{\mu(s)\}_{s=0}^{m+1}\subset K$ with $\mu(m+1)=\mu(0)$,
			\begin{equation}
				\prod_{s=0}^{m}\left\langle \pi(\mu(s)),\frac{\mu(s+1)}{\mu(s)}\right\rangle \ge 1.
			\end{equation}
	\end{definition}
	
	The significance of MCM lies in its role as a prerequisite for pseudo-arbitrage. As shown in \cite{Wong2016TheGeometryofRelativeArbitrage}, violating MCM implies a systematic decay in wealth over repeated cycles.
	
	\begin{lemma}[Pseudo-arbitrage and MCM, \cite{Wong2016TheGeometryofRelativeArbitrage}]
		\label{lem:Pseudo-arbitrageMCM}
		Let $K$ be the market state space satisfying Assumption \ref{ass:Diversity}. If a bankruptcy-proof portfolio $\pi:K \rightarrow H$ fails to satisfy the MCM property on $K$, then there exists a market weight sequence taking values in a finite subset of $K$ such that $V_\pi(t) \to 0$ as $t \to \infty$. Consequently, $\pi$ can be a pseudo-arbitrage on $K$ only if it satisfies the MCM property on $K$.
	\end{lemma}
	
	\begin{proof}[Sketch of Proof]
		The argument follows \cite[Lemma~3]{Wong2016TheGeometryofRelativeArbitrage}. If MCM fails, there exists a cycle over which the one-cycle relative return $\eta<1$. Constructing a market sequence that periodically repeats this cycle yields $V_\pi(k(m+1))=\eta^k\to 0$ as $k\to\infty$. This violates the bounded downside condition required for pseudo-arbitrage.
	\end{proof}
	
	\subsubsection{Failure of CWPs under Short Selling}
	Before presenting the geometric formalism, it is instructive to explain why constant-weighted portfolios (CWPs) with short positions are unlikely to provide a robust mechanism for arbitrage-like growth. In continuous-time models of equity markets, the logarithmic growth rate of a self-financing portfolio admits a decomposition into a weighted average of the individual stock growth rates and an additional correction term that depends on the covariation structure of returns; see, for example, Fernholz, Karatzas and Kardaras \cite{Fernholz2005Diversity}. In the long-only case, this correction term is nonnegative and represents a volatility-capture effect. Once negative weights are allowed, however, the same correction term need not remain nonnegative and can instead act as a systematic volatility drag on long-run compounded performance.
	
	A particularly transparent manifestation of this effect appears in constant-proportion strategies with negative exposure. Avellaneda and Zhang \cite{AvellanedaZhang2010LETF} derive an explicit formula for leveraged and inverse funds, in which the value process contains a multiplicative factor of the form $\exp\{(\beta-\beta^2)\int_0^t \sigma_s^2\,ds/2\}$. For $\beta\notin[0,1]$, and in particular for $\beta<0$, the coefficient $(\beta-\beta^2)$ is strictly negative, resulting in an accumulating penalty proportional to realized variance. This illustrates that maintaining a static short (negative-weight) component typically requires sufficiently large favorable cumulative returns of the underlying (e.g., a large directional move) to offset this variance-related drag. Consequently, this variance-drag mechanism cautions against expecting CWPs with short positions to deliver pseudo-arbitrage in a structurally robust way, which motivates the need for dynamically rebalanced admissible strategies.
	
	Equipped with this intuition and the MCM criterion, we now formally examine these strategies. Recall that in the functional generation framework, a CWP is constructed \cite{Fernholz2002SPT} by $\Phi(p) = \prod_{i=1}^n p_i^{\pi_i}$, leading to constant weights $\pi(p) \equiv \pi$.
	
	While long-only CWPs are known to satisfy MCM, the introduction of short positions fundamentally alters the geometry. In particular, along directions associated with shorted assets, the generating potential may become locally convex, which is consistent with the above ``volatility drag'' intuition in the sense that fluctuations can reduce, rather than increase, the relative value. We formalize this by showing that CWPs with short positions violate the MCM condition over simple cycles.
	
	\begin{prop}[Failure of MCM for Short-Selling CWPs]
		\label{prop:MCM_Failure}
		Let $K$ be the market state space satisfying Assumption \ref{ass:Diversity}, and let $\pi\in \mathcal{BP}(K)$ be a constant-weighted portfolio (i.e., $\pi(p)\equiv \pi$ for all $p\in K$). Consider a cycle $p \to q \to p$ in $K$, where
		$q = p + t(e_1 - p)$ for some $t \in (0,1)$.
		If the portfolio holds a short position in asset $1$ (i.e., $\pi_1 < 0$), then the relative value over this cycle satisfies $V_\pi < 1$. Consequently, $\pi$ violates MCM.
	\end{prop}
	
	\begin{proof}
		The relative value over the cycle $p \to q \to p$ is
		\[
		V_\pi = \left\langle \pi, \frac{q}{p} \right\rangle \left\langle \pi, \frac{p}{q} \right\rangle.
		\]
		A direct computation using $q=p+t(e_1-p)$ yields
		\begin{equation}
			V_\pi
			= 1 + \frac{t^2 \pi_1 (1 - \pi_1)}{p_1(1-t)\,[p_1(1-t) + t]}.
		\end{equation}
		Since $\pi_1 < 0$, we have $\pi_1(1-\pi_1) < 0$. The denominator is strictly positive for $p \in K$. Therefore, the second term is strictly negative, implying $V_\pi < 1$.
		This violates the MCM condition in Definition \ref{def:MCM}.
	\end{proof}
	
	\begin{coro}[Incompatibility of CWP and Arbitrage]
		\label{coro:CWP_Arbitrage}
		Let $K$ be the market state space satisfying Assumption \ref{ass:Diversity} and let $\pi$ be a constant-weighted bankruptcy-proof portfolio. If $\pi$ involves a short position in any asset, then $\pi$ cannot be a pseudo-arbitrage on $K$.
	\end{coro}
	
	\paragraph{Transition to Functional Generation.}
	Corollary \ref{coro:CWP_Arbitrage} reveals a fundamental limitation: static short positions are geometrically incompatible with cyclical market stability. Unlike long-only strategies where constant weights can capture growth, admissible strategies that involve short selling require dynamic rebalancing in order to satisfy the MCM condition. To systematically construct such dynamic strategies, we now turn to the framework of \emph{functionally generated portfolios}, which provides the geometric machinery to engineer bankruptcy-proof portfolios with short-selling positions.
	
	\section{Geometric Characterization of Admissible Arbitrage}
	\label{sec:GeometricCharacterization}
	This section generalizes the functional generation framework \cite{Wong2016TheGeometryofRelativeArbitrage} to admissible long-short portfolios. While retaining the classical algebraic structure (gradient maps and divergence decomposition), we relax the domain of generation from the global simplex $\Delta^{(n)}$ to its open convex subset $D$. This shift allows us to characterize a richer class of strategies that exploit short positions while maintaining solvency.
	
	\subsection{Functionally Generated Portfolios and Pseudo-Arbitrage}
	\label{subsec:FGP_PseudoArbitrage}
	We begin by defining the generation mechanism on a general domain.
	
	\begin{definition}[Functionally Generated Portfolio]
		\label{def:FGP_Admissible}
		Let $D \subset \Delta^{(n)}$ be a nonempty open convex set. A portfolio map $\pi: D \to H$ is said to be \emph{generated} by a positive concave function $\Phi: D \rightarrow (0, \infty)$ if it satisfies the supergradient inequality:
		\begin{equation}
			\left\langle \pi(p), \frac{q}{p} \right\rangle = 1 + \left\langle \frac{\pi(p)}{p}, q-p \right\rangle \ge \frac{\Phi(q)}{\Phi(p)}, \quad \forall p, q \in D.
			\label{eq:MCMIneq}
		\end{equation}
		The function $\Phi$ is termed the \emph{generating function} of $\pi$.
	\end{definition}
	
	Standard convex analysis guarantees that the superdifferential of a concave function is nonempty on the relative interior of its domain \cite[Theorem 23.4]{Rockafellar1970ConvexAnalysis}. In our context, we typically assume $K \subset D$ so that the strategy is well defined on the market state space.
	
	\begin{remark}
		Functionally generated portfolios are inherently \emph{bankruptcy-proof} on their domain. Since $\Phi$ is strictly positive, the inequality \eqref{eq:MCMIneq} implies $\langle \pi(p), q/p \rangle \ge \Phi(q)/\Phi(p) > 0$ for all $p, q \in D$.
	\end{remark}
	
	To understand the geometric mechanics of these strategies, particularly how they generate short positions, we restrict our attention to differentiable generating functions. In this setting, the portfolio weights admit an explicit representation in terms of the logarithmic gradient.
	
	\begin{prop}[Gradient representation]
			\label{prop:FGExpression}
			Let $D$ be a nonempty open convex subset of $\Delta^{(n)}$, and let $\Phi:D\to(0,\infty)$ be a differentiable concave function. If $\pi:D\to H$ is generated by $\Phi$ in the sense of Definition~\ref{def:FGP_Admissible}, then for every $p\in D$ and each $i=1,\ldots,n$,
			\begin{equation}
				\pi_i(p)=p_i\left(1+\left\langle\nabla\varphi(p),\,e_i-p\right\rangle\right),
				\label{eq:FGDifferential_Local}
			\end{equation}
			where $\varphi:=\log\Phi$, and $\nabla\varphi(p)\in T$ denotes the gradient of $\varphi$ relative to the affine space $H$.
	\end{prop}
	
	\begin{proof}
			Fix $p\in D$ and $i\in\{1,\ldots,n\}$. Since $D$ is open relative to $H$, for all $h$ sufficiently close to $0$ we have $p+h(e_i-p)\in D$. Applying \eqref{eq:MCMIneq} with $q=p+h(e_i-p)$ gives
			\[
			1+h\left(\frac{\pi_i(p)}{p_i}-1\right)
			=1+\left\langle \frac{\pi(p)}{p},\,h(e_i-p)\right\rangle
			\ge \frac{\Phi(p+h(e_i-p))}{\Phi(p)},
			\]
			where we used $\sum_{j=1}^n \pi_j(p)=1$. Taking logarithms, dividing by $h$, and letting $h\to 0^\pm$, we obtain
			\[
			\frac{\pi_i(p)}{p_i}-1=D_{e_i-p}\log\Phi(p).
			\]
			Since $\varphi=\log\Phi$ is differentiable on $H$,
			\[
			D_{e_i-p}\varphi(p)=\langle \nabla\varphi(p),e_i-p\rangle.
			\]
			Hence
			\[
			\pi_i(p)=p_i\left(1+\langle \nabla\varphi(p),e_i-p\rangle\right).
			\]
	\end{proof}
	
	Equation \eqref{eq:FGDifferential_Local} is not merely an algebraic formula; it admits a precise geometric interpretation in terms of the supporting hyperplane of $\Phi$ \cite{Wong2019InfromationGeometryinPortfolio}. Let $L_p(x)$ denote the tangent plane of $\Phi$ at the market state $p$. The geometric intercept of this plane at the vertex $e_i$ is given by:
	\begin{equation}
		c_i \coloneqq L_p(e_i) = \Phi(p) + \langle \nabla \Phi(p), e_i - p \rangle.
	\end{equation}
	Using the relation $\nabla \varphi = \nabla \Phi / \Phi$, equation \eqref{eq:FGDifferential_Local} can be rewritten as $\pi_i(p) = \frac{p_i c_i}{\Phi(p)}$.
	
	This formulation explicitly reveals the mechanism of short selling: the sign of the portfolio weight $\pi_i(p)$ is determined solely by the sign of the geometric intercept $c_i$. If the tangent plane becomes sufficiently steep, typically due to a gradient blow-up near a boundary where $\Phi$ approaches zero, the intercept at a distant vertex may become negative ($c_i < 0$), thereby structurally enforcing a short position. This geometric intuition is visualized in the simplest case ($n=2$) in Figure \ref{fig:TangentLineGeometry}.
	
	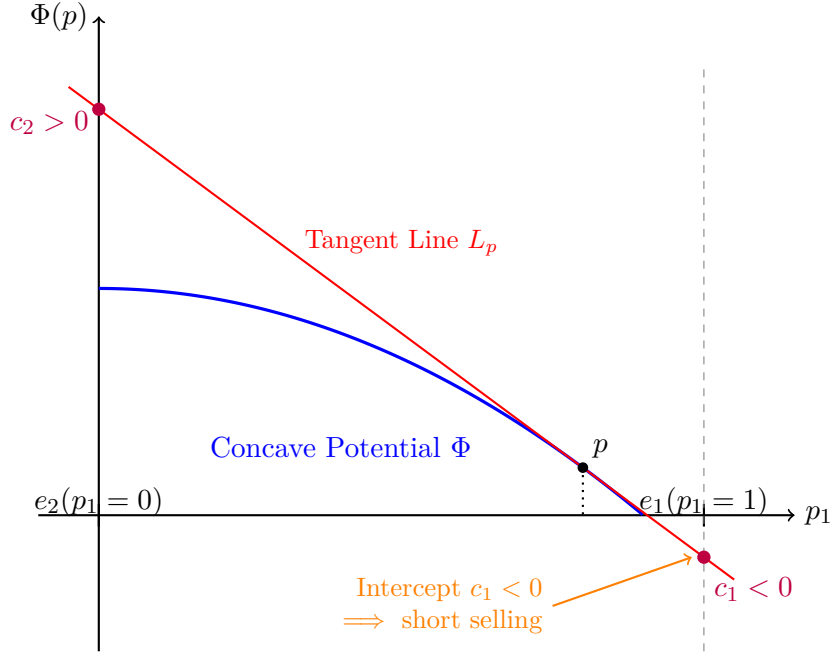
\begin{figure}[htbp]
		\centering
		\begin{tikzpicture}[x=8cm, y=3cm]
			
			\def\xp{0.8}
			
			\pgfmathsetmacro{\yp}{1 - (\xp/0.9)^2}
			\pgfmathsetmacro{\slope}{-2 * \xp / (0.9*0.9)}
			
			\pgfmathsetmacro{\cOne}{\slope*(1-\xp) + \yp}
			\pgfmathsetmacro{\cTwo}{\slope*(0-\xp) + \yp}
			
			
			\draw[->, thick] (-0.1, 0) -- (1.15, 0) node[right] {$p_1$};
			\draw[->, thick] (0, -0.6) -- (0, 2.2) node[left] {$\Phi(p)$};
			
			\draw[thick] (0, 0.05) -- (0, -0.05) node[above] {$e_2 (p_1=0)$};
			\draw[thick] (1, 0.05) -- (1, -0.05) node[above] {$e_1 (p_1=1)$};
			\draw[dashed, gray] (1, -0.6) -- (1, 2.0);
			
			\draw[blue, very thick, domain=0:0.9, samples=100] plot (\x, {1 - (\x/0.9)^2});
			\node[blue] at (0.4, 0.3) {Concave Potential $\Phi$};
			
			\draw[red, thick] (-0.05, {\slope*(-0.05-\xp) + \yp}) -- (1.05, {\slope*(1.05-\xp) + \yp});
			\node[red, font=\small] at (0.5, 1.2) {Tangent Line $L_p$};
			
			\draw[dotted, thick] (\xp, \yp) -- (\xp, 0);
			\fill[black] (\xp, \yp) circle (2pt) node[above right] {$p$};
			
			
			\fill[purple] (0, \cTwo) circle (2.5pt);
			\node[purple, left, yshift=-5pt] at (0, \cTwo) {$c_2 > 0$};
			
			\fill[purple] (1, \cOne) circle (2.5pt);
			\node[purple, right, yshift=-12pt] at (1, \cOne) {$c_1 < 0$};
			
			\draw[->, thick, orange] (0.75, -0.4) -- (0.98, \cOne);
			\node[orange, left, align=right, font=\small] at (0.75, -0.4) {Intercept $c_1 < 0$\\$\implies$ short selling};
			
		\end{tikzpicture}
		\caption{Geometric intuition of portfolio weights in the 2D case ($n=2$). The potential $\Phi$ (blue curve) vanishes at $p_1=0.9$. The tangent line (red) at point $p=0.8$ yields a positive intercept $c_2$ at $e_2$ but a strictly negative intercept $c_1$ at $e_1$, structurally enforcing a short position in asset 1.}
		\label{fig:TangentLineGeometry}
	\end{figure}
	
	It is important to note that the differentiability of $\Phi$ is not a strict requirement for strategy generation. In the general case where $\Phi$ is merely concave but not necessarily smooth, the gradient in \eqref{eq:FGDifferential_Local} can be replaced by a \emph{supergradient} vector $g \in \partial \varphi(p)$. While the specific choice of supergradient at points of non-differentiability may lead to non-unique portfolio weights locally, the strategy remains well-defined in an almost-everywhere sense. Since a concave function on an open set is differentiable almost everywhere, the generating function $\Phi$ determines a unique portfolio strategy $\pi(p)$ for almost all $p \in D$. This is a direct consequence of the classical result of Pal and Wong \cite[Proposition 5]{Wong2016TheGeometryofRelativeArbitrage}.
	
	Crucially, because we do not require $\Phi$ to extend to a non-negative concave function on the entire simplex in Definition \ref{def:FGP_Admissible}, the gradient term in \eqref{eq:FGDifferential_Local} can be negative enough to yield $\pi_i < 0$, thus enabling short positions.
	
	The equivalence between the MCM property and functional generation, proved in \cite{Wong2016TheGeometryofRelativeArbitrage} on the whole simplex, carries over to the present local setting.
	
	\begin{prop}[Equivalence of MCM and functional generation]
			\label{prop:MCMConcaveFunction}
			Let $K\subset\Delta^{(n)}$ be the nonempty open convex market state space in Assumption~\ref{ass:Diversity}, and let $\pi:K\to H$ be a bankruptcy-proof portfolio. Then $\pi$ satisfies MCM on $K$ if and only if there exists a positive concave function $\Phi:K\to(0,\infty)$ such that $\pi$ is generated by $\Phi$ in the sense of Definition~\ref{def:FGP_Admissible}.
	\end{prop}
	
	\begin{proof}
			The implication ``$\Leftarrow$'' follows immediately by multiplying \eqref{eq:MCMIneq} along any cycle in $K$.
			
			For ``$\Rightarrow$'', assume that $\pi$ satisfies MCM on $K$. Fix $p_0\in K$ and define
			\begin{equation}
				\Phi(p):=\inf_{\substack{t\ge0,\ \{\mu(s)\}_{s=0}^{t+1}\subset K,\ \mu(0)=p_0,\ \mu(t+1)=p}}  \prod_{s=0}^{t}\left(1+\left\langle \frac{\pi(\mu(s))}{\mu(s)},\,\mu(s+1)-\mu(s)\right\rangle\right).
				\label{eq:Phi_Construction}
			\end{equation}
			By bankruptcy-proofness, $\Phi\ge0$. As in the proof of \cite[Proposition 4]{Wong2016TheGeometryofRelativeArbitrage}, $\Phi$ is concave as the pointwise infimum of affine functions of $p$. Moreover, MCM implies $\Phi(p_0)=1$, hence $\Phi$ is strictly positive on $K$. Finally, extending any admissible path from $p$ to $q$ by one step yields
			\[
			\Phi(q)\le \Phi(p)\left(1+\left\langle \frac{\pi(p)}{p},\,q-p\right\rangle\right)
			=\Phi(p)\left\langle \pi(p),\frac{q}{p}\right\rangle,
			\]
			which is exactly \eqref{eq:MCMIneq}. Therefore, $\pi$ is generated by the positive concave function $\Phi$ on $K$.
	\end{proof}
	
	We now adapt the fundamental geometric tools of SPT, the $L$-divergence and wealth decomposition, to our restricted domain setting. These definitions parallel those established in \cite{Wong2016TheGeometryofRelativeArbitrage}, with the crucial distinction that they are valid locally on $D$ rather than requiring the global structure of the entire simplex.
	
	\begin{definition}[$L$-divergence, {\cite{Wong2016TheGeometryofRelativeArbitrage}}]
			\label{def:L_divergence}
			Let $D$ be a nonempty open convex subset of $\Delta^{(n)}$, and let $\pi$ be a portfolio generated by a positive concave function $\Phi:D\to(0,\infty)$. Write $\varphi:=\log\Phi$. The \emph{$L$-divergence} of the pair $(\Phi,\pi)$ is defined by
			\begin{equation}
				T_\varphi(q\mid p)
				:=\log\left(1+\left\langle \frac{\pi(p)}{p},\,q-p\right\rangle\right)-\bigl(\varphi(q)-\varphi(p)\bigr),
				\qquad p,q\in D.
				\label{eq:LDivergence}
			\end{equation}
	\end{definition}
	
	We use the notation $T_\varphi$ to emphasize the dependence on the logarithmic generating function $\varphi=\log\Phi$, while keeping in mind that, in general, the divergence is attached to the pair $(\Phi,\pi)$. The quantity $T_\varphi$ is a logarithmic analogue of the classical Bregman divergence \cite{Amari2010InformationGeometry}. Unlike the global formulation in \cite{Wong2016TheGeometryofRelativeArbitrage}, our definition is purely local: it is defined only for $(p,q)\in D\times D$, and no extension beyond $D$ is assumed.
	From the definition, it follows immediately that $T_\varphi(q \mid p) \geq 0$ due to the concavity of $\Phi$, with equality if and only if $\Phi$ is affine on the segment connecting $p$ and $q$.
	
	Based on this local divergence, we establish that the canonical pathwise decomposition of relative wealth remains valid for trajectories confined within the domain.
	
	\begin{lemma}[Value decomposition]
			\label{lem:ValueDecomposition}
			Let $D$ be a nonempty open convex subset of $\Delta^{(n)}$, and let $\pi$ be a portfolio generated by a positive concave function $\Phi:D\to(0,\infty)$. Write $\varphi:=\log\Phi$. Then, for any market weight sequence $\{\mu(k)\}_{k=0}^{t}\subset D$, the relative wealth process satisfies
			\begin{equation}
				\log V_\pi(t)=\log\frac{\Phi(\mu(t))}{\Phi(\mu(0))}+A(t),
			\end{equation}
			where
			\[
			A(t):=\sum_{k=0}^{t-1}T_\varphi(\mu(k+1)\mid\mu(k))
			\]
			is non-decreasing. Moreover, $A(t)\equiv0$ for every market weight sequence in $D$ if and only if $\Phi$ is affine on $D$.
	\end{lemma}

	\begin{proof}
			By the self-financing identity and Definition~\ref{def:L_divergence},
			\[
			\log\frac{V_\pi(k+1)}{V_\pi(k)}
			=\log\frac{\Phi(\mu(k+1))}{\Phi(\mu(k))}+T_\varphi(\mu(k+1)\mid\mu(k)).
			\]
			Summing over $k$ gives the decomposition. The monotonicity of $A(t)$ follows from $T_\varphi\ge0$, and the final assertion follows from the fact that $T_\varphi\equiv0$ on $D\times D$ if and only if $\Phi$ is affine on $D$.
	\end{proof}
	
	With these geometric tools in place, we now provide the proofs for our main result stated in the Introduction.
	
	\subsubsection*{Proof of Characterization of Admissible Pseudo-Arbitrage}

	\begin{proof}[Proof of Theorem~\ref{thm:AdmissibleArbitrage}]
			(\emph{Sufficiency}) Suppose that such a function $\Phi:D\to[0,\infty)$ exists. Since $\mu(t)\in K$ for all $t\ge0$ and $\pi|_K$ is generated by $\Phi|_K$, Lemma~\ref{lem:ValueDecomposition} yields
			\[
			\log V_\pi(t)=\log \Phi(\mu(t))-\log \Phi(\mu(0))+A(t).
			\]
			By condition (i), $\Phi|_K$ is bounded away from zero on $K$. Since $\Phi|_K$ is concave on the bounded convex set $K$, it is bounded above on $K$. Hence the potential term is uniformly bounded below, and therefore $\log V_\pi(t)$ is bounded below uniformly over all market weight sequences in $K$. This proves condition (i) in the definition of pseudo-arbitrage.
			
			To verify condition (ii), note that $\Phi|_K$ is not affine by assumption. Hence there exist $p,q\in K$ such that $\Phi|_K$ is not affine on the line segment $[p,q]\subset K$, and therefore
			\[
			T_\varphi(q\mid p)+T_\varphi(p\mid q)>0.
			\]
			Repeating the cycle $p\to q\to p\to q\to\cdots$ yields $A(t)\to\infty$. Since the potential term remains uniformly bounded, it follows that $V_\pi(t)\to\infty$ along this sequence. Thus $\pi$ is a pseudo-arbitrage on $K$.
			
			(\emph{Necessity}) Conversely, suppose that $\pi:D\to H$ is a pseudo-arbitrage on $K$. By Lemma~\ref{lem:Pseudo-arbitrageMCM}, the restriction $\pi|_K$ satisfies MCM on $K$. Hence, by Proposition~\ref{prop:MCMConcaveFunction}, there exists a positive concave function $\Phi_K:K\to(0,\infty)$ such that $\pi|_K$ is generated by $\Phi_K$.
			
			We now extend $\Phi_K$ to a function $\Phi:D\to[0,\infty)$. Fix $p_0\in K$ and define
		
			\[
			\Phi(p):=
			\begin{cases}
				\Phi_K(p), & p\in K,\\[1ex]
				\displaystyle
				\inf_{\substack{
						t\ge0,\ \{\mu(s)\}_{s=0}^{t+1}\subset D\\
						\mu(0)=p_0,\ \mu(t+1)=p
				}}
				\prod_{s=0}^{t}\left\langle \pi(\mu(s)),\frac{\mu(s+1)}{\mu(s)}\right\rangle,
				& p\in D\setminus K.
			\end{cases}
			\]
			
			Since $\pi$ is bankruptcy-proof on $D$, every factor in the above product is strictly positive, so $\Phi(p)\ge0$ for all $p\in D$. Moreover, $\Phi|_K=\Phi_K$, and therefore $\Phi|_K$ is positive and concave on $K$, and $\pi|_K$ is generated by $\Phi|_K$.
			
			It remains to verify conditions (i) and (ii). If $\Phi|_K$ were affine, then Lemma~\ref{lem:ValueDecomposition} would give $A(t)\equiv0$ for every market weight sequence in $K$, and since $\Phi|_K$ is bounded above and bounded away from zero on $K$, the relative wealth process would remain uniformly bounded above along every such sequence, contradicting pseudo-arbitrage. Thus $\Phi|_K$ is not affine on $K$.
			
			Finally, suppose that $\Phi|_K$ is not bounded away from zero on $K$. Then $0$ is a limit point of $\Phi_K(K)$. As in the proof of \cite[Theorem~1]{Wong2016TheGeometryofRelativeArbitrage}, one may choose a sequence along a line segment in $K$ approaching a boundary point $q\in\overline{K}$ with $\Phi_K(q)=0$, and use the gradient representation together with the inequality $\log(1+x)\le x$ for $x>-1$ to construct a market weight sequence in $K$ along which $V_\pi(t)\to0$. In the present setting, the required positivity of the one-step wealth factors is guaranteed by bankruptcy-proofness. This contradicts the pseudo-arbitrage property. Therefore $\Phi|_K$ is bounded away from zero on $K$.
	\end{proof}

	\section{The Geometry of the Short-Selling Mechanism}
	\label{sec:GeometryShortSelling}
	While Theorem \ref{thm:LongOnlyArbitrage} characterizes globally long-only strategies via the existence of a non-negative concave extension to the entire simplex, this global condition can be overly restrictive for portfolios defined only on a local domain $D$. In particular, a strategy may be long-only on $D$ even though no non-negative concave extension exists on $\Delta^{(n)}$.
	
	This section establishes geometric criteria for short selling in this admissible setting. We first solve the identification problem by introducing the \emph{maximal concave extension}, proving that its negativity is a necessary and sufficient condition for the presence of short positions (Theorem \ref{thm:ShortSellingEquivalence}). Building on this, we analyze generating functions satisfying a \emph{zero-boundary condition}. We demonstrate that such potentials can lead to a \emph{geometric phase transition} (Theorem \ref{thm:GeometricPhaseTransition}), in the sense that one can identify a long-only core and, under an additional geometric condition, a short-selling shell driven by the steepness of the logarithmic potential near the boundary.
	
	\subsection{The Maximal Concave Extension and the Short-Selling Condition}
	
	To discriminate between long-only and short-selling strategies, we rely on the property that a proper upper semicontinuous concave function is the pointwise infimum of its supporting hyperplanes. For a portfolio $\pi$ generated by a differentiable concave function $\Phi$ on $K$, the supporting hyperplane at $p\in K$ can be written as
	$L_p(x) \coloneqq \Phi(p)\left\langle \pi(p), x/p \right\rangle$, where $x/p$ denotes the vector of coordinatewise ratios. We define the maximal extension by intersecting these hyperplanes.
	
	\begin{definition}[Maximal Concave Extension]
		\label{def:MaximalExtension}
		Let $K$ be a market state space satisfying Assumption \ref{ass:Diversity}. Let $\Phi: K \to (0, \infty)$ be a differentiable concave function generating a portfolio $\pi:K \rightarrow H$ via equation \eqref{eq:FGDifferential_Local}. The \emph{maximal concave extension} of $\Phi$ to the simplex $\Delta^{(n)}$ is defined by
		\begin{equation}
			\hat{\Phi}(x) \coloneqq \inf_{p \in K} L_p(x)
			= \inf_{p \in K} \left\{ \Phi(p)\left\langle \pi(p), \frac{x}{p} \right\rangle \right\},
			\quad x \in \Delta^{(n)}.
			\label{eq:MaximalExtension}
		\end{equation}
	\end{definition}
	
	\begin{remark}
		The terminology ``maximal'' is justified as follows. The function $\hat{\Phi}$ is concave on $\Delta^{(n)}$, and one has $\hat{\Phi}(x)=\Phi(x)$ for all $x\in K$. Moreover, if $\Psi:\Delta^{(n)}\to\mathbb{R}$ is any concave extension of $\Phi$ (that is, $\Psi|_K=\Phi$), then $\Psi(x)\le L_p(x)$ for all $p\in K$ and all $x\in \Delta^{(n)}$, hence $\Psi(x)\le \hat{\Phi}(x)$ for all $x$. In particular, if $\hat{\Phi}$ takes a negative value, then no non-negative concave extension of $\Phi$ to $\Delta^{(n)}$ can exist.
	\end{remark}
	
	We now establish the main result of this subsection: the behavior of $\hat{\Phi}$ provides a necessary and sufficient condition for the existence of short positions within the market state space.
	
	\subsubsection*{Proof of the Equivalence of Short Selling and Negative Extension}
	
	\begin{proof}[Proof of Theorem \ref{thm:ShortSellingEquivalence}]
			First suppose that $\hat{\Phi}(x)<0$ for some $x\in\Delta^{(n)}$. Since
			\[
			\hat{\Phi}(x)=\inf_{p\in K}L_p(x),
			\]
			there exists $p\in K$ such that $L_p(x)<0$, that is,
			\[
			\Phi(p)\sum_{j=1}^n \pi_j(p)\frac{x_j}{p_j}<0.
			\]
			Because $\Phi(p)>0$ and $x_j/p_j>0$ for all $j$, at least one coefficient $\pi_j(p)$ must be negative. Hence $\pi$ involves short selling on $K$.
			
			Conversely, suppose that $\pi_i(p)<0$ for some $p\in K$ and some index $i$. Choose any sequence $\varepsilon_m\downarrow0$ and define
			\[
			x^{(m)}:=(1-\varepsilon_m)e_i+\varepsilon_m p\in\Delta^{(n)}.
			\]
			Then $x^{(m)}\to e_i$ and
			\[
			L_p(x^{(m)})=\Phi(p)\left\langle \pi(p),\frac{x^{(m)}}{p}\right\rangle
			\longrightarrow \Phi(p)\frac{\pi_i(p)}{p_i}<0.
			\]
			Therefore $L_p(x^{(m)})<0$ for all sufficiently large $m$. By definition of $\hat{\Phi}$,
			\[
			\hat{\Phi}(x^{(m)})\le L_p(x^{(m)})<0
			\]
			for all such $m$. Hence $\hat{\Phi}$ is negative at some point of $\Delta^{(n)}$.
	\end{proof}
	
	\subsection{Geometric Construction via the Zero-Boundary Condition}
	\label{subsec:ZeroBoundary}
	Theorem \ref{thm:ShortSellingEquivalence} establishes that short selling arises precisely when the generating function fails to admit a non-negative concave extension to the entire simplex. To systematically construct such strategies, we seek functions that are guaranteed to fail this extension property.
	
	We next introduce a simple geometric criterion for non-extendability. If a positive concave function vanishes at a boundary point of its domain that still lies in the interior of the simplex, then it cannot admit a non-negative concave extension to $\Delta^{(n)}$.
	\begin{prop}[Zero-Boundary Implies Non-Extendability]
			\label{prop:ZeroBoundary}
			Let $D\subset\Delta^{(n)}$ be a nonempty open convex set, and let $\Phi:\overline{D}\to[0,\infty)$ be a concave function that is strictly positive on $D$. If there exists $q\in\partial D\cap\Delta^{(n)}$ such that $\Phi(q)=0$, then $\Phi$ admits no non-negative concave extension to $\Delta^{(n)}$.
	\end{prop}
	
	\begin{proof}
			Choose any $p\in D$. Since $q\in\partial D\cap\Delta^{(n)}$ and $\Delta^{(n)}$ is relatively open in $H$, there exist $r\in\Delta^{(n)}\setminus\overline{D}$ and $\lambda\in(0,1)$ such that $q=(1-\lambda)p+\lambda r$.
			Suppose that $\Phi$ admits a non-negative concave extension $\Psi$ to $\Delta^{(n)}$. By concavity,
			\[
			\Psi(q)\ge (1-\lambda)\Psi(p)+\lambda\Psi(r).
			\]
			Using $\Psi(q)=\Phi(q)=0$ and $\Psi(p)=\Phi(p)>0$, we obtain
			\[
			\Psi(r)\le -\frac{1-\lambda}{\lambda}\Psi(p)<0,
			\]
			which contradicts the non-negativity of $\Psi$ on $\Delta^{(n)}$.
	\end{proof}
	
	Proposition~\ref{prop:ZeroBoundary} provides a concrete geometric criterion for failure of non-negative concave extendability. In the subsequent analysis, we use this observation to construct explicit examples and identify parameter regimes under which short selling occurs.
	
	\subsection{Analytical Mechanism: Blow-Up Near a Zero Boundary}
	\label{subsec:AnalyticalMechanism}
	The zero-boundary condition ensures that short positions must occur somewhere on $D$, but it is also useful to understand how they arise as the market approaches the boundary. The key point is that the logarithmic potential $\log\Phi$ can become arbitrarily steep near a boundary point where $\Phi$ vanishes. We formalize this mechanism using tools from convex analysis.
	
	\begin{prop}[Blow-Up Near a Zero Boundary]
			\label{prop:LogGradientDivergence}
			Let $D\subset\Delta^{(n)}$ be a nonempty open convex set, and let $\Phi:\overline{D}\to[0,\infty)$ be a continuous concave function. Suppose that $\Phi(p_0)>0$ for some $p_0\in D$, and that $\Phi(q)=0$ for some $q\in\partial D\cap\Delta^{(n)}$ lying strictly on the line segment joining $p_0$ and a vertex $e_i$. Then there exists a selection $g(p)\in\partial\Phi(p)$ for all $p$ sufficiently close to $q$ along $[p_0,q)$ such that
			\begin{equation}
				\lim_{p\to q,\ p\in[p_0,q)}\frac{\langle g(p),e_i-p\rangle}{\Phi(p)}=-\infty.
			\end{equation}
	\end{prop}
	
	\begin{proof}
		By continuity at $q$ and the fact that $\Phi(q)=0<\Phi(p_0)$, for $\varepsilon>0$ small enough the point
		\[
		q_\varepsilon:=p_0+(1-\varepsilon)(q-p_0)\in(p_0,q)
		\]
		satisfies $\Phi(q_\varepsilon)<\Phi(p_0)$. Let $u:=\frac{q-p_0}{\|q-p_0\|}$. Applying the mean value theorem for convex functions to $-\Phi$ along the segment from $p_0$ to $q_\varepsilon$ \cite{Wegge1974MVT}, there exist $\lambda_\varepsilon\in(0,1)$ and $g_\varepsilon\in\partial\Phi(\xi_\varepsilon)$, where
		$\xi_\varepsilon:=p_0+\lambda_\varepsilon(q_\varepsilon-p_0)$,
		such that
		\[
		\langle g_\varepsilon,u\rangle
		=\frac{\Phi(q_\varepsilon)-\Phi(p_0)}{\|q_\varepsilon-p_0\|}
		\eqqcolon -C<0.
		\]
		
		For $p\in D$ and $\nu\in T$, let
		\[
		\Phi'(p;\nu):=\lim_{h\downarrow0}\frac{\Phi(p+h\nu)-\Phi(p)}{h}
		\]
		denote the one-sided directional derivative. Since $\Phi$ is concave on the open convex set $D$,
		\[
		\Phi'(p;u)=\min_{g\in\partial\Phi(p)}\langle g,u\rangle,
		\qquad p\in D,
		\]
		with the minimum attained \cite[Theorem 23.4]{Rockafellar1970ConvexAnalysis}. In particular,
		\[
		\Phi'(\xi_\varepsilon;u)\le \langle g_\varepsilon,u\rangle=-C.
		\]
		
		Now define
		\[
		f(\lambda):=\Phi\bigl(p_0+\lambda(q-p_0)\bigr),\quad \lambda\in[0,1].
		\]
		Then $f$ is concave, so its right derivative $f'_+$ is nonincreasing \cite[Theorem 24.1]{Rockafellar1970ConvexAnalysis}. Since $\xi_\varepsilon=p_0+\lambda_\varepsilon(1-\varepsilon)(q-p_0)$, for $\lambda\in(\lambda_\varepsilon(1-\varepsilon),1)$ and $p=p_0+\lambda(q-p_0)$ we have
		\[
		f'_+(\lambda)=\Phi'(p;q-p_0)=\|q-p_0\|\,\Phi'(p;u),
		\]
		and hence
		\[
		\Phi'(p;u)\le \Phi'(\xi_\varepsilon;u)\le -C,
		\qquad p\in(\xi_\varepsilon,q).
		\]
		For each such $p$, choose $g(p)\in\partial\Phi(p)$ so that
		\[
		\langle g(p),u\rangle=\Phi'(p;u)\le -C.
		\]
		Extend $g$ arbitrarily on $[p_0,\xi_\varepsilon]$.
		
		Since $q$ lies strictly on the segment between $p_0$ and $e_i$, there exists $k>0$ such that
		$e_i-q=k\|q-p_0\|\,u$. Also, for $p\in(\xi_\varepsilon,q)$, $q-p=\|q-p\|\,u$.
		Therefore,
		\[
		\langle g(p),e_i-p\rangle
		=\langle g(p),e_i-q\rangle+\langle g(p),q-p\rangle
		\le -kC\|q-p_0\|-C\|q-p\|
		\le -M,
		\]
		where $M:=kC\|q-p_0\|>0$. Since $\Phi(p)\to0$ as $p\to q$ along $[p_0,q)$,
		\[
		\frac{\langle g(p),e_i-p\rangle}{\Phi(p)}
		\le \frac{-M}{\Phi(p)}
		\longrightarrow -\infty.
		\]
	\end{proof}
	
	\begin{coro}[Emergence of Short Positions along a Ray to the Boundary]
			\label{coro:ShortSellingNearBoundary}
			Let $D\subset\Delta^{(n)}$ be open and convex, and let $\Phi:\overline{D}\to[0,\infty)$ be continuous on $\overline{D}$ and differentiable, concave, and strictly positive on $D$. Suppose that $\Phi(q)=0$ for some $q\in\partial D\cap\Delta^{(n)}$ lying strictly on the line segment joining a point $p_0\in D$ and a vertex $e_i$. Let $\pi$ be the portfolio generated by $\Phi$ via \eqref{eq:FGDifferential_Local}. Then $\pi_i(p)<0$ for all $p$ sufficiently close to $q$ along $[p_0,q)$.
	\end{coro}
	
	\begin{proof}
			By Proposition~\ref{prop:LogGradientDivergence} and differentiability of $\Phi$ on $D$,
			\[
			\frac{\langle \nabla\Phi(p),e_i-p\rangle}{\Phi(p)}
			\longrightarrow -\infty
			\quad\text{as }p\to q,\ p\in[p_0,q).
			\]
			Hence
			\[
			1+\frac{\langle \nabla\Phi(p),e_i-p\rangle}{\Phi(p)}<0
			\]
			for all $p$ sufficiently close to $q$ along $[p_0,q)$. Using \eqref{eq:FGDifferential_Local},
			\[
			\pi_i(p)
			=
			p_i\left(1+\langle \nabla\varphi(p),e_i-p\rangle\right)
			=
			p_i\left(1+\frac{\langle \nabla\Phi(p),e_i-p\rangle}{\Phi(p)}\right),
			\]
			where $\varphi=\log\Phi$. Since $p_i>0$ for all $p\in\Delta^{(n)}$, it follows that $\pi_i(p)<0$ for all such $p$.
	\end{proof}

	\begin{remark}
			Corollary~\ref{coro:ShortSellingNearBoundary} shows that short positions arise on a nonempty open region of the domain. Indeed, there exists a point $p'\in[p_0,q)$ sufficiently close to $q$ such that $\pi_i(p')<0$. Since $\pi$ is continuous on $D$, there exists an open neighborhood $U\subset D$ of $p'$ such that
			\[
			\pi_i(p)<0,\quad p\in U.
			\]
			Therefore, for any nonempty open convex set $D'\subset U$, the restricted portfolio $\pi|_{D'}$ is generated by $\Phi|_{D'}$ and satisfies
			\[
			\pi_i(p)<0,\quad p\in D'.
			\]
			In this way, the vanishing condition $\Phi(q)=0$, together with concavity of $\Phi$, provides a natural mechanism for producing short-selling regions. For the concrete examples studied later, a more precise characterization of the corresponding short-selling region is given in Proposition~\ref{prop:ShortCondition}.
	\end{remark}
	While the divergence near the boundary forces short positions at the periphery, it is equally important to characterize the behavior at the center of the domain. Completing the geometric picture, the following proposition ensures that these strategies remain strictly long-only in a neighborhood of an interior maximizer.
	
	\begin{prop}[Existence of a Strictly Long-Only Region]
		\label{prop:LongOnlyRegion}
		Let $D\subset\Delta^{(n)}$ be open and convex, and let $\Phi:D\to(0,\infty)$ be a $C^1$ concave function. Let $\pi$ be the portfolio generated by $\Phi$ via \eqref{eq:FGDifferential_Local}. Suppose that $\Phi$ attains its global maximum on $D$ at some point $p^*\in D$. Then $\pi(p^*)=p^*$. Moreover, there exists an open neighborhood $U$ of $p^*$ such that
		\[
		\pi_i(p)>0,\qquad i=1,\ldots,n,\quad p\in U.
		\]
	\end{prop}
	\begin{proof}
		Since $p^*$ is an interior maximizer of $\Phi$ on $D$, the directional derivative of $\Phi$ at $p^*$ vanishes in every tangent direction $v\in T$. Equivalently,
		\[
		\langle \nabla\Phi(p^*),v\rangle=0,\qquad v\in T.
		\]
		By our convention, $\nabla\Phi(p^*)\in T$, and therefore $\nabla\Phi(p^*)=0$. Hence $\nabla\varphi(p^*)=0$, where $\varphi=\log\Phi$. Since $e_i-p^*\in T$ for each $i$,
		\[
		\langle \nabla\varphi(p^*),e_i-p^*\rangle=0,\qquad i=1,\ldots,n.
		\]
		Substituting into \eqref{eq:FGDifferential_Local} gives
		\[
		\pi_i(p^*)=p_i^*,\qquad i=1,\ldots,n,
		\]
		and thus $\pi(p^*)=p^*$.
		
		Because $p^*\in\Delta^{(n)}$, we have $p_i^*>0$ for all $i$. Moreover, $\pi$ is continuous on $D$ since $\Phi\in C^1(D)$. Therefore, after possibly shrinking to an open neighborhood $U$ of $p^*$, we obtain
		\[
		\pi_i(p)>0,\qquad i=1,\ldots,n,\quad p\in U.
		\]
	\end{proof}
	\subsubsection*{Proof of Geometric Phase Transition of Shrunken Portfolios}
	\begin{proof}[Proof of Theorem \ref{thm:GeometricPhaseTransition}]
		Assertion (i) follows from Proposition~\ref{prop:LongOnlyRegion}. Indeed, by symmetry and concavity, the barycenter $\bar e$ is a maximizer of $\Phi$ on $D$, and Proposition~\ref{prop:LongOnlyRegion} yields an open neighborhood $U$ of $\bar e$ on which $\pi$ is strictly long-only.
		
		Assertion (ii) is exactly Corollary~\ref{coro:ShortSellingNearBoundary}.
	\end{proof}

	\section{Constructive Strategies via Barycentric Scaling}
	\label{sec:ConstructionViaScaling}
	While the preceding section established the geometric necessity of short selling via the zero-boundary condition (Proposition \ref{prop:ZeroBoundary} and Theorem \ref{thm:GeometricPhaseTransition}), these results are primarily qualitative. To bridge the gap between abstract theory and practical implementation, we require a systematic methodology to engineer generating functions with prescribed boundary behavior.
	
	In this section, we introduce the \emph{barycentric scaling transformation}. This geometric technique allows us to continuously contract the domain of classical generating functions defined on the closed simplex $\overline{\Delta^{(n)}}$ into a strictly interior subdomain. By doing so, we create a user-defined zero boundary within the market space, thereby enforcing admissible short positions in a controlled and transparent manner.
	
	\subsection{The Barycentric Scaling Methodology}
	\label{subsec:ScalingMethodology}
	\begin{definition}[Barycentric Scaling Transformation]
		\label{def:ScalingTransformation}
		Let $\bar{e} = \left( \frac{1}{n}, \frac{1}{n}, \ldots, \frac{1}{n} \right)$ denote the barycenter of the simplex $\Delta^{(n)}$. For any scaling factor $0 < \lambda < 1$, the barycentric scaling map $S_\lambda$ is defined by
		\begin{align}
			S_\lambda: \overline{\Delta^{(n)}} &\to \overline{\Delta^{(n)}}, \nonumber \\
			x &\mapsto \lambda x + (1 - \lambda) \bar{e}. \label{eq:ShrinkTransformation}
		\end{align}
		This transformation uniformly contracts the simplex toward the barycenter $\bar{e}$ by the factor $\lambda$.
		
		The inverse map, termed the barycentric rescaling map, is given by
		\begin{align}
			S_\lambda^{-1}: \operatorname{Im}(S_\lambda) &\to \overline{\Delta^{(n)}}, \nonumber \\
			x &\mapsto \frac{1}{\lambda} x + \left(1 - \frac{1}{\lambda}\right) \bar{e}. \label{eq:ExpandTransformation}
		\end{align}
		Since $0 < \lambda < 1$, we have $1/\lambda>1$, and therefore $S_\lambda^{-1}$ expands away from the barycenter.
	\end{definition}
	
	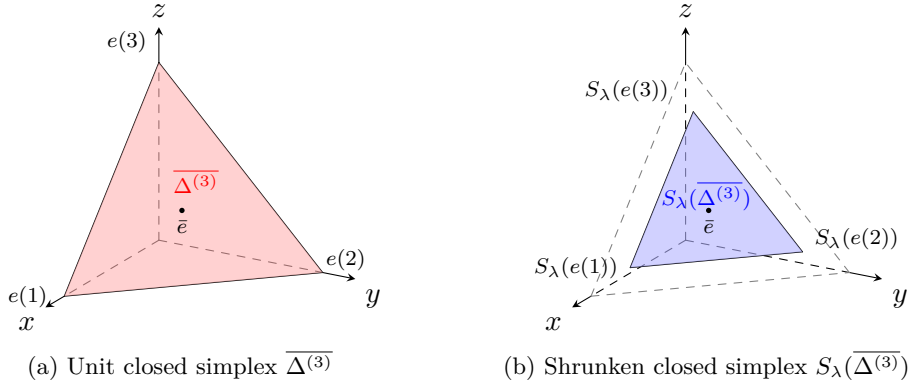
\begin{figure}[ht]
		\centering
		\subfloat[Unit closed simplex $\overline{\Delta^{(3)}}$]{
			\tdplotsetmaincoords{70}{120}
			\begin{tikzpicture}[scale=2.5, tdplot_main_coords, line join=round, line cap=round, >=stealth]
				\coordinate (O) at (0,0,0);
				\coordinate (A) at (1,0,0);
				\coordinate (B) at (0,1,0);
				\coordinate (C) at (0,0,1);
				
				\coordinate (Center) at (0.3333,0.3333,0.3333);
				
				\draw[-, dashed] (0,0,0) -- (1,0,0);
				\draw[->] (1,0,0) -- (1.2,0,0) node[anchor=north east]{$x$};
				\draw[-, dashed] (0,0,0) -- (0,1,0);
				\draw[->] (0,1,0) -- (0,1.2,0) node[anchor=north west]{$y$};
				\draw[-, dashed] (0,0,0) -- (0,0,1);
				\draw[->] (0,0,1) -- (0,0,1.2) node[anchor=south]{$z$};
				
				\draw[-] (A) -- (B);
				\draw[-] (A) -- (C);
				\draw[-] (B) -- (C);
				
				\fill[red!30, opacity=0.6] (A) -- (B) -- (C) -- cycle;
				
				\node[anchor=west, font=\scriptsize, xshift=-25pt] at (A) {$e(1)$};
				\node[anchor=east, font=\scriptsize, xshift=20pt, yshift=5pt] at (B) {$e(2)$};
				\node[anchor=south, font=\scriptsize, xshift=-12pt] at (C) {$e(3)$};
				
				\fill[black] (Center) circle (0.4pt);
				\node[anchor=north, font=\scriptsize] at (Center) {$\bar{e}$};
				
				\node[font=\scriptsize, red] at (0.3,0.4,0.5) {$\overline{\Delta^{(3)}}$};
			\end{tikzpicture}
		}
		\hspace{1cm}
		\subfloat[Shrunken closed simplex $S_{\lambda}(\overline{\Delta^{(3)}})$]{
			\tdplotsetmaincoords{70}{120}
			\begin{tikzpicture}[scale=2.5, tdplot_main_coords, line join=round, line cap=round, >=stealth]
				\coordinate (O) at (0,0,0);
				\coordinate (A) at (1,0,0);
				\coordinate (B) at (0,1,0);
				\coordinate (C) at (0,0,1);
				
				\coordinate (Center) at (0.3333,0.3333,0.3333);
				
				\coordinate (As) at ($(Center) + 0.6667*($(A)-(Center)$)$);
				\coordinate (Bs) at ($(Center) + 0.6667*($(B)-(Center)$)$);
				\coordinate (Cs) at ($(Center) + 0.6667*($(C)-(Center)$)$);
				
				\draw[-, dashed] (0,0,0) -- (1,0,0);
				\draw[->] (1,0,0) -- (1.2,0,0) node[anchor=north east]{$x$};
				\draw[-, dashed] (0,0,0) -- (0,1,0);
				\draw[->] (0,1,0) -- (0,1.2,0) node[anchor=north west]{$y$};
				\draw[-, dashed] (0,0,0) -- (0,0,1);
				\draw[->] (0,0,1) -- (0,0,1.2) node[anchor=south]{$z$};
				
				\draw[gray, thin, dashed] (A) -- (B);
				\draw[gray, thin, dashed] (A) -- (C);
				\draw[gray, thin, dashed] (B) -- (C);
				
				\draw[-] (As) -- (Bs);
				\draw[-] (As) -- (Cs);
				\draw[-] (Bs) -- (Cs);
				
				\fill[blue!30, opacity=0.6] (As) -- (Bs) -- (Cs) -- cycle;
				
				\node[anchor=west, font=\scriptsize, xshift=-40pt] at (As) {$S_\lambda(e(1))$};
				\node[anchor=east, font=\scriptsize, xshift=40pt, yshift=5pt] at (Bs) {$S_\lambda(e(2))$};
				\node[anchor=south, font=\scriptsize, xshift=-25pt] at (Cs) {$S_\lambda(e(3))$};
				
				\fill[black] (Center) circle (0.4pt);
				\node[anchor=north, font=\scriptsize] at (Center) {$\bar{e}$};
				
				\node[font=\scriptsize, blue] at (0.3,0.3,0.4) {$S_\lambda(\overline{\Delta^{(3)}})$};
			\end{tikzpicture}
		}
		\caption{Visualization of the barycentric scaling transformation in $\mathbb{R}^3$. The left panel shows the standard simplex, while the right panel displays the shrunken domain $S_\lambda(\overline{\Delta^{(3)}})$ nested within the original simplex (dashed lines).}
		\label{fig:ShrinkTransformationR3}
	\end{figure}
	
	As illustrated in Figure~\ref{fig:ShrinkTransformationR3}, $S_{\lambda}$ preserves the simplicial structure while mapping the standard closed simplex $\overline{\Delta^{(n)}}$ to a compact subset strictly contained in the relative interior of $\Delta^{(n)}$.
	
	\paragraph{Domain Construction and Strategy Generation.}
	We utilize this transformation to define the nested domains required for our analysis. Let $\lambda \in (0,1)$ be the domain parameter and $\varepsilon \in (0, \lambda)$ be the market domain parameter.
	\begin{enumerate}
		\item \textbf{Generating domain ($\overline{D_\lambda}$):} The domain of the generating function is the image of the closed simplex,
		\begin{equation}
			\overline{D_\lambda} \coloneqq S_\lambda(\overline{\Delta^{(n)}}) \subset \Delta^{(n)}.
		\end{equation}
		Equivalently, we write $D_\lambda \coloneqq S_\lambda(\Delta^{(n)})$ so that $\overline{D_\lambda}=S_\lambda(\overline{\Delta^{(n)}})$.
		The boundary $\partial D_\lambda = S_\lambda(\partial \Delta^{(n)})$ serves as the zero boundary required by Proposition \ref{prop:LogGradientDivergence}.
		
		\item \textbf{Market state space ($K_\varepsilon$):} The market state space is defined as a slightly smaller shrunken simplex,
		\begin{equation}
			K_\varepsilon \coloneqq S_\varepsilon(\Delta^{(n)}) \subset D_\lambda.
		\end{equation}
	\end{enumerate}

	Since $\varepsilon<\lambda$, the market state space $K_\varepsilon$ is a smaller shrunken simplex contained in $D_\lambda$. As $\varepsilon\uparrow\lambda$, the set $K_\varepsilon$ expands toward $\partial D_\lambda$.
	With this geometric setup, we construct strategies by pulling back standard concave functions. Let $\Psi: \overline{\Delta^{(n)}} \to [0, \infty)$ be a base concave function that is positive on $\Delta^{(n)}$ and vanishes on some points of the boundary $\partial \Delta^{(n)}$. We define the \emph{shrunken generating function} $\Phi_\lambda$ on $\overline{D_\lambda}$ by
	\begin{equation}
		\Phi_\lambda(p) \coloneqq \Psi(S_\lambda^{-1}(p)).
	\end{equation}
	By construction, $\Phi_\lambda$ is concave on $\overline{D_\lambda}$. Moreover, $\Phi_\lambda$ vanishes on those boundary points $q\in\partial D_\lambda$ for which $S_\lambda^{-1}(q)$ belongs to the zero set of $\Psi$ on $\partial\Delta^{(n)}$. The gap $\lambda - \varepsilon$ controls the proximity of the market state space $K_\varepsilon$ to this zero boundary. When $\varepsilon$ is close to $\lambda$, the set $K_\varepsilon$ approaches $\partial D_\lambda$ and hence enters the regime where the logarithmic gradient becomes steep, which leads to short positions by Corollary \ref{coro:ShortSellingNearBoundary}.
	
	We now examine two canonical classes of such portfolios, distinguished by the structure of their zero sets.
	
	\subsection{Example 1: The Shrunken Equal-Weighted Portfolio (SEWP)}
	\label{subsubsec:SEWP}
	Our first example generalizes the classical equal-weighted portfolio (EWP) within the bankruptcy-proof setting. The naive $1/n$ rule is arguably the most widely accepted and transparent diversification benchmark: it requires essentially no estimation and is frequently used as a baseline in the academic asset-allocation literature \cite{BenartziThaler2001NaiveDiversification,DeMiguel2009Equalstragtegy}. In particular, DeMiguel, Garlappi and Uppal \cite{DeMiguel2009Equalstragtegy} show that the rebalanced $1/n$ strategy is remarkably difficult to outperform out of sample, which helps explain its continued role as a strong reference strategy alongside more sophisticated optimization-based allocations \cite{Duchin2009MarkowitzVersusTalmudic,Pflug2012EqualInvestmentStrategyIsOptimal,TU2011MarkowitzMeetsTalmud}.
	From a mathematical-finance perspective, EWP is a special case of constant-weighted portfolios (CWPs), for which one can obtain model-free long-run performance guarantees in terms of log-wealth growth. In particular, Cover's universal portfolio theory provides a distribution-free procedure whose asymptotic growth rate competes with the best retrospectively chosen CWP \cite{Cover1991UniversalPortfolios}; further algorithmic and continuous-time links to stochastic portfolio theory and the num\'eraire portfolio are developed in \cite{Cuchiero2019Cover'sUniversalPortfolioSPT,Helmbold1998MultiplicativeUpdates}.
	We apply barycentric scaling to the EWP to regulate its boundary behavior. The base function is the geometric mean,
	\begin{equation}
		\Psi_{\text{EWP}}(x) = \left( \prod_{i=1}^{n} x_i \right)^{\frac{1}{n}},
	\end{equation}
	which is concave on $\Delta^{(n)}$ and vanishes on the entire boundary $\partial \Delta^{(n)}$.
	
	By applying the barycentric scaling construction from Section \ref{subsec:ScalingMethodology}, we define the \emph{shrunken EWP} (SEWP).
	
	\begin{definition}[Shrunken Equal-Weighted Portfolio]
		\label{def:SEWP}
		Let $\lambda \in (0, 1)$ define the generating domain $D_\lambda = S_\lambda(\Delta^{(n)})$. The SEWP is the portfolio $\pi$ generated by the function
		\begin{equation}
			\Phi_{\text{SEWP}}(p) \coloneqq \Psi_{\text{EWP}} \circ S_\lambda^{-1}(p) : D_\lambda \to (0, \infty).
		\end{equation}
		By standard calculus, the portfolio weights for any market weight vector $\mu \in D_\lambda$ are given component-wise by
		\begin{equation}
			\frac{\pi_i(\mu)}{\mu_i} = \frac{1}{n \mu_i + \lambda - 1} + 1 - \sum_{j=1}^{n} \frac{\mu_j}{n \mu_j + \lambda - 1}.
			\label{eq:SEWPShortsell}
		\end{equation}
	\end{definition}
	
	Since $0<\lambda<1$, the SEWP is no longer a static portfolio; when $\lambda=1$, the construction reduces to the classical equal-weighted portfolio. This strategy exhibits a geometric phase transition. Since $\Phi_{\text{SEWP}}$ attains its maximum at the barycenter $\bar{e}$, Proposition \ref{prop:LongOnlyRegion} guarantees the existence of a long-only region around the geometric center of the simplex. Conversely, since $\Phi_{\text{SEWP}}$ vanishes on the entire boundary $\partial D_\lambda$, Corollary \ref{coro:ShortSellingNearBoundary} implies that short positions are structurally enforced near $\partial D_\lambda$.
	
	We now quantify this behavior within the market state space $K_\varepsilon = S_\varepsilon(\Delta^{(n)})$ (where $0 < \varepsilon < \lambda$). Based on the portfolio weights, we distinguish three regions in $K_\varepsilon$: a central core of long-only positions, an outer shell involving short-selling positions, and a transition zone. This is a qualitative description based on sufficient conditions and does not claim an exact partition.
	
	\paragraph{Region I: The Long-Only Core.}
	The following proposition identifies a sufficient condition for the strategy to remain long-only within the market state space $K_\varepsilon$.
	Since $\Phi_{\mathrm{SEWP}}$ is $C^1$ on $D_\lambda$, the portfolio map $\pi$ given by \eqref{eq:FGDifferential_Local} is continuous on $D_\lambda$. Moreover, $\overline{K_\varepsilon}\subset D_\lambda$ and $\overline{K_\varepsilon}$ is compact. Therefore, $\pi$ is long-only on $K_\varepsilon$ if and only if
	$\min_{\mu \in \overline{K_\varepsilon}}\pi_i(\mu)\ge 0$ for all $1\le i\le n$.
	In particular, it suffices to analyze nonnegativity on the compact set $\overline{K_\varepsilon}$.
	
	\begin{prop}[Threshold for Pure Long-Only Behavior]
		\label{prop:EWPAllLongArea}
		Consider the SEWP defined by $\lambda$. There exists a unique critical threshold $\varepsilon^* \in (0, \lambda)$ given by
		\begin{equation}
			\varepsilon^* = \frac{\lambda^2}{(n-1) - (n-2)\lambda},
			\label{eq:OptimalEpsilon}
		\end{equation}
		such that if $0 < \varepsilon \le \varepsilon^*$, then the portfolio is long-only (i.e., $\pi_i(\mu)\ge 0$) for all $\mu \in \overline{K_\varepsilon}$.
	\end{prop}
	
	\begin{proof}
		Let $f_i(\mu)$ denote the normalized portfolio weight function defined in \eqref{eq:SEWPShortsell}:
		\[
		f_i(\mu) \coloneqq \frac{\pi_i(\mu)}{\mu_i}
		= \frac{1}{n \mu_i + \lambda - 1} + 1 - \sum_{j=1}^{n} \frac{\mu_j}{n \mu_j + \lambda - 1}.
		\]
		Since $\mu_i > 0$, the condition $\pi_i(\mu) \ge 0$ is equivalent to $f_i(\mu) \ge 0$.
		By symmetry, it suffices to analyze $f_1$. Since $f_1$ is continuous on $D_\lambda$ and $\overline{K_\varepsilon}$ is compact, we have
		$\inf_{\mu\in K_\varepsilon} f_1(\mu)=\min_{\mu\in \overline{K_\varepsilon}} f_1(\mu)$.
		Therefore, $\pi$ is long-only on $K_\varepsilon$ if and only if $\min_{\mu\in \overline{K_\varepsilon}} f_1(\mu)\ge 0$.
		
		We analyze the gradient of $f_1(\mu)$ to locate its minimum. Direct computation yields
		\begin{equation}
			\frac{\partial f_1}{\partial \mu_1} = -\frac{n+\lambda-1}{(n\mu_1+\lambda-1)^2},
			\qquad
			\frac{\partial f_1}{\partial \mu_j} = \frac{1-\lambda}{(n\mu_j+\lambda-1)^2} \quad \text{for } j \neq 1.
		\end{equation}
		Given $n \ge 2$ and $0 < \lambda < 1$, we have $-(n+\lambda-1) < 0$ and $1-\lambda > 0$, and hence
		\begin{equation}
			\frac{\partial f_1}{\partial \mu_1} < 0,
			\qquad
			\frac{\partial f_1}{\partial \mu_j} > 0 \quad \text{for all } j \neq 1.
		\end{equation}
		
		These inequalities imply that $f_1$ has no interior critical points. Moreover, to decrease $f_1$ one must increase $\mu_1$ and decrease $\mu_j$ for $j\neq 1$. Since $\overline{K_\varepsilon}$ is a simplex defined by linear constraints, the unique point that maximizes the first coordinate and minimizes the others is the vertex
		$\mu^* \coloneqq S_\varepsilon(e_1)$. In coordinates,
		\[
		\mu^* = \left(\varepsilon + \frac{1-\varepsilon}{n}, \frac{1-\varepsilon}{n}, \ldots, \frac{1-\varepsilon}{n}\right).
		\]
		
		Let $g(\varepsilon) \coloneqq f_1(\mu^*)$ denote the minimum value of $f_1$ over $\overline{K_\varepsilon}$ as a function of $\varepsilon$. Substituting $\mu^*$ into the expression for $f_1$, we obtain
		\begin{equation}
			g(\varepsilon)
			= 1 + \frac{(1-\varepsilon)(n-1)}{n}\left[\frac{1}{\lambda+(n-1)\varepsilon}-\frac{1}{\lambda-\varepsilon}\right].
		\end{equation}
		We record the boundary behavior:
		\begin{itemize}
			\item As $\varepsilon \to 0$, $\mu^* \to \bar{e}$. Since the SEWP is long-only at $\bar e$ (Proposition \ref{prop:LongOnlyRegion}), we have $g(0) = f_1(\bar{e}) = 1 > 0$.
			\item As $\varepsilon \uparrow \lambda$, $\mu^*$ approaches the boundary of $D_\lambda$, and the denominator $\lambda-\varepsilon$ tends to $0$. Consequently, $g(\varepsilon)\to -\infty$, which is consistent with Corollary \ref{coro:ShortSellingNearBoundary}.
		\end{itemize}
		A direct computation shows that for $\varepsilon\in(0,\lambda)$
		\[g'(\varepsilon)=\frac{n-1}{n}[\frac{\lambda-1}{(\lambda-\varepsilon)^2}-\frac{\lambda+n-1}{(\lambda+(n-1)\varepsilon)^2}]<0,\]
		and hence $g$ is strictly decreasing on $(0,\lambda)$.
		By the intermediate value theorem, there exists a unique $\varepsilon^* \in (0, \lambda)$ such that $g(\varepsilon^*) = 0$.
		
		Solving $g(\varepsilon^*) = 0$ yields \eqref{eq:OptimalEpsilon}.
		Therefore, $f_i(\mu) \ge 0$ for all $i$ and all $\mu \in \overline{K_\varepsilon}$ holds if and only if the global minimum is nonnegative, that is, $g(\varepsilon)\ge 0$, which is equivalent to $0 < \varepsilon \le \varepsilon^*$.
	\end{proof}
	
	\begin{remark}
		Notably, for a fixed domain parameter $\lambda$, $\varepsilon^*$ is a decreasing function of the dimension $n$. Therefore, the sufficient long-only simplex region $\overline{K_\varepsilon}$ guaranteed by Proposition \ref{prop:EWPAllLongArea} becomes smaller as $n$ increases (for fixed $\lambda$ and $\varepsilon\le \varepsilon^*$).
		We stress that this conclusion concerns only this sufficient region and does not claim that the entire long-only region of the SEWP shrinks with $n$.
	\end{remark}
	
	\paragraph{Region II: The Short-Selling Shell.}
	When the market moves sufficiently far from the center, the strategy is forced to take short positions in assets that become relatively large in market weight within the domain.\footnote{We use ``short-selling shell'' as a geometric descriptor; this is not meant to imply an exact partition of $K_\varepsilon$.} The next proposition provides a sufficient condition for short selling in terms of market concentration.
	
	\begin{prop}[Sufficient Condition for Short Positions]
		\label{prop:ShortCondition}
		Let $\varepsilon^*$ be the unique threshold defined in Proposition \ref{prop:EWPAllLongArea}, and define
		\begin{equation}
			r_1^* \coloneqq \varepsilon^* + \frac{1-\varepsilon^*}{n}.
		\end{equation}
		If the market weight vector $\mu$ satisfies $\mu_i > r_1^*$ for some $i$, then the SEWP necessarily takes a short position in asset $i$, that is,
		\begin{equation}
			\mu_i > r_1^* \implies \pi_i(\mu) < 0.
		\end{equation}
	\end{prop}
	
	\begin{proof}
		Without loss of generality, assume $\mu_1 > r_1^*$. We aim to show that $f_1(\mu) < 0$, where $f_1$ is defined in \eqref{eq:SEWPShortsell}.
		
		Let $r = S_{\varepsilon^*}(e_1)$. By definition, its first coordinate is $r_1 = r_1^*$, and from the proof of Proposition \ref{prop:EWPAllLongArea} we have $f_1(r) = 0$.
		
		Our proof compares $\mu$ with a reference point on the symmetry axis $L$ connecting the barycenter $\bar{e}$ to the vertex $e_1$.
		
		\emph{Step 1: Monotonicity along the symmetry axis.}
		Any $q \in L$ has the form $q = (x, \frac{1-x}{n-1}, \dots, \frac{1-x}{n-1})$. Using the sign structure of the partial derivatives established above, increasing the first component and decreasing the remaining components strictly decreases $f_1$. Thus, for any $q \in L$ with $q_1 > r_1^*$,
		\[
		f_1(q) < f_1(r) = 0.
		\]
		
		\emph{Step 2: Jensen's inequality.}
		Now fix $\mu\notin L$ with $\mu_1>r_1^*$ and let $q\in L$ be the point with $q_1=\mu_1$ and $q_j=(1-\mu_1)/(n-1)$ for $j\ge 2$. Then $q_1>r_1^*$, so Step 1 gives $f_1(q)<0$. We claim that $f_1(\mu)\le f_1(q)$.
		
		Write
		\[
		f_1(\mu) - f_1(q)
		= \sum_{j=2}^{n} \left( \frac{q_j}{nq_j+\lambda-1} - \frac{\mu_j}{n\mu_j+\lambda-1} \right)
		= \sum_{j=2}^{n} \bigl( h(q_j) - h(\mu_j) \bigr),
		\]
		where $h(x) \coloneqq \frac{x}{nx+\lambda-1}$. A computation gives
		\[
		h''(x) = \frac{-2n(\lambda-1)}{(nx+\lambda-1)^3},
		\]
		which is strictly positive on $D_\lambda$ (since $\lambda<1$ and $nx+\lambda-1>0$ on $D_\lambda$). Hence $h$ is strictly convex on $D_\lambda$.
		Moreover, $q_j$ is the arithmetic mean of the coordinates $\{\mu_j\}_{j=2}^n$. By Jensen's inequality,
		\[
		\frac{1}{n-1}\sum_{j=2}^n h(\mu_j)
		\ge h\!\left(\frac{\sum_{j=2}^n \mu_j}{n-1}\right)
		= h(q_j).
		\]
		Multiplying by $n-1$ yields $\sum_{j=2}^n h(\mu_j) \ge \sum_{j=2}^n h(q_j)$, and hence
		\[
		f_1(\mu) - f_1(q) \le 0.
		\]
		Therefore $f_1(\mu)\le f_1(q)<0$, which implies $f_1(\mu)<0$ and hence $\pi_1(\mu)<0$.
	\end{proof}
	
	\paragraph{Geometry of the State Space.}
	Combining these results, for a given domain parameter $\lambda$, the size $\varepsilon$ of the market state space determines the qualitative behavior of the strategy:
	\begin{enumerate}
		\item \textbf{Pure Long-Only Phase ($0 < \varepsilon \le \varepsilon^*$):} The market state space $K_\varepsilon$ is contained entirely within the long-only region. The strategy behaves like a classical functionally generated portfolio, rebalancing without leverage.
		\item \textbf{Mixed Phase ($\varepsilon^* < \varepsilon < \lambda$):} The market state space $K_\varepsilon$ extends into a region where short positions occur.
		\begin{itemize}
			\item The core region $K_{\varepsilon^*}$ remains long-only.
			\item A sufficient condition for short selling is $\max_{1\le i\le n}\mu_i>r_1^*$, as in Proposition \ref{prop:ShortCondition}.
		\end{itemize}
	\end{enumerate}

	To visualize the qualitative regimes predicted by Propositions \ref{prop:EWPAllLongArea} and \ref{prop:ShortCondition}, we present a numerical example in $\mathbb{R}^3$ with $\lambda = 0.885$.\footnote{The choice of $\lambda$ is aligned with the empirical study in Section~\ref{sec:EmpiricalExamples}. The parameter is chosen to be neither too small nor too close to $1$. If $\lambda$ is too small, then for market weights close to the boundary of the simplex, the SEWP may fail to be well defined on the relevant market state space. If $\lambda$ is too close to the market boundary, then the portfolio can exhibit a large short-selling magnitude near $\partial D_\lambda$ (cf.\ Corollary~\ref{coro:ShortSellingNearBoundary}), which obscures the qualitative behavior in three-dimensional visualizations. On the other hand, if $\lambda$ is too close to $1$, then \eqref{eq:SEWPShortsell} shows that the SEWP approaches a constant-weighted (equal-weighted) portfolio, making it harder to illustrate the distinctive features of the SEWP.}
	Figure \ref{fig:f_3_SEWP_R3} plots the normalized portfolio weight of the third asset,
	\begin{equation}
		f^{\text{SEWP}}_3(\mu) \coloneqq \frac{\pi_3(\mu)}{\mu_3}
		= \frac{1}{3 \mu_3 + \lambda - 1} + 1 - \sum_{j=1}^{3} \frac{\mu_j}{3 \mu_j + \lambda - 1}.
	\end{equation}
	Since $\mu_3 > 0$, the sign of $f^{\text{SEWP}}_3$ is the same as the sign of $\pi_3$. Therefore, it suffices to analyze $f^{\text{SEWP}}_3$ in order to identify regions of short selling.
	
	\begin{figure}[H]
		\centering
		\subfloat[Pure long-only phase: $K=S_{\varepsilon^*}(\Delta^{(3)})$]{
			\includegraphics[width=0.45\linewidth]{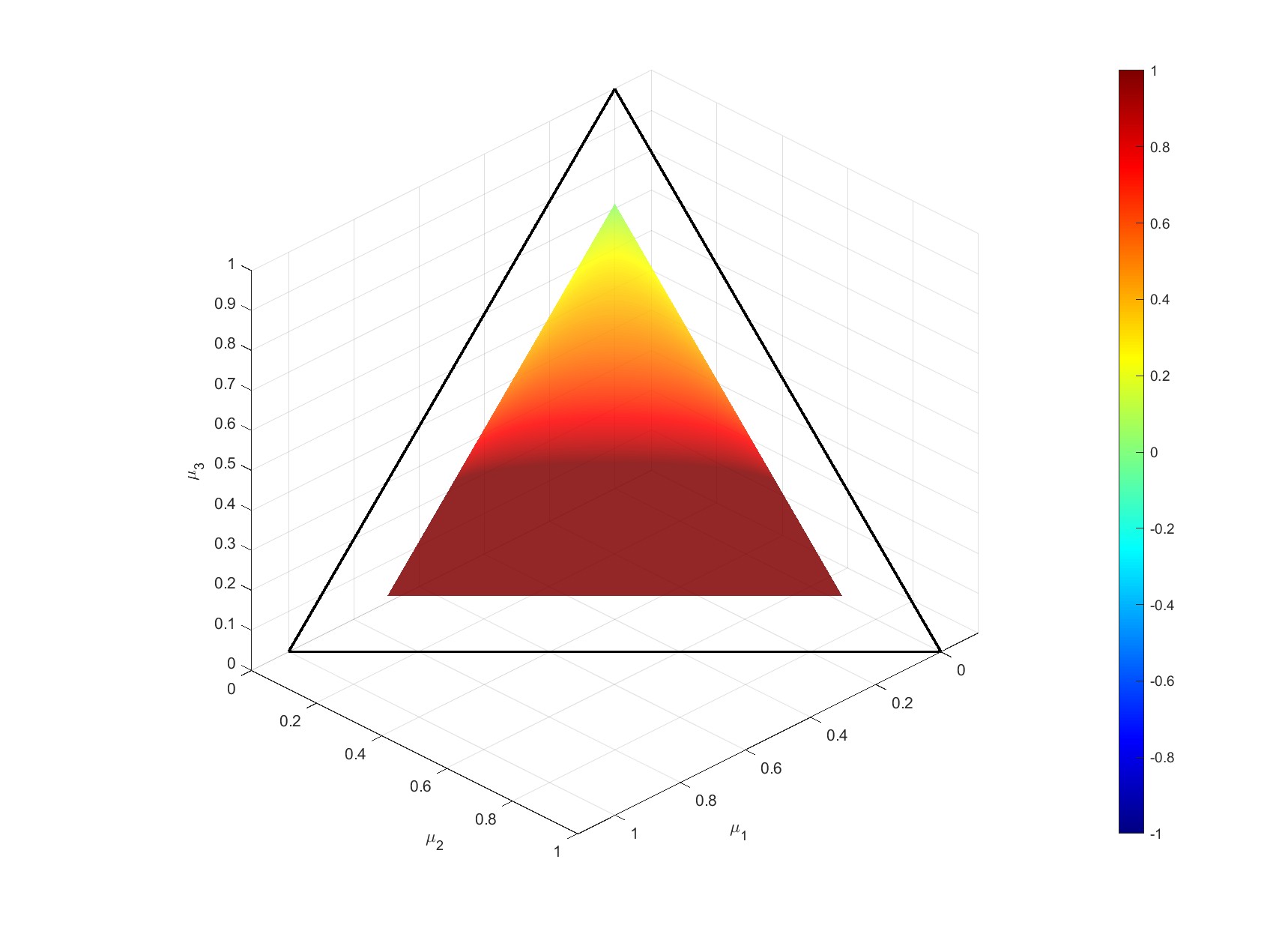}}
		\hspace{0in}
		\subfloat[Mixed phase: $K=S_{0.88}(\Delta^{(3)})$]{
			\includegraphics[width=0.45\linewidth]{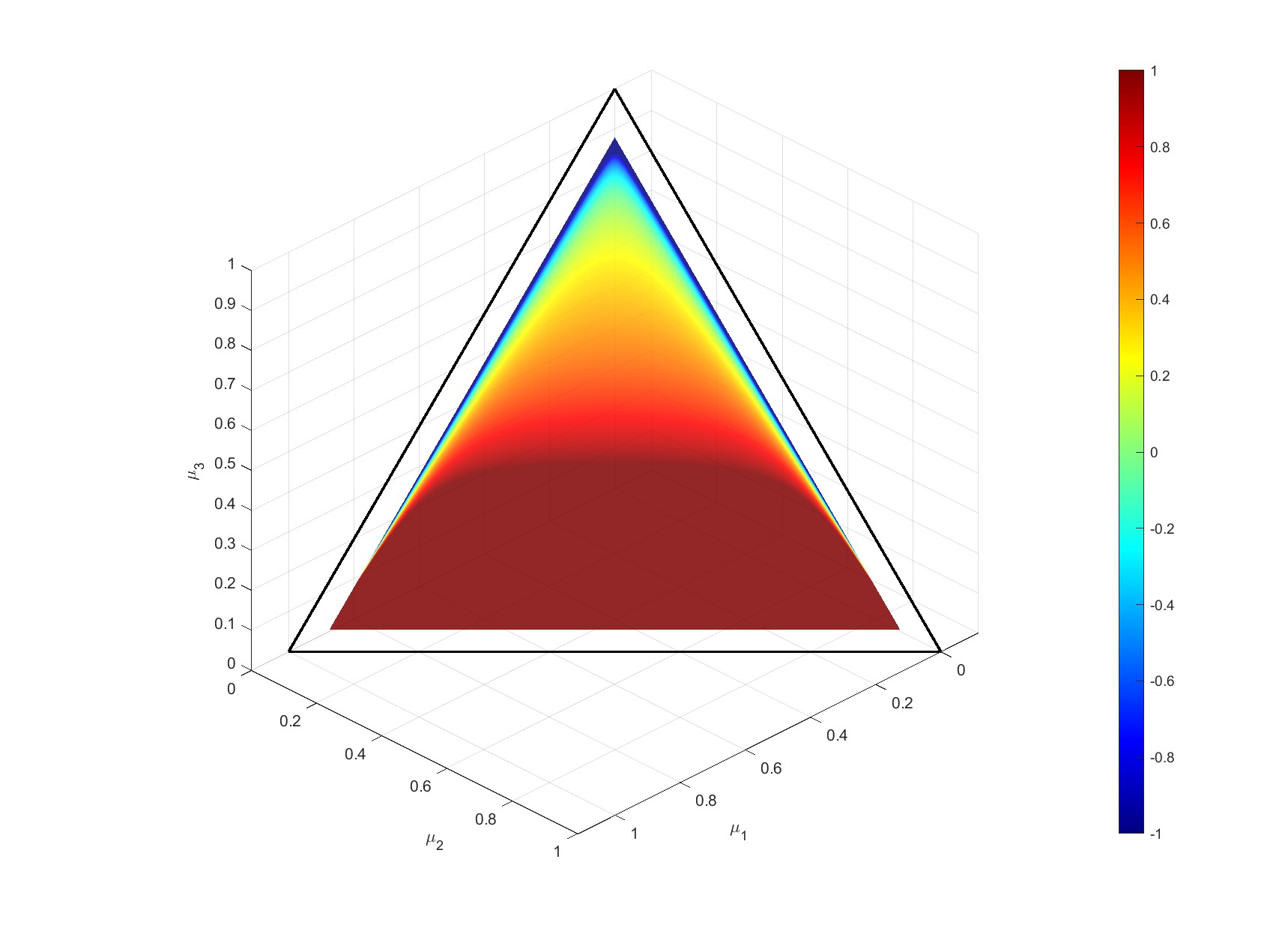}}
		\caption{Heatmap of the normalized third component $f^{\text{SEWP}}_3 = \pi_3/\mu_3$ with $\lambda=0.885$. Warm colors denote long positions ($f > 0$), while cool colors denote short positions ($f < 0$).}
		\label{fig:f_3_SEWP_R3}
	\end{figure}
	
	Subfigure (a) corresponds to the critical threshold $\varepsilon^*=0.70$ derived in Proposition \ref{prop:EWPAllLongArea}. The entire surface remains nonnegative, confirming that the strategy is long-only on this core region.
	
	Subfigure (b) expands the market state space to $\varepsilon = 0.88 > \varepsilon^*$.\footnote{The value $\varepsilon=0.88$ is chosen primarily for visualization. Since we aim to exhibit short selling, we require $\varepsilon>\varepsilon^*=0.70$. We therefore choose $\varepsilon$ to be above $\varepsilon^*$ while remaining below $\lambda$.}
	As predicted by Proposition \ref{prop:ShortCondition}, short positions (blue regions) emerge near the vertex $e_3$ where the market weight $\mu_3$ is highly concentrated.
	It is worth noting that the short-selling region visualized here extends beyond the sufficient condition in Proposition \ref{prop:ShortCondition}. The zero contour of $f^{\text{SEWP}}_3$ defines the exact boundary of the set $\{\mu:\pi_3(\mu)<0\}$, revealing a transition region that is more complex than the sufficient conditions alone.
	\subsection{Example 2: The Shrunken Entropy Portfolio (SEP)}
	\label{subsubsec:SEP}
	
	Our second example follows the same barycentric scaling framework but uses the Shannon entropy function \cite{Cover2006ElementsofInformationTheory,Jost2006EntropyandDiversity}. While the resulting expressions differ from those of the SEWP, the geometric intuition is analogous.
	
	The choice of Shannon entropy is motivated by robust theoretical and empirical foundations. In mathematical finance, it serves as a canonical measure of market diversity in Stochastic Portfolio Theory \cite{Fernholz2002SPT}.
	
	From a risk management perspective, recent work by Mercurio et al.\ \cite{Mercurio2020Entropy-BasedApproach} highlights entropy's efficacy in accommodating non-normal and asymmetric return distributions without relying on covariance matrices. Furthermore, Ruf and Xie \cite{Ruf2020ImpactPropotioalTransactionCosts} empirically study the impact of proportional transaction costs across several systematically generated portfolios and find that the entropy-weighted portfolio is less sensitive to transaction costs than the equally weighted portfolio; however, its relative performance depends on the rebalancing and market-configuration choices (e.g., trading frequency and constituent list size).
	
	Leveraging these information-theoretic and geometric properties, the underlying generating function is defined as
	\begin{equation}
		\Psi_{\text{Ent}}(x) = -\sum_{i=1}^{n} x_i \log x_i.
	\end{equation}
	By applying the barycentric scaling transformation, we construct the \emph{shrunken entropy-weighted portfolio} (SEP).
	
	\begin{definition}[Shrunken Entropy-weighted Portfolio]
		\label{def:SEP}
		Let $\lambda \in (0, 1)$ define the generating domain $D_\lambda = S_\lambda(\Delta^{(n)})$. The SEP is generated by the function
		\begin{equation}
			\Phi_{\text{SEP}}(p) \coloneqq \Psi_{\text{Ent}} \circ S_\lambda^{-1}(p) : D_\lambda \to (0, \infty).
		\end{equation}
		
		Substituting $\Phi_{\text{SEP}}$ into the local generation formula \eqref{eq:FGDifferential_Local}, we obtain the portfolio weights explicitly.
		For a market weight vector $\mu \in D_\lambda$, the normalized weight for asset $i$ is given by
		\begin{equation}
			\label{eq:SEP_Weights_Mu}
			\frac{\pi_i(\mu)}{\mu_i}
			= 1 + \frac{\log y_i(\mu) - \sum_{j=1}^{n} \mu_j \log y_j(\mu)}{\lambda \sum_{j=1}^{n} y_j(\mu) \log y_j(\mu) }
			\coloneqq f^{\text{SEP}}_i(\mu),
		\end{equation}
		where $y(\mu) = S^{-1}_\lambda(\mu)$ represents the inverse image of $\mu$ in the original simplex.
		In particular, $\pi$ is well-defined and bankruptcy-proof on any market state space $K_\varepsilon\subset D_\lambda$ with $0<\varepsilon<\lambda$.
	\end{definition}
	
	\paragraph{Region I: The Pure Long-Only Core.}
	Similar to the SEWP, the SEP exhibits a pure long-only region in a neighborhood of the barycenter. However, due to the transcendental nature of the entropy function, an explicit algebraic expression for the threshold is generally unavailable. Instead, we establish existence of a critical threshold by continuity. As in Proposition \ref{prop:EWPAllLongArea}, since $\Phi_{\text{SEP}}\in C^1(D_\lambda)$, the portfolio map $\pi$ (and hence $f_i^{\text{SEP}}$) is continuous on $D_\lambda$, and therefore it suffices to analyze the minimum of $f_i^{\text{SEP}}$ over the compact set $\overline{K_\varepsilon}$ to determine whether $\pi$ is long-only on $K_\varepsilon$.
	
	\begin{prop}[Existence of Long-Only Threshold]
		\label{prop:SEPAllLongArea}
		Consider the market state space $K_\varepsilon = S_\varepsilon(\Delta^{(n)})$ with $0 < \varepsilon < \lambda$. There exists a unique threshold $\varepsilon^* \in (0, \lambda)$ such that:
		\begin{enumerate}
			\item If $0 < \varepsilon \le \varepsilon^*$, the SEP maintains long-only positions for all $\mu \in \overline{K_\varepsilon}$.
			\item If $\varepsilon^* < \varepsilon < \lambda$, the SEP involves short selling for some $\mu \in \overline{K_\varepsilon}$.
		\end{enumerate}
	\end{prop}
	
	\begin{proof}
		Define
		\[
		g(\varepsilon) \coloneqq \min_{\mu \in \overline{K_\varepsilon}} \min_{1\le i\le n} f^{\text{SEP}}_i(\mu).
		\]
		Since $f^{\text{SEP}}$ is continuous on $D_\lambda$ and $\overline{K_\varepsilon}$ is compact, the minimum is attained for each $\varepsilon\in(0,\lambda)$. Moreover, the family $\{\overline{K_\varepsilon}\}_{\varepsilon\in(0,\lambda)}$ is nested, and hence $g(\varepsilon)$ is nonincreasing in $\varepsilon$.
		
		Let $\mathcal{E} = \{ \varepsilon \in (0, \lambda) \mid g(\varepsilon) \ge 0 \}$ be the set of parameters for which the SEP is long-only on $K_\varepsilon$.
		\begin{itemize}
			\item \emph{Non-emptiness:} At the barycenter, we have $f^{\text{SEP}}_i(\bar{e}) = 1$ for all $i$. By continuity, $g(\varepsilon) > 0$ for sufficiently small $\varepsilon$, implying $\mathcal{E} \neq \emptyset$.
			\item \emph{Failure near the boundary:} As $\varepsilon \uparrow \lambda$, the vertices of $K_\varepsilon$ approach the boundary of $D_\lambda$. By Corollary \ref{coro:ShortSellingNearBoundary}, some portfolio weights become negative near $\partial D_\lambda$, and hence $g(\varepsilon) < 0$ for $\varepsilon$ sufficiently close to $\lambda$. In particular, $\mathcal{E}$ is bounded above by a value strictly less than $\lambda$.
		\end{itemize}
		
		Let $\varepsilon^* \coloneqq \sup \mathcal{E}\in(0,\lambda)$. Since $g$ is nonincreasing, it follows that $g(\varepsilon)\ge 0$ for all $\varepsilon<\varepsilon^*$ and $g(\varepsilon)<0$ for all $\varepsilon>\varepsilon^*$. By continuity of $g$, we have $g(\varepsilon^*)=0$.
		This proves the existence of the threshold and the two regimes stated in the proposition.
	\end{proof}
	
	\paragraph{Geometric Asymptotic Analysis.}
	The following heuristic asymptotic analysis helps explain the geometry of the short-selling region for the SEP.
	Since an algebraic solution for the long-only threshold is intractable, we rely on asymptotic analysis to clarify the mechanism. The portfolio weight is driven by the relative gradient term $\nabla \Phi/\Phi$.
	For the SEP, let $H(y) \coloneqq -\sum_{k=1}^n y_k \log y_k$ denote the entropy function on the latent variable $y = S_\lambda^{-1}(p)$. Since $S_\lambda$ is affine, the divergence behavior of the gradient with respect to $p$ is asymptotically equivalent to that of $\nabla_y H(y)/H(y)$.
	
	Let $d\downarrow 0$ quantify proximity to the boundary, for example $d\asymp \min_{1\le j\le n} y_j$, which is comparable (up to constants) to the Euclidean distance from $y$ to $\partial\Delta^{(n)}$ in the regimes considered below.
	\begin{itemize}
		\item \emph{Regime A: Near the vertices ($y \to e_i$).}
		Consider a sequence approaching the vertex $e_i$. In this limit, $H(y)\to 0$ since $1\log 1=0$ and $\lim_{z\downarrow 0} z\log z=0$.
		For instance, if $y_j\sim d$ for all $j\ne i$, then $H(y)\sim -(n-1)d\log d$.
		Moreover, the partial derivative diverges logarithmically: $\partial_{y_j}H(y)=-1-\log y_j \sim -\log d$.
		Consequently,
		\[
		\frac{\|\nabla H(y)\|}{H(y)} \sim \text{constant}\times\frac{|\log d|}{-d \log d} = O(\frac{1}{d}).
		\]
		This $O(1/d)$ divergence represents a comparatively strong effect: the vanishing denominator $H(y)\to 0$ amplifies the gradient term, so short selling may occur even when the market moves moderately far from the barycenter toward a vertex.
		
		\item \emph{Regime B: Near the center of a boundary face.}
		Consider a sequence approaching the center of a boundary face (for example, for $n=3$ one may take $y\to (1/2,1/2,0)$). Here, the minor component satisfies $y_n\to 0$, so $|\nabla H(y)|\sim |\log d|$.
		Crucially, $H(y)$ does \emph{not} vanish; it converges to the strictly positive entropy of the lower-dimensional face. Thus,
		\[
		\frac{\|\nabla H(y)\|}{H(y)} \sim \frac{|\log d|}{\text{constant}} = O(|\log d|).
		\]
		This logarithmic divergence is much weaker than the vertex case. It implies that the ``long-only'' region extends much closer to the boundary near the face centers than it does near the vertices.
	\end{itemize}
	This analysis suggests that the strongest blow-up of the logarithmic gradient is localized near the vertices. As a result, the short-selling region for the SEP is expected to form distinct ``lobes'' near the corners of the simplex, rather than a uniform short-selling shell.
	
	We now illustrate the behavior of the SEP and compare it with the SEWP. Figure \ref{fig:f_3_SEP_R3} displays the third component $f^{\text{SEP}}_3$ on the market state space $K_\varepsilon \subset \Delta^{(3)}\subset \mathbb{R}^3$ with $\lambda = 0.885$. The visualization convention mirrors that of the SEWP: cool colors indicate short positions, while warm colors represent long positions.
	
	\begin{figure}[H]
		\centering
		\subfloat[Pure long-only phase: $K=S_{0.5}(\Delta^{(3)})$]{
			\includegraphics[width=0.45\linewidth]{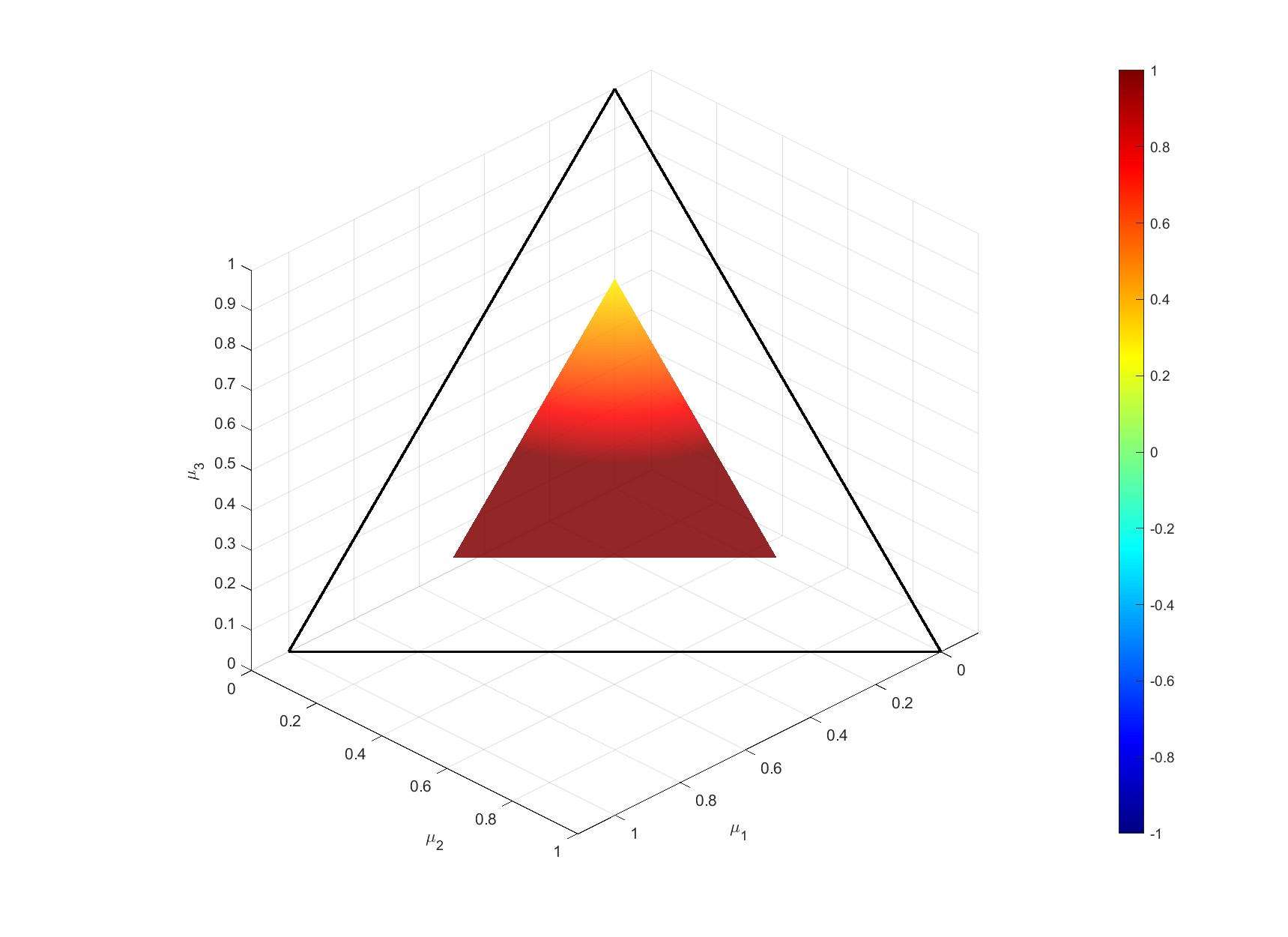}}
		\hspace{0in}
		\subfloat[Mixed phase: $K=S_{0.88}(\Delta^{(3)})$]{
			\includegraphics[width=0.45\linewidth]{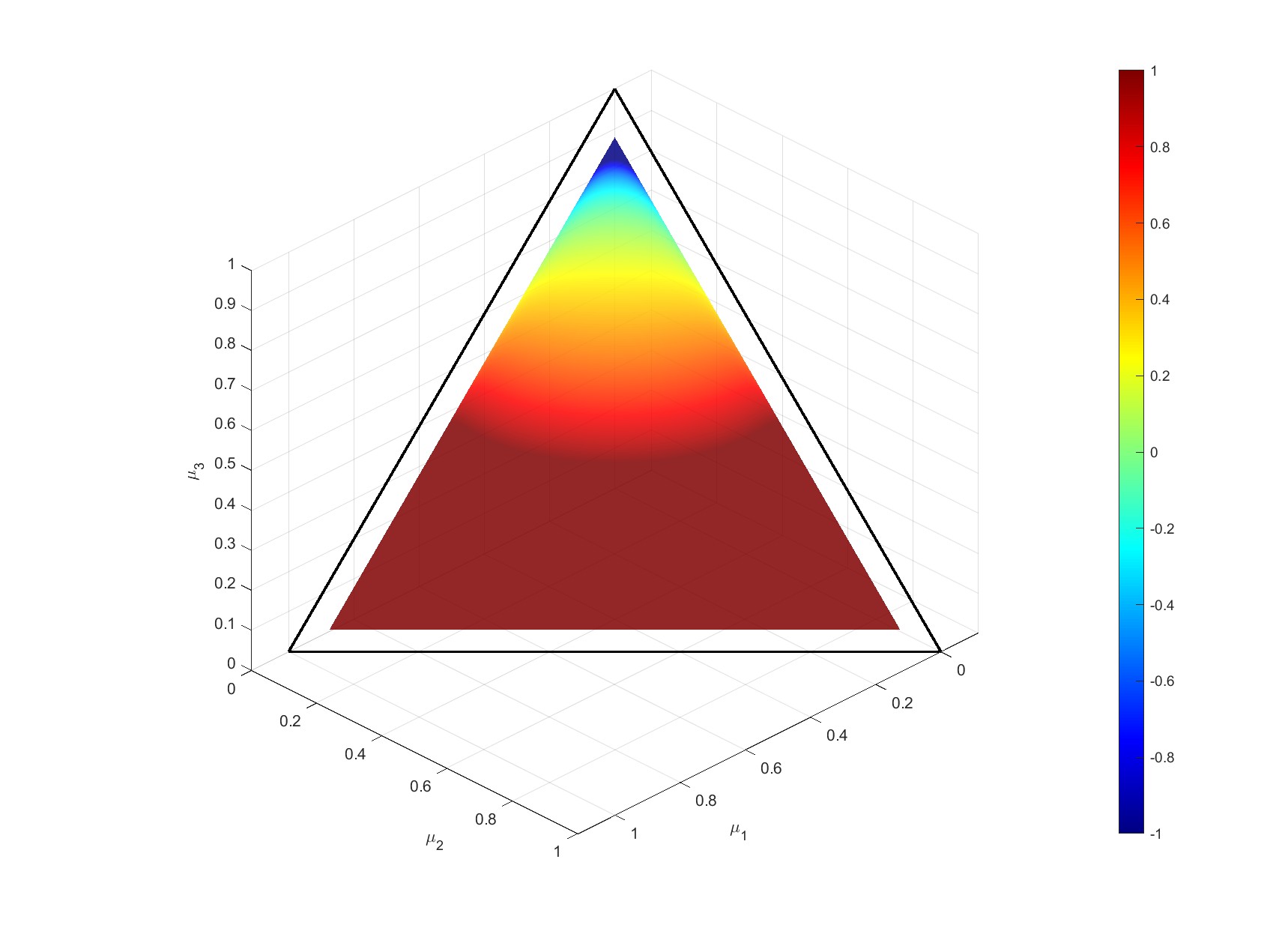}}
		\caption{Heatmap of the normalized third component $f^{\text{SEP}}_3=\pi_3/\mu_3$ for the SEP in $\mathbb{R}^3$ with $\lambda=0.885$. Subfigure (a) illustrates a long-only regime, while (b) shows that short positions emerge near a vertex.}
		\label{fig:f_3_SEP_R3}
	\end{figure}
	
	Subfigure (a) confirms the existence of the all-long regime established in Proposition \ref{prop:SEPAllLongArea}: for a sufficiently small domain ($\varepsilon=0.5$), the strategy remains long-only. Subfigure (b) illustrates the mixed phase ($\varepsilon=0.88$), where short positions emerge. Notably, compared to the SEWP, the SEP exhibits a smoother transition and, crucially, a distinct geometric pattern of short selling.
	
	\subsection{Comparative Analysis of Short-Selling Regions}
	\label{subsec:ComparativeAnalysis}
	A comparison of Figure \ref{fig:f_3_SEWP_R3} (SEWP) and Figure \ref{fig:f_3_SEP_R3} (SEP) shows a fundamental geometric difference in risk exposure. This difference is most apparent in Figure \ref{fig:f_3_SEP_R3_Negative}.
	\begin{itemize}
		\item \emph{SEWP (a ``shell'' pattern):} short-selling positions appear near the boundary of the market state space $K_\varepsilon$.
		\item \emph{SEP (a ``lobe'' pattern):} short-selling positions are concentrated in distinct ``lobes'' near the vertices, while the strategy remains long near the centers of the faces.
	\end{itemize}
	
	\begin{figure}[H]
		\centering
		\subfloat[Short-selling ``shell'' of SEWP]{
			\includegraphics[width=0.45\linewidth]{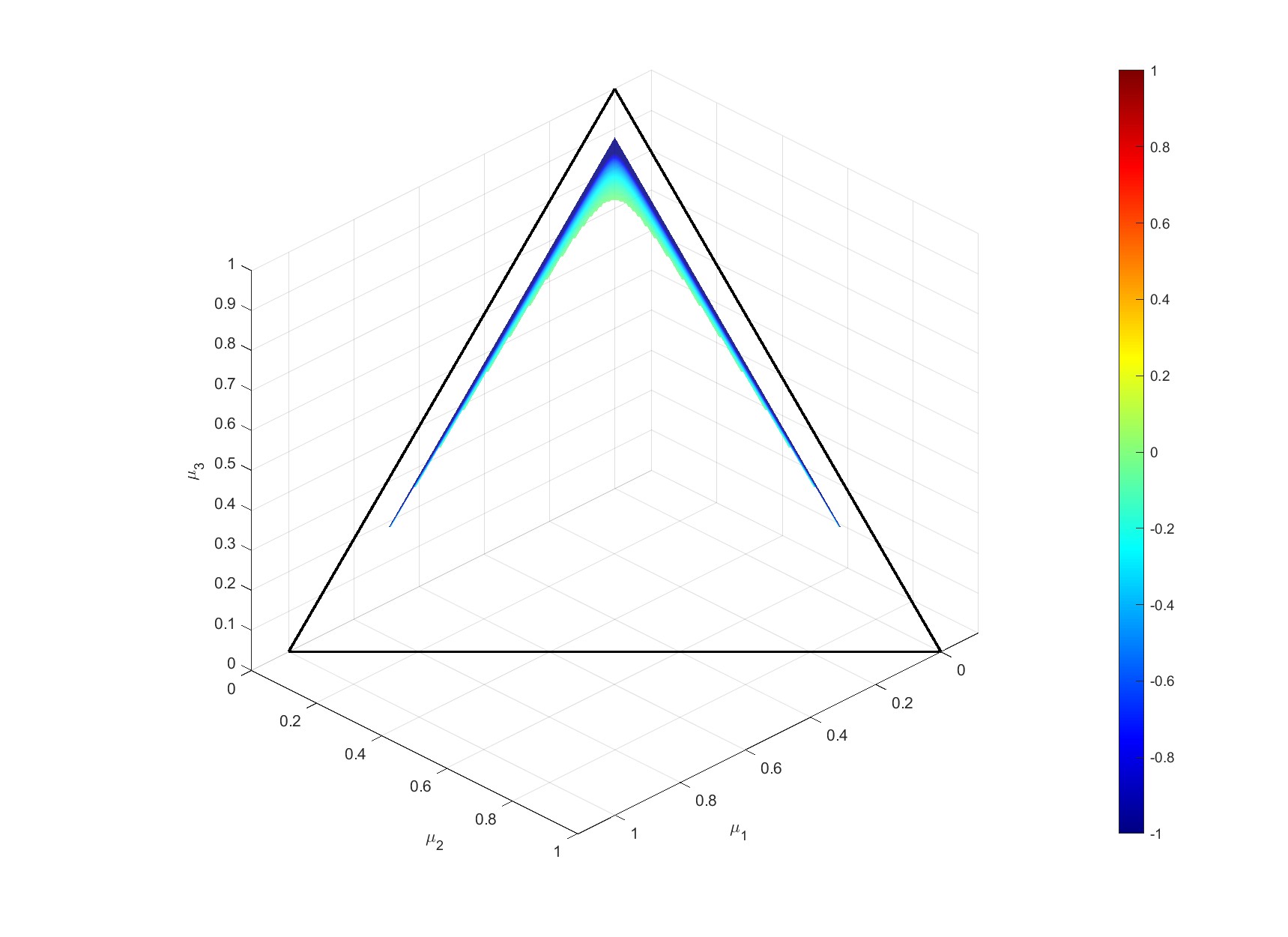}}
		\hspace{0in}
		\subfloat[Short-selling ``lobe'' of SEP]{
			\includegraphics[width=0.45\linewidth]{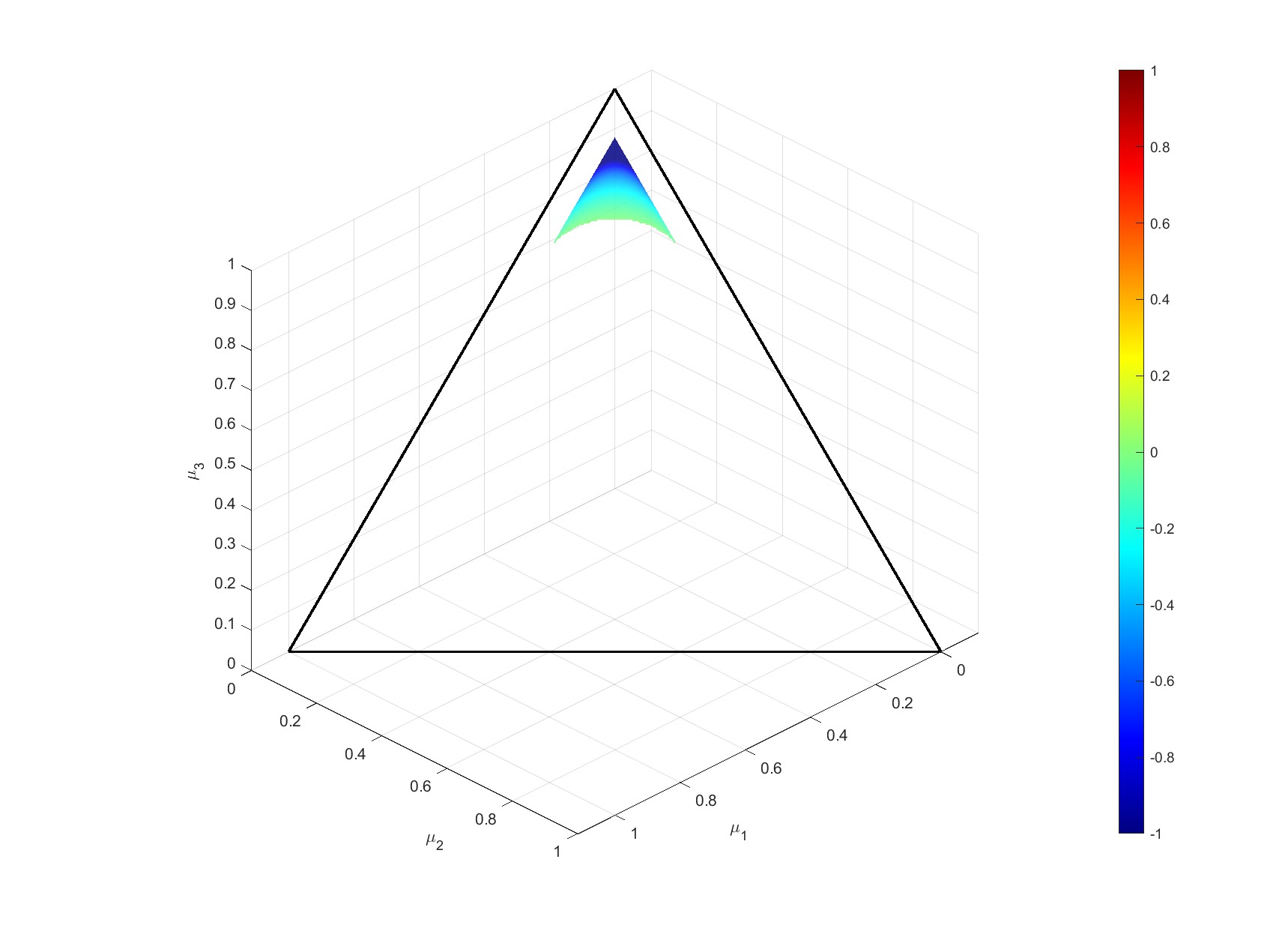}}
		\caption{Short-selling regions in heatmaps of the normalized third component $f^{\text{SEWP}}_3$ and $f^{\text{SEP}}_3$ for the SEWP and SEP in $\mathbb{R}^3$ with $\lambda=0.885$ and $\varepsilon=0.88$. Subfigure (a) illustrates the ``shell'' pattern, while (b) shows the ``lobe'' pattern near a vertex.}
		\label{fig:f_3_SEP_R3_Negative}
	\end{figure}
	
	This difference is theoretically grounded in Proposition \ref{prop:LogGradientDivergence} and Corollary \ref{coro:ShortSellingNearBoundary}. The emergence of short selling is driven by the divergence of the logarithmic directional derivative, equivalently the divergence of the ratio $\nabla \Phi / \Phi = \nabla \log \Phi$ on the region where $\Phi>0$. As stated in Corollary \ref{coro:ShortSellingNearBoundary}, this divergence occurs along rays approaching boundary points $q$ where the generating function vanishes ($\Phi(q)=0$) on $\partial D_\lambda$. Since $K_\varepsilon\subset D_\lambda$ and $\varepsilon$ close to $\lambda$ implies that $K_\varepsilon$ lies close to $\partial D_\lambda$, the location of the zero set of $\Phi$ on $\partial D_\lambda$ determines where short-selling positions are observed within $K_\varepsilon$.
	
	\begin{enumerate}
		\item For the SEWP: The base function is the geometric mean $\Psi_{\text{EWP}}(x) = (\prod x_i)^{1/n}$. This function vanishes on the \emph{entire boundary} $\partial \Delta^{(n)}$ (whenever some $x_i=0$). Because $S_\lambda^{-1}$ maps $\partial D_\lambda=S_\lambda(\partial\Delta^{(n)})$ onto $\partial\Delta^{(n)}$, it follows that $\Phi_{\text{SEWP}}=\Psi_{\text{EWP}}\circ S_\lambda^{-1}$ vanishes on the entire boundary $\partial D_\lambda$. Consequently, the divergence mechanism in Corollary \ref{coro:ShortSellingNearBoundary} can be triggered as the market approaches any portion of $\partial D_\lambda$, which is consistent with the ``shell'' pattern observed in Figure \ref{fig:f_3_SEWP_R3}.
		
		\item For the SEP: The base function is the Shannon entropy $\Psi_{\text{Ent}}(x) = -\sum x_i \log x_i$. Crucially, entropy is strictly positive on the relative interior of each face and vanishes \emph{only at the vertices} (where one $x_i=1$ and all others are $0$). Therefore, the zero set of $\Phi_{\text{SEP}}=\Psi_{\text{Ent}}\circ S_\lambda^{-1}$ on $\partial D_\lambda$ consists exactly of the vertices $\{S_\lambda(e_i)\}_{i=1}^n$. For any boundary point $q\in\partial D_\lambda$ that is not a vertex, one has $\Phi_{\text{SEP}}(q)>0$.
		Along a ray approaching a point in the relative interior of a face of $\partial D_\lambda$, $\Phi_{\text{SEP}}$ converges to a strictly positive limit. In this case the sufficient blow-up mechanism described in Proposition \ref{prop:LogGradientDivergence} does not apply. The strong divergence condition is met \emph{only} when approaching the vertices, which is consistent with the ``lobe'' pattern.
	\end{enumerate}
	
	In financial terms, the SEWP is sensitive to a broad range of boundary-approaching configurations, while the SEP is primarily sensitive to extreme concentration in a single asset. Equivalently, the SEP can remain long-only in markets that are volatile but retain effective diversity across several assets, whereas the SEWP tends to produce short-selling positions under a wider variety of diversity reducing moves toward the boundary. Figure \ref{fig:f_3_SEP_R3_Negative} also indicates that the SEP tends to take larger short positions as the market approaches a near-monopoly configuration, meaning that one asset weight becomes close to one.
	
	\section{Empirical Examples}
	\label{sec:EmpiricalExamples}
	
	\subsection{Data and Methodology}
	We evaluate the empirical performance of the SEWP (Definition \ref{def:SEWP}) and the SEP (Definition \ref{def:SEP}) using daily U.S.\ equity data from the CRSP database. Our sample spans from the first trading day of 2000 (January 3, 2000) to the last trading day of 2021 (December 30, 2021), and the data were downloaded on March 16, 2026. The raw CRSP fields used in the construction are \texttt{date}, \texttt{TICKER}, \texttt{SHROUT}, and \texttt{PRC}.
	
	In this section, market weights are constructed via \eqref{eq:MarketPortfolio}. For each stock $i$ and date $t$, we compute the market capitalization $X_i(t)=\texttt{PRC}_i(t)\cdot \texttt{SHROUT}_i(t)$, and we then normalize across the stocks in the selected submarket to obtain the market weight vector $\mu(t)\in\Delta^{(n)}$ \cite{Wong2023FunctionalPortfolioOptimization}.
	We evaluate performance in terms of the relative value process with the market portfolio as num\'eraire. Let $V_\pi(t)$ denote the ratio of the wealth of a portfolio $\pi$ to that of the market portfolio $\mu$ at time $t$. In discrete time, we compute $V_\pi$ using \eqref{eq:RelativeValue}.
	
	A key empirical input is the market stability parameter $\varepsilon$, which determines the market state space $K_\varepsilon=S_\varepsilon(\Delta^{(n)})$. In real markets $\varepsilon$ is not observed ex ante. For each selected three-stock market, we estimate $\varepsilon$ from the realized path by requiring that all observed market weights lie in $K_\varepsilon$. Concretely, we define the estimate
	\[
	\widehat{\varepsilon}\coloneqq \inf\left\{\varepsilon\in(0,1):\ \mu(t)\in K_\varepsilon\ \text{for all observed }t\right\}.
	\]
	This construction makes $\widehat{\varepsilon}$ a lower bound on the effective diversification level observed in the sample, in the sense that the market weights never leave $K_{\widehat{\varepsilon}}$.
	
	The empirical analysis is conducted in two parts. In the first part, we work with a fixed scaling parameter $\lambda$ in three-stock markets. For each market, we choose $\lambda>\widehat{\varepsilon}$ so that $K_{\widehat{\varepsilon}}\subset D_\lambda$, and we take $\lambda$ close enough to $\widehat{\varepsilon}$ to place the observed market trajectory near the zero boundary $\partial D_\lambda$ where Corollary \ref{coro:ShortSellingNearBoundary} predicts the emergence of short-selling positions.
	We visualize (i) the portfolio weights of the SEWP and SEP (denoted by $\pi^\text{SEWP}$ and $\pi^\text{SEP}$, respectively) as three-dimensional weight surfaces over the market simplex and (ii) the realized trajectory $\{\mu(t)\}$ relative to the corresponding short-selling regions.
	We also compare these results with long-only, unleveraged benchmark portfolios, namely the equal-weighted portfolio (EWP) and the entropy-weighted portfolio (EP), denoted by $\pi^\text{EWP}$ and $\pi^\text{EP}$, respectively. These benchmarks correspond to the unshrunken case $\lambda=1$, in which the SEWP and SEP reduce to the EWP and EP, respectively.
	In the second part, we fix a three-stock market and change the scaling parameter $\lambda$ to show the influence of $\lambda$ on relative values and on the prevalence of short-selling positions along the realized path.
	
	\subsection{Results and Discussion}
	By plotting the 3D portfolio weight surfaces, we aim to show that the SEWP and SEP behave differently according to how the realized market weight trajectory approaches the zero boundary $\partial D_\lambda$. This empirical design directly reflects the theoretical mechanism in Proposition \ref{prop:LogGradientDivergence} and Corollary \ref{coro:ShortSellingNearBoundary}, which link short selling to the behavior of $\nabla\Phi/\Phi$ near boundary points where $\Phi$ vanishes.
	
	Market 1 (Marathon Oil Corp (MRO), Allstate Corp (ALL), and Rockwell Automation Inc (ROK)).
	For this market, the estimated market stability parameter is $\widehat{\varepsilon}=0.84609$. We set $\lambda=0.885$, so that $K_{\widehat{\varepsilon}}\subset D_\lambda$ and the observed trajectory lies close to $\partial D_\lambda$.
	Empirically, the market weights leave the long-only core identified in Proposition \ref{prop:EWPAllLongArea} and fluctuate near the relative interiors of the edges. In this regime, the SEWP exhibits short-selling positions across a broad neighborhood of the boundary, which is consistent with the fact that $\Phi_{\mathrm{SEWP}}$ vanishes on the entire $\partial D_\lambda$. In contrast, the SEP remains predominantly long in edge-center configurations, which is consistent with the fact that $\Phi_{\mathrm{SEP}}$ does not vanish on the relative interior of any face and vanishes only at the vertices of $\partial D_\lambda$.
	
	Market 2 (Marathon Oil Corp (MRO), Allstate Corp (ALL), and Devon Energy Corp New (DVN)).
	After a comprehensive analysis of Market 1, we replace ROK with DVN to form a new market. For this market, the estimated market stability parameter is $\widehat{\varepsilon}=0.79159$, and we set $\lambda=0.843$.
	The smaller $\lambda$ places the zero boundary $\partial D_\lambda$ closer to the barycenter, which strengthens boundary effects for a given level of dispersion in the market weights. In this market, the realized trajectory exhibits more pronounced concentration episodes near the vertices. In accordance with the ``lobe'' geometry implied by the vertex-only zero set of $\Phi_{\mathrm{SEP}}$, the SEP produces larger short-selling positions primarily during near-monopoly configurations, while the SEWP continues to show short-selling positions more broadly as the trajectory approaches $\partial D_\lambda$.
	
	Our results are presented as follows.
	We first compare the behavior of portfolio weights under different three-stock markets, using the 3D weight figures to show the difference between the SEWP and SEP. After that, comparisons between short-selling strategies and long-only portfolios in each three-stock market are conducted with the help of the relative value process $V_\pi$. Finally, we fix a three-stock market to show the influence of the scaling parameter $\lambda$ on relative values.
	
	\subsubsection{Three-Stock Markets: MRO, ALL and ROK}
	\paragraph{Market Dynamics and the Long-Only Core.}
	Figure \ref{fig:TimeSeriesMarket10} depicts the evolution of market weights. The sample period covers distinct regimes of market dominance, with MRO, ALL, and ROK sequentially holding the largest market shares.
	Geometrically, Figure \ref{fig:TimeSeriesMarket10}(b) projects this trajectory onto the simplex. The light green area represents the theoretical long-only core $S_{\varepsilon^*}(\overline{\Delta^{(3)}})$ derived in Proposition \ref{prop:EWPAllLongArea}.
	Consistent with our theoretical framework, the market trajectory frequently exits this core during periods of increased asset concentration (e.g., the dominance of ALL in 2001). These excursions bring the market closer to the boundary $\partial D_\lambda$ and trigger the short-selling mechanism of the admissible strategies.
	
	\begin{figure}[H]
		\centering
		\subfloat[Time series of market weights]{
			\includegraphics[width=0.45\linewidth]{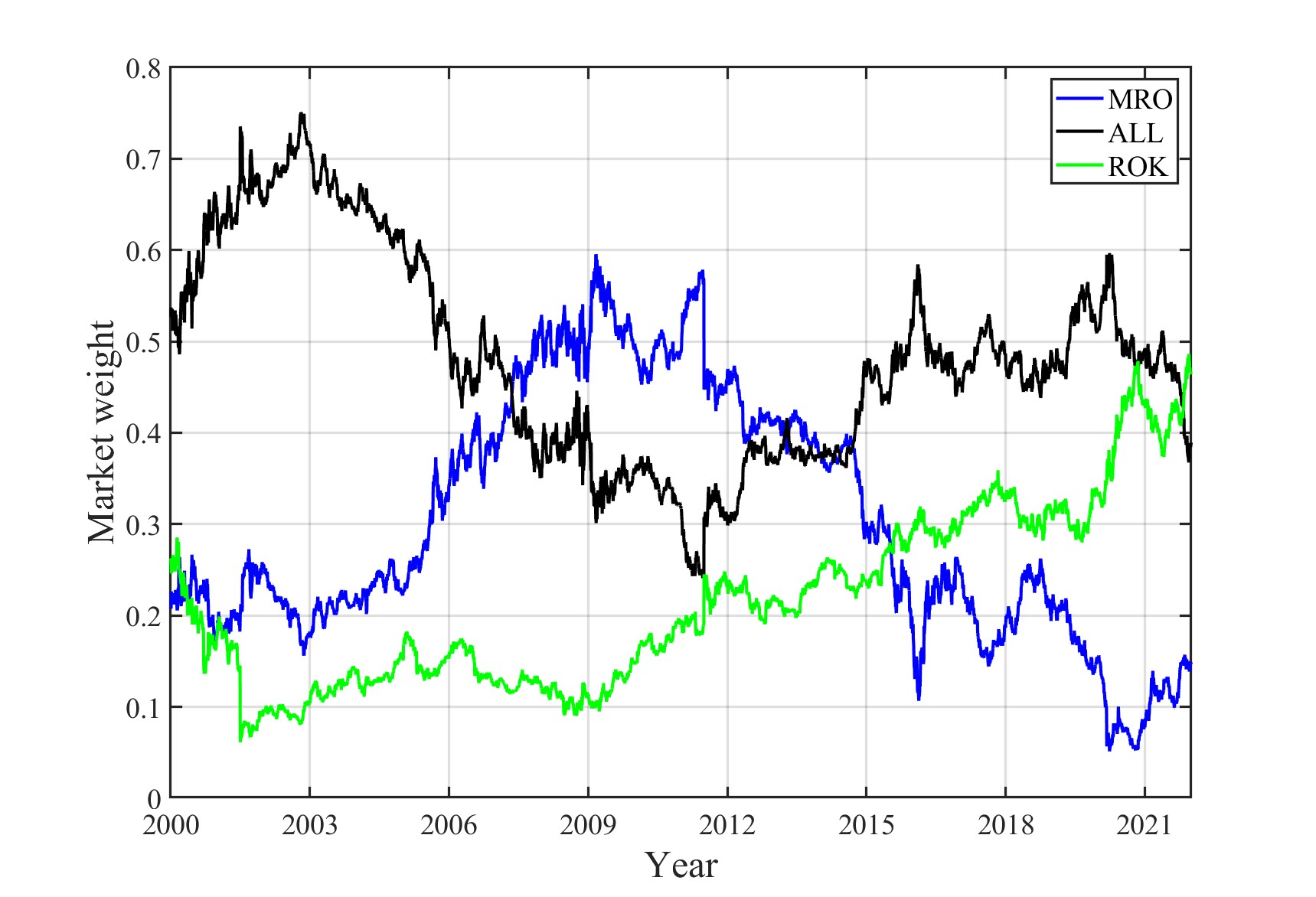}}
		\hspace{0in}
		\subfloat[Trajectory in the simplex]{
			\includegraphics[width=0.45\linewidth]{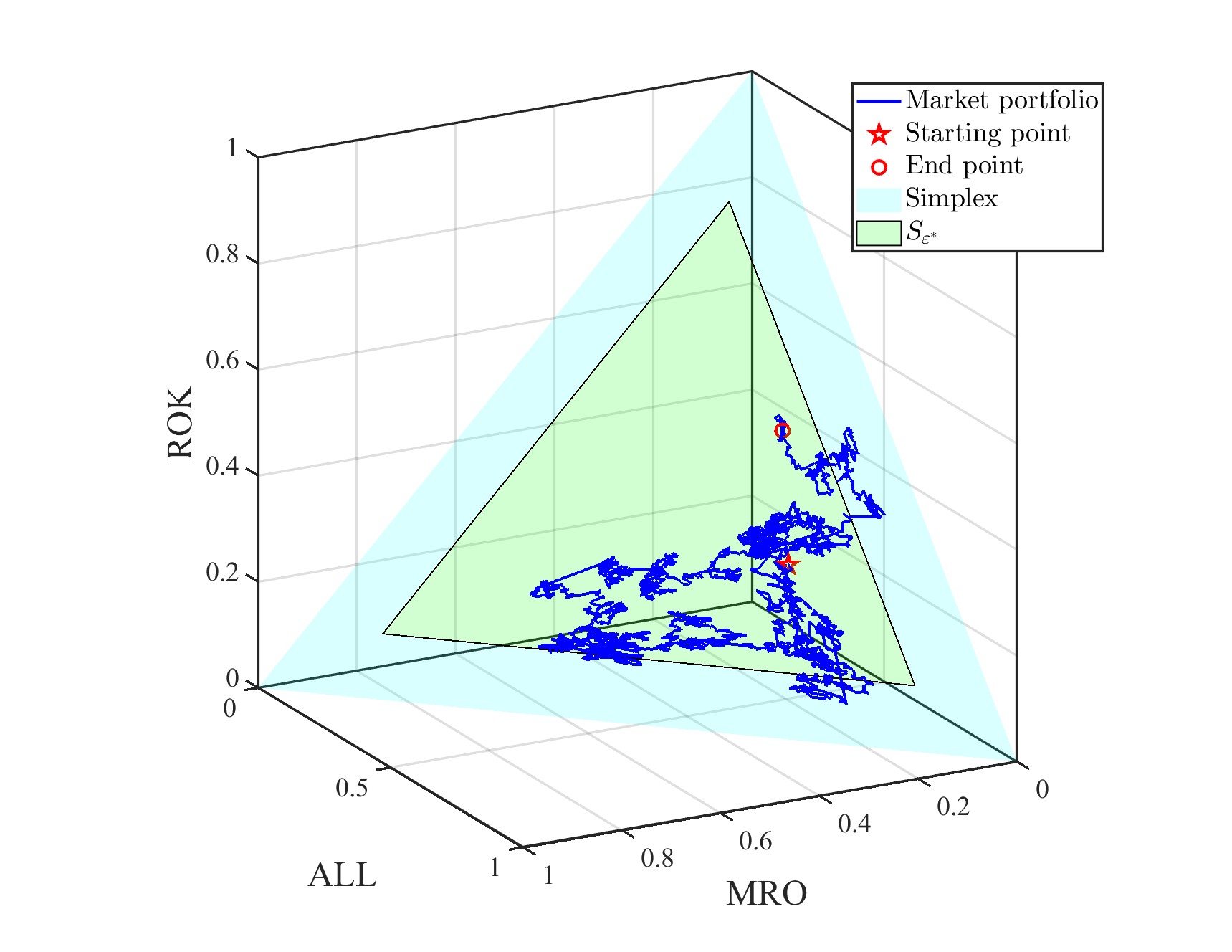}}
		\caption{Evolution of market weights (2000--2021). The green region in (b) denotes the theoretical long-only core $S_{\varepsilon^*}(\overline{\Delta^{(3)}})$ for the SEWP.}
		\label{fig:TimeSeriesMarket10}
	\end{figure}
	
	\paragraph{SEWP: Aggressive Response to Concentration.}
	Figure \ref{fig:TimeSeriesSEWP10} illustrates the behavior of the SEWP.
	Consistent with the qualitative regimes described in Section~\ref{subsubsec:SEWP}, the strategy reacts strongly to market concentration. As shown in Figure \ref{fig:TimeSeriesSEWP10}(a), when the market weight of an asset becomes sufficiently low (e.g., MRO in 2020), the strategy typically increases its exposure to that asset, and it finances this tilt by taking short positions in one or more relatively high-weight assets. Short selling can involve multiple stocks simultaneously; for instance, from October 9, 2020, to October 21, 2020, the weights for both ALL and ROK were negative.
	This is further visualized in Figure \ref{fig:TimeSeriesSEWP10}(b), where the portfolio trajectory (blue line) frequently extends outside the simplex $\Delta^{(3)}$ into the region of $H\setminus\overline{\Delta^{(3)}}$.
	Although the maximum market weight $\max_{1\le i\le 3}\mu_i$ remained below the sufficient threshold $r_1^*=0.8$ defined in Proposition \ref{prop:ShortCondition}, short selling occurred frequently. This confirms that $r_1^*$ is a sufficient but not necessary condition. 
	
	\begin{figure}[H]
		\centering
		\subfloat[SEWP weights time series]{
			\includegraphics[width=0.45\linewidth]{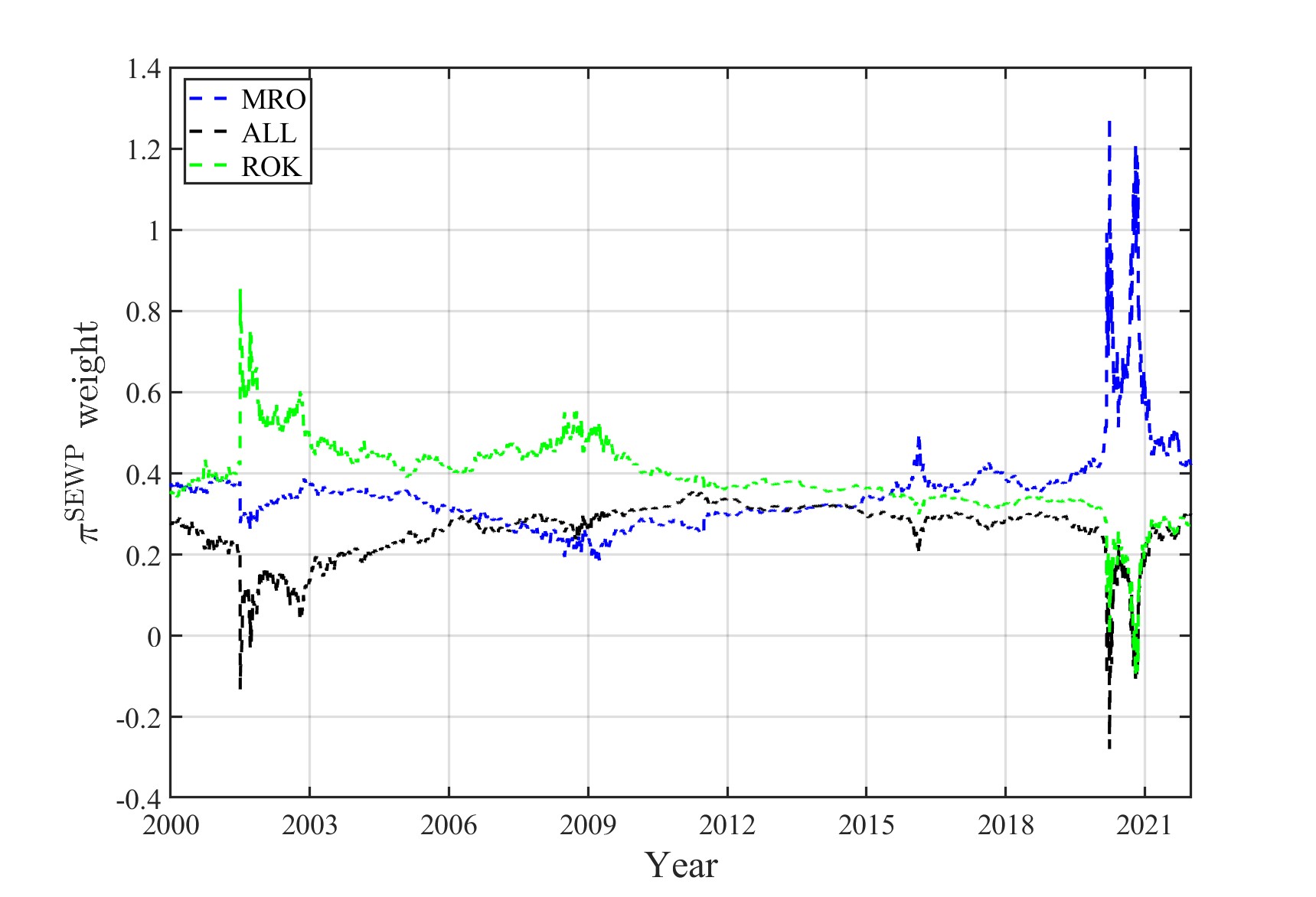}}
		\hspace{0in}
		\subfloat[SEWP trajectory in hyperplane]{
			\includegraphics[width=0.45\linewidth]{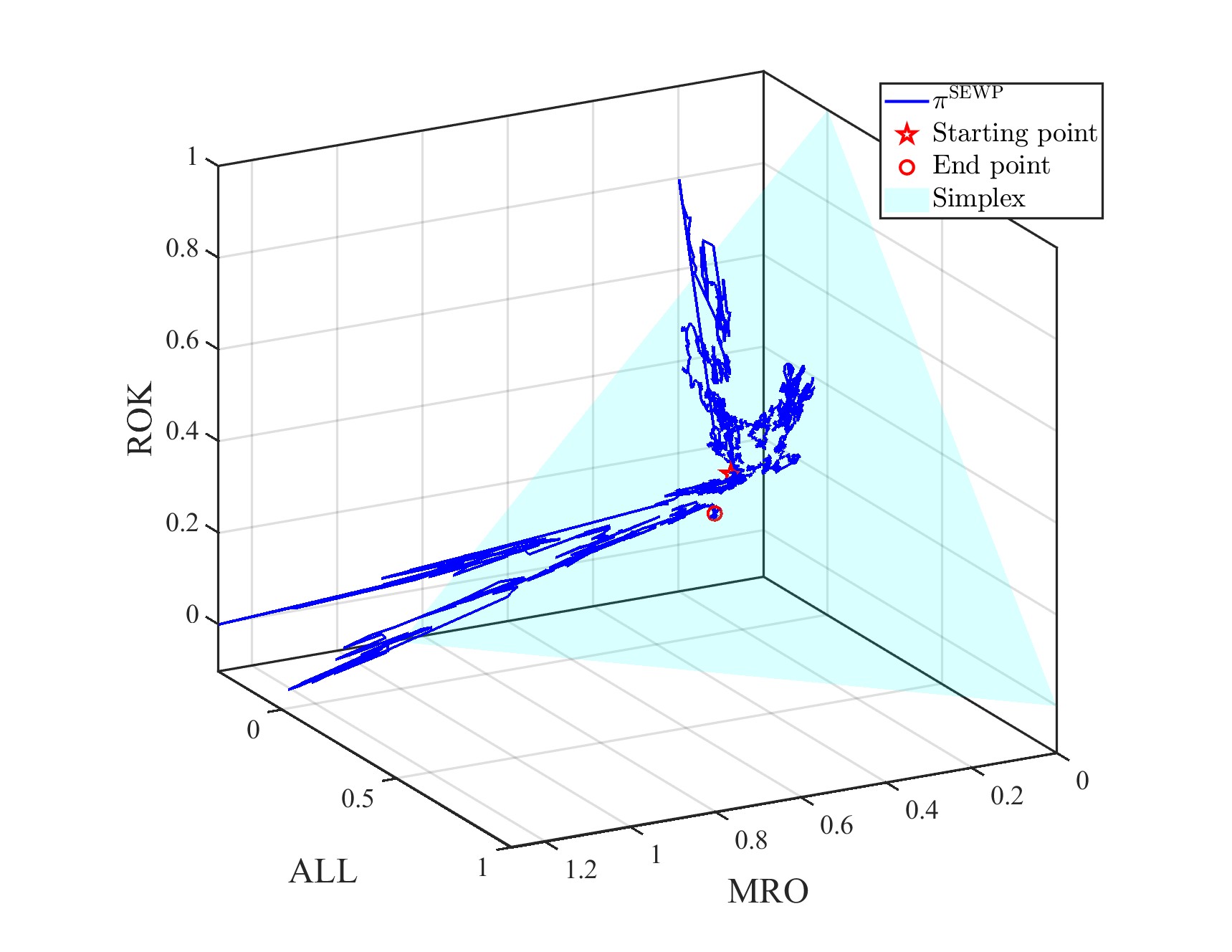}}
		\caption{SEWP performance. The strategy frequently leverages (weights $>1$) and takes short positions (weights $<0$), with the trajectory exiting the standard simplex.}
		\label{fig:TimeSeriesSEWP10}
	\end{figure}
	
	\paragraph{SEP: Conservative Stability.}
	In contrast, the SEP (Figure \ref{fig:TimeSeriesSEP10}) remains long-only throughout the entire period.
	This behavior aligns with the geometric asymptotic analysis in Section \ref{subsubsec:SEP}. Unlike the SEWP, for which the potential function $\Phi_{\mathrm{SEWP}}$ vanishes on the entire boundary, the entropy potential $\Phi_{\mathrm{SEP}}$ vanishes only at the vertices. The realized market trajectory did not approach the vertices closely enough to enter the lobe-like short-selling regions shown in Figure \ref{fig:f_3_SEP_R3_Negative}. Consequently, the SEP remained within the long-only regime.
	Figure \ref{fig:TimeSeriesSEP10}(b) confirms that the SEP trajectory stays tightly clustered within the simplex, demonstrating that short selling was not triggered for this strategy under the observed market conditions.
	
	\begin{figure}[H]
		\centering
		\subfloat[SEP weights time series]{
			\includegraphics[width=0.45\linewidth]{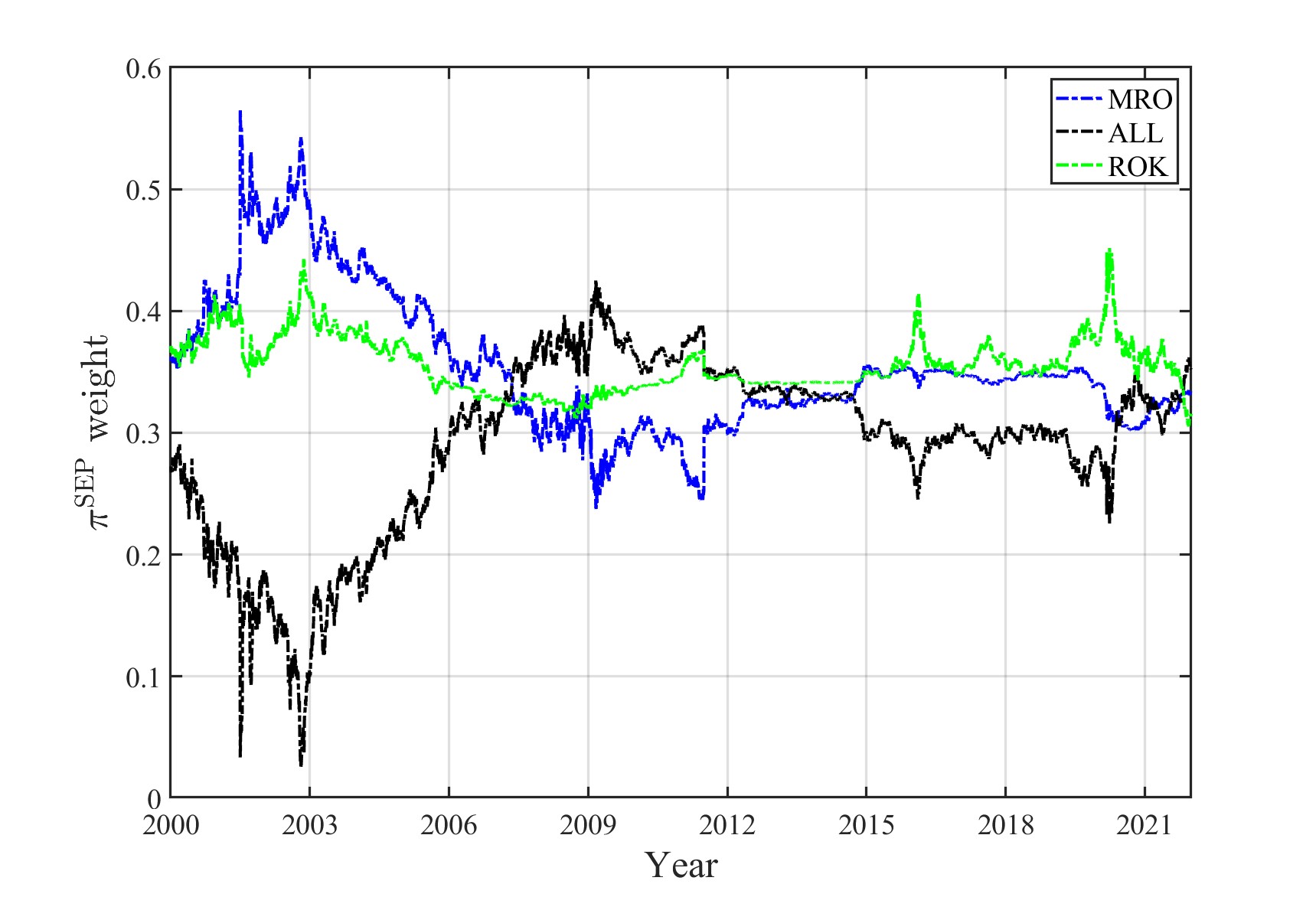}}
		\hspace{0in}
		\subfloat[SEP trajectory in hyperplane]{
			\includegraphics[width=0.45\linewidth]{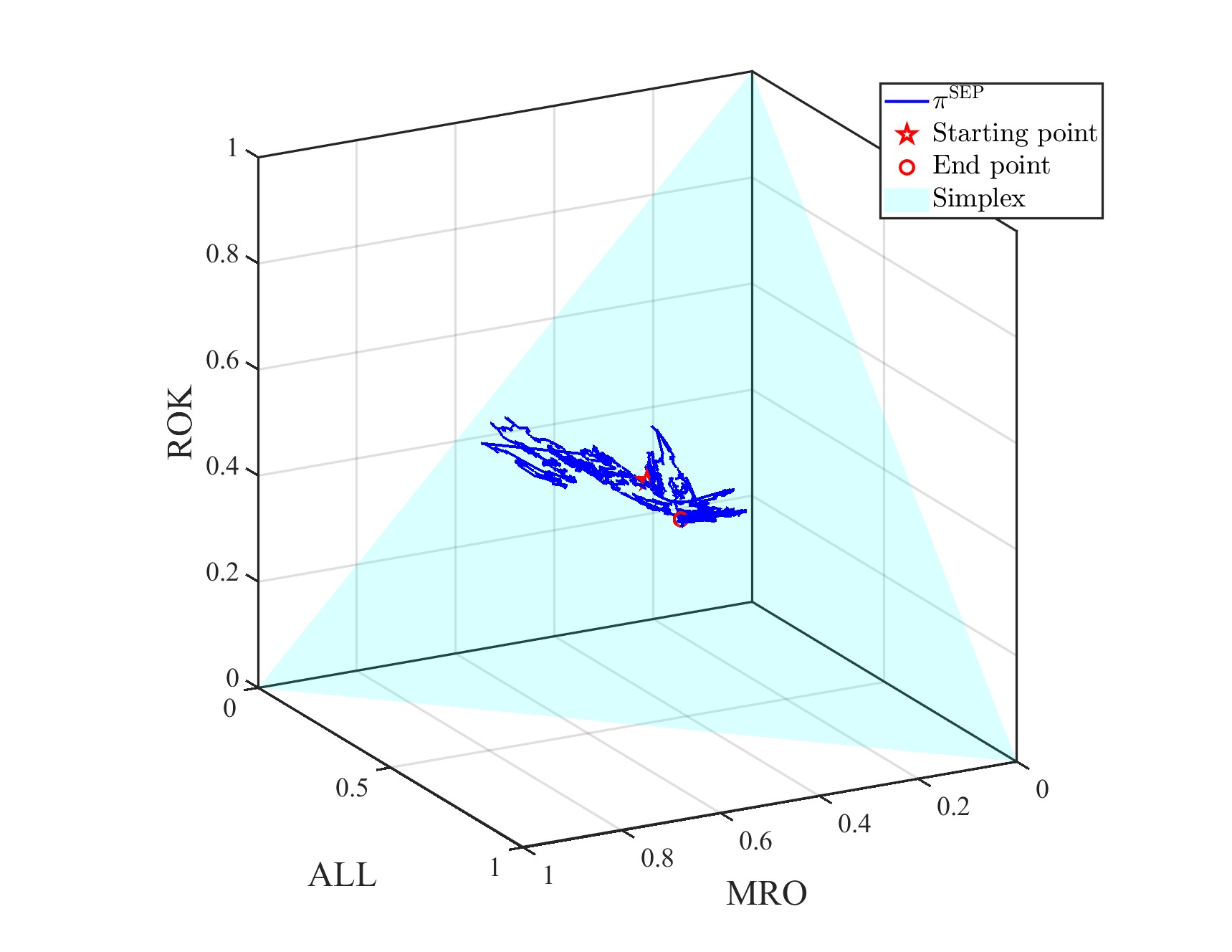}}
		\caption{SEP performance. The strategy remains long-only and contained within the simplex, reflecting the slower divergence of the gradient for the entropy-based potential away from the vertices.}
		\label{fig:TimeSeriesSEP10}
	\end{figure}

	\subsubsection{Three-Stock Markets: MRO, ALL and DVN}
	Compared with ROK, DVN has a higher market weight for most trading days. This change tends to keep the three-stock market weight trajectory closer to the interior of the simplex, rather than near an edge.
	However, since 2018, the market capitalizations of DVN and MRO declined substantially relative to ALL. As a result, by around 2020 the market weights became highly concentrated near the simplex vertex corresponding to ALL, which indicates a near-monopoly configuration.
	
	\paragraph{Market Dynamics and the Long-Only Core.}
	Figure \ref{fig:TimeSeriesMarket} depicts the evolution of market weights of MRO, ALL and DVN. Compared to the market of MRO, ALL, and ROK, it is clear from Figure \ref{fig:TimeSeriesMarket}(b) that the trajectory is mainly near the center of the simplex and away from the boundary for most of the sample. Starting from 2020, the market becomes concentrated near the vertex corresponding to ALL, which is consistent with the discussion above.
	As discussed in Proposition \ref{prop:ShortCondition}, such a near-monopoly configuration can cause the SEWP to take a short-selling position in the dominant stock (here, ALL).
	
	\begin{figure}[H]
		\centering
		\subfloat[Time series of market weights]{
			\includegraphics[width=0.45\linewidth]{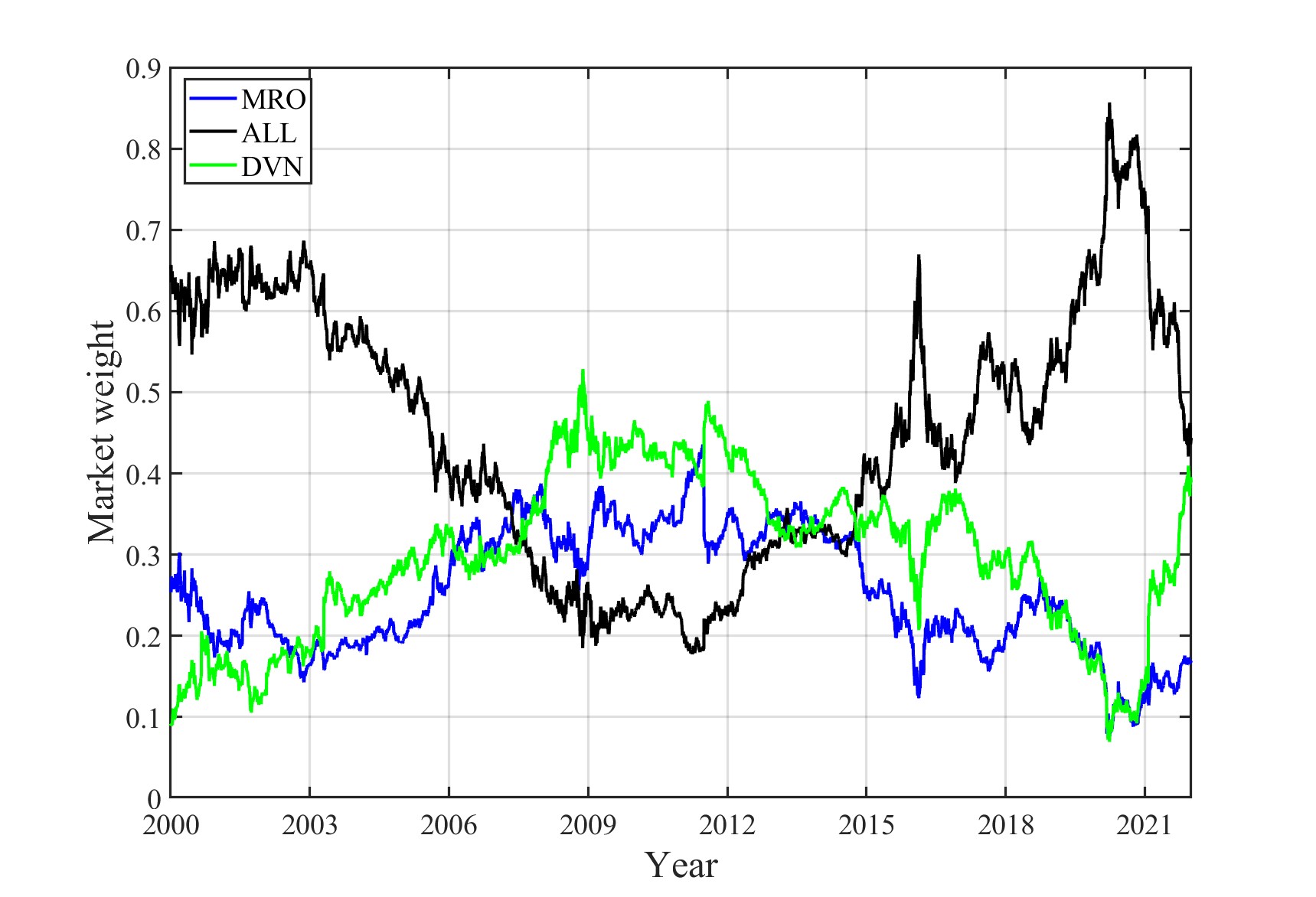}}
		\hspace{0in}
		\subfloat[Trajectory in the simplex]{
			\includegraphics[width=0.45\linewidth]{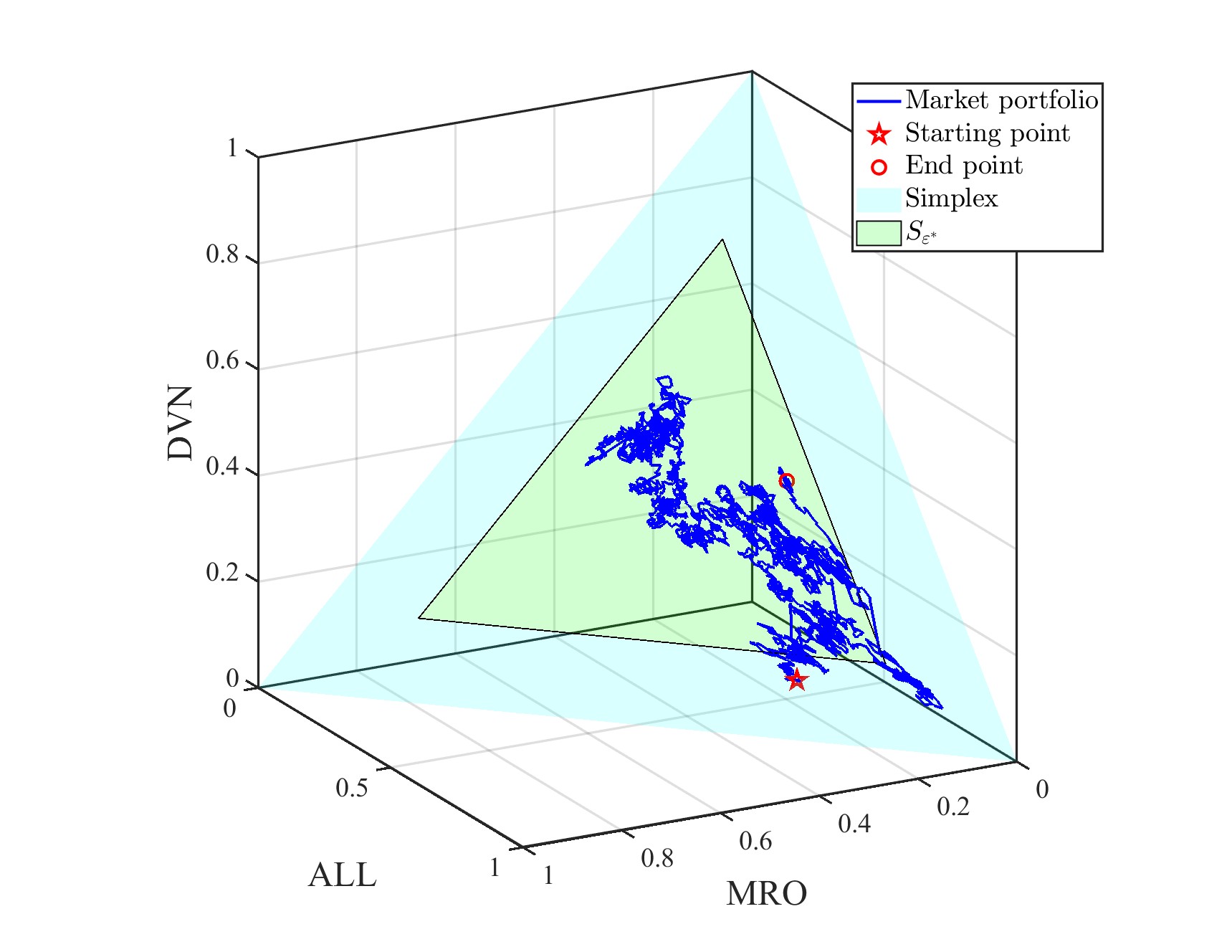}}
		\caption{Evolution of market weights (2000--2021). The green region in (b) denotes the theoretical long-only core $S_{\varepsilon^*}(\overline{\Delta^{(3)}})$ for the SEWP.}
		\label{fig:TimeSeriesMarket}
	\end{figure}
	
	\paragraph{SEWP: Aggressive Response to Concentration.}
	Figure \ref{fig:TimeSeriesSEWP} illustrates the behavior of the SEWP.
	Consistent with the sufficient regimes described in Proposition \ref{prop:ShortCondition}, the strategy reacts strongly to market concentration around a vertex in the sense that short selling can be triggered when $\max_{1\le i\le 3}\mu_i\ge r_1^*=0.74$. As shown in Figure \ref{fig:TimeSeriesSEWP}(a), when the market weight of a single asset becomes dominant (e.g., ALL between March 2020 and December 2020), the strategy takes substantial short-selling positions in the corresponding asset.
	Outside these near-monopoly episodes, the market trajectory lies mostly inside the long-only core $S_{\varepsilon^*}(\overline{\Delta^{(3)}})$, and the SEWP rarely takes short-selling positions.
	
	\begin{figure}[H]
		\centering
		\subfloat[SEWP weights time series]{
			\includegraphics[width=0.45\linewidth]{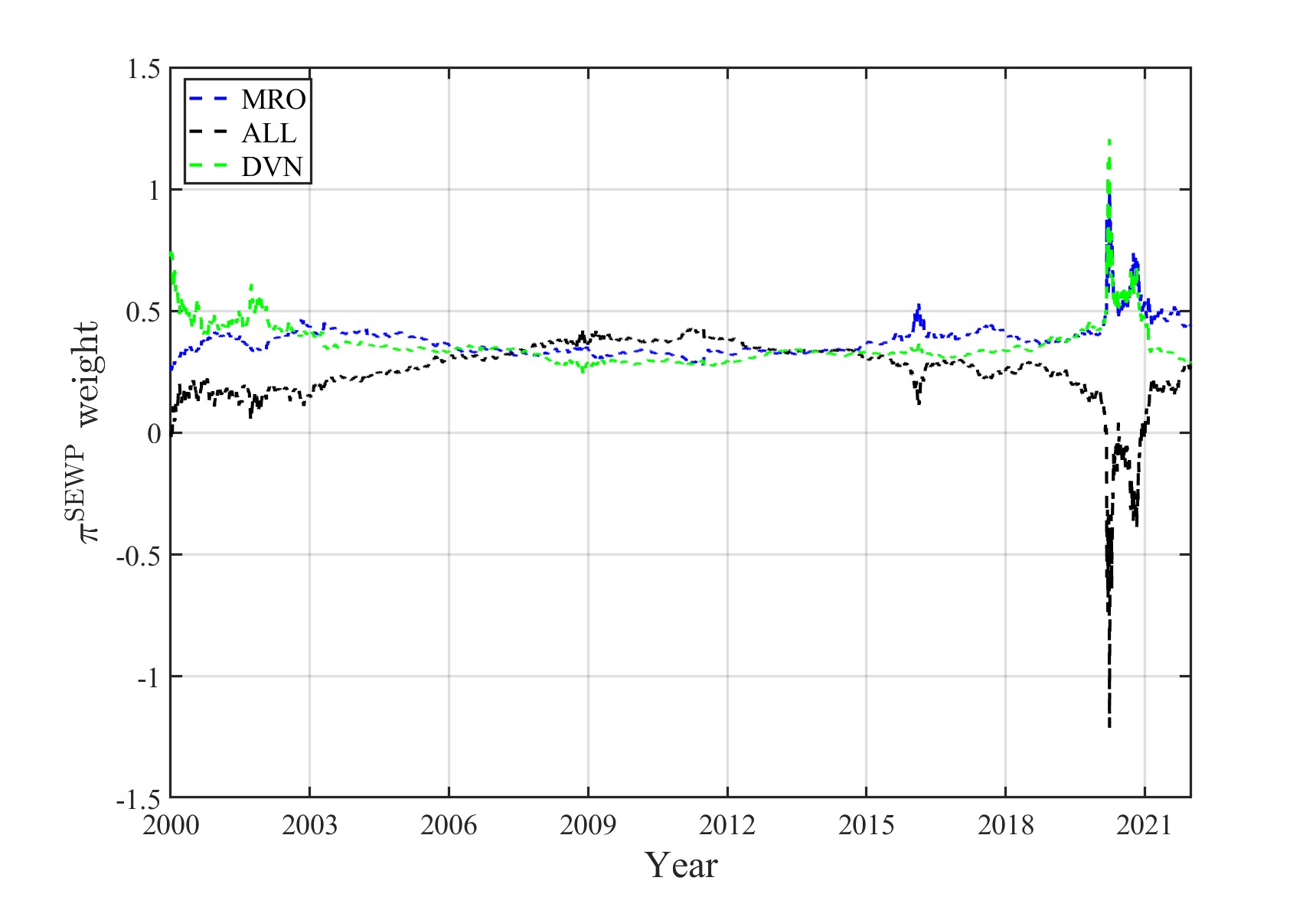}}
		\hspace{0in}
		\subfloat[SEWP trajectory in hyperplane]{
			\includegraphics[width=0.45\linewidth]{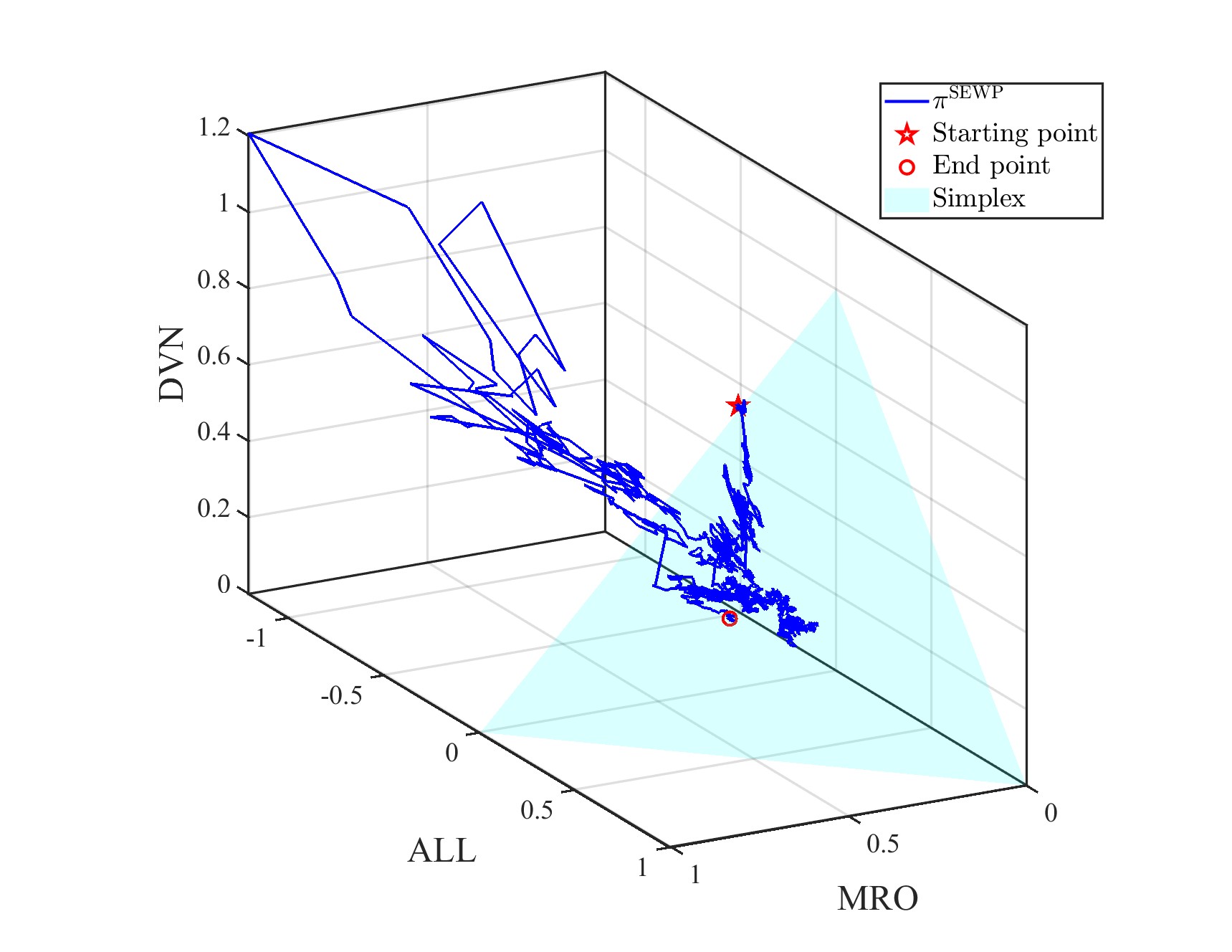}}
		\caption{SEWP performance. The strategy frequently leverages (weights $>1$) and takes short positions (weights $<0$), with the trajectory exiting the standard simplex.}
		\label{fig:TimeSeriesSEWP}
	\end{figure}
	
	\paragraph{SEP: Conservative Stability.}
	The geometric asymptotic analysis in Subsection \ref{subsubsec:SEP} implies that the SEP is primarily sensitive to vertex-approaching configurations, because its generating function vanishes only at the vertices of $\partial D_\lambda$. In contrast, the SEP is comparatively insensitive to trajectories approaching the relative interior of an edge.
	This is consistent with the heatmap in Figure \ref{fig:f_3_SEP_R3_Negative}.
	As shown in Figure \ref{fig:TimeSeriesSEP}(a), for most of the sample the SEP weights remain stable and stay within the long-only region.
	However, when the market trajectory approaches a vertex closely enough to enter the lobe-like short-selling region, the SEP takes a substantial short-selling position. Under the observed trajectory, this short position can be larger than that of the SEWP.
	
	\begin{figure}[H]
		\centering
		\subfloat[SEP weights time series]{
			\includegraphics[width=0.45\linewidth]{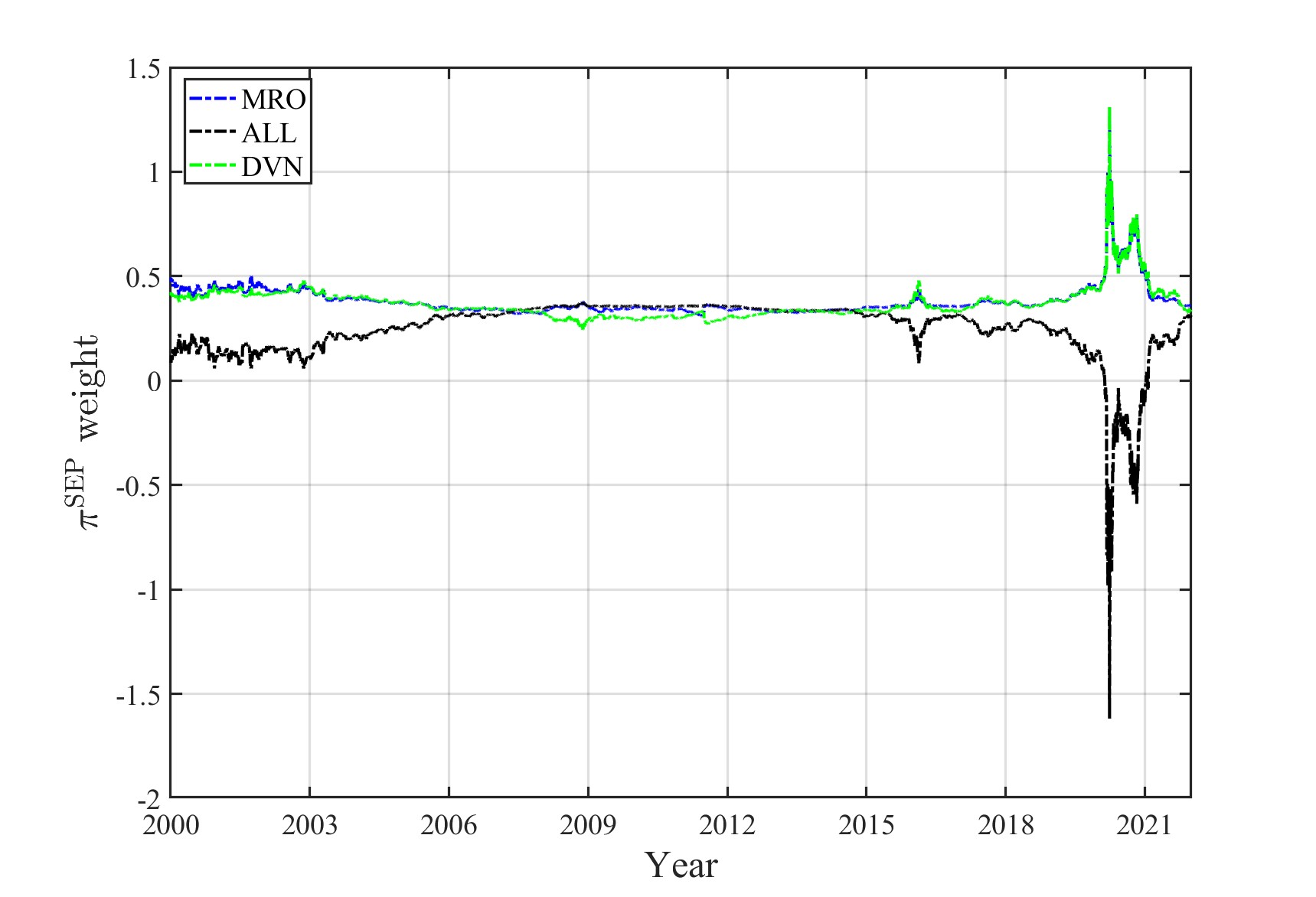}}
		\hspace{0in}
		\subfloat[SEP trajectory in hyperplane]{
			\includegraphics[width=0.45\linewidth]{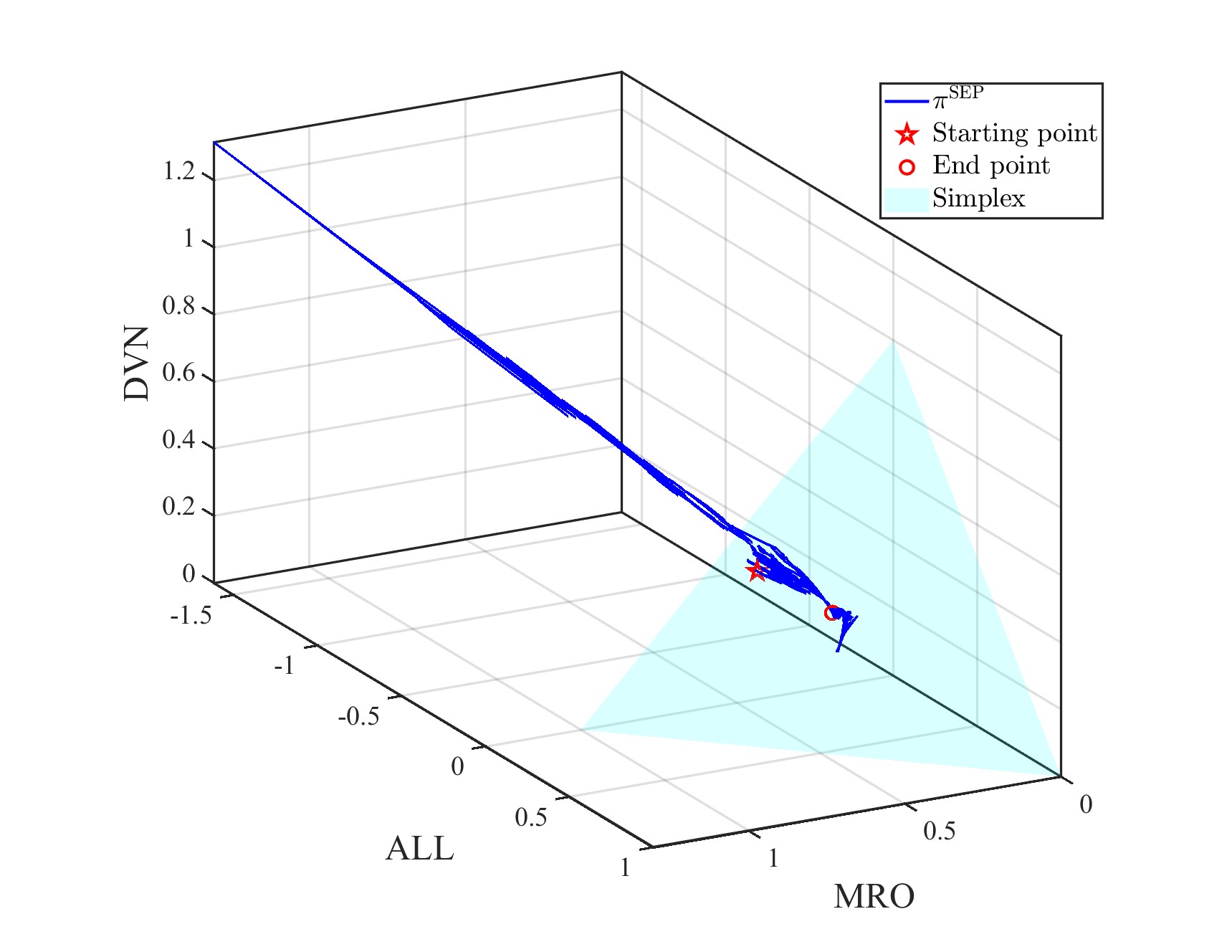}}
		\caption{SEP performance. The strategy is no longer long-only during near-vertex episodes and holds substantial short positions.}
		\label{fig:TimeSeriesSEP}
	\end{figure}
	
	\subsubsection{Relative Value with Fixed Scaling Parameter}
	\label{subsubsec:RelativeValueFixed}
	\paragraph{Relative Value and the Trade-off between Risk and Return.}
	Figure \ref{fig:RV_SEWP_SEP} compares the relative values of the shrunken portfolios (SEWP and SEP) against their classical counterparts (EWP and EP, respectively) in the two three-stock markets.
	
	The results indicate that the barycentric scaling transformation can amplify growth but also increases fluctuations in the relative value.
	\begin{itemize}
		\item \emph{Performance:} Both shrunken portfolios outperformed their baselines. The SEWP achieved the highest terminal relative values in both markets, significantly outpacing the SEP. This reflects its more aggressive use of leverage during persistent trends.
		For the market composed of MRO, ALL, and DVN, the empirical pattern reveals the mechanism through which functionally generated portfolios can achieve relative value growth. When a dominant stock continues to rise and the market weight approaches the boundary (e.g., ALL rose sharply from March to April 2020), the relative value of the functionally generated portfolio incurs significant losses. Conversely, when the dominant stock declines and the market weight moves back toward the interior (e.g., ALL fell from February to March 2020), the strategy begins to realize gains. This observation is consistent with the boundary-driven mechanism described in Proposition \ref{prop:LogGradientDivergence}.
		\item \emph{Risk:} The excess return of shrunken portfolios came at the cost of deeper drawdowns. During the consecutive trading days of March 6, 2020, and March 9, 2020, our trading strategies experienced significant drawdowns in both markets. For the market composed of MRO, ALL, and ROK, the SEWP suffered a loss of 19\% during this period, the SEP incurred a loss of 10\%, whereas the EWP recorded a loss of only 10\%, and the EP experienced a loss of merely 6\%.
		In general, the shrunken portfolios exhibited maximum drawdowns approximately double those of the long-only baselines.
	\end{itemize}

	\begin{figure}[H]
		\centering
		\subfloat[SEWP vs.\ EWP in MRO, ALL and ROK]{
			\includegraphics[width=0.45\linewidth]{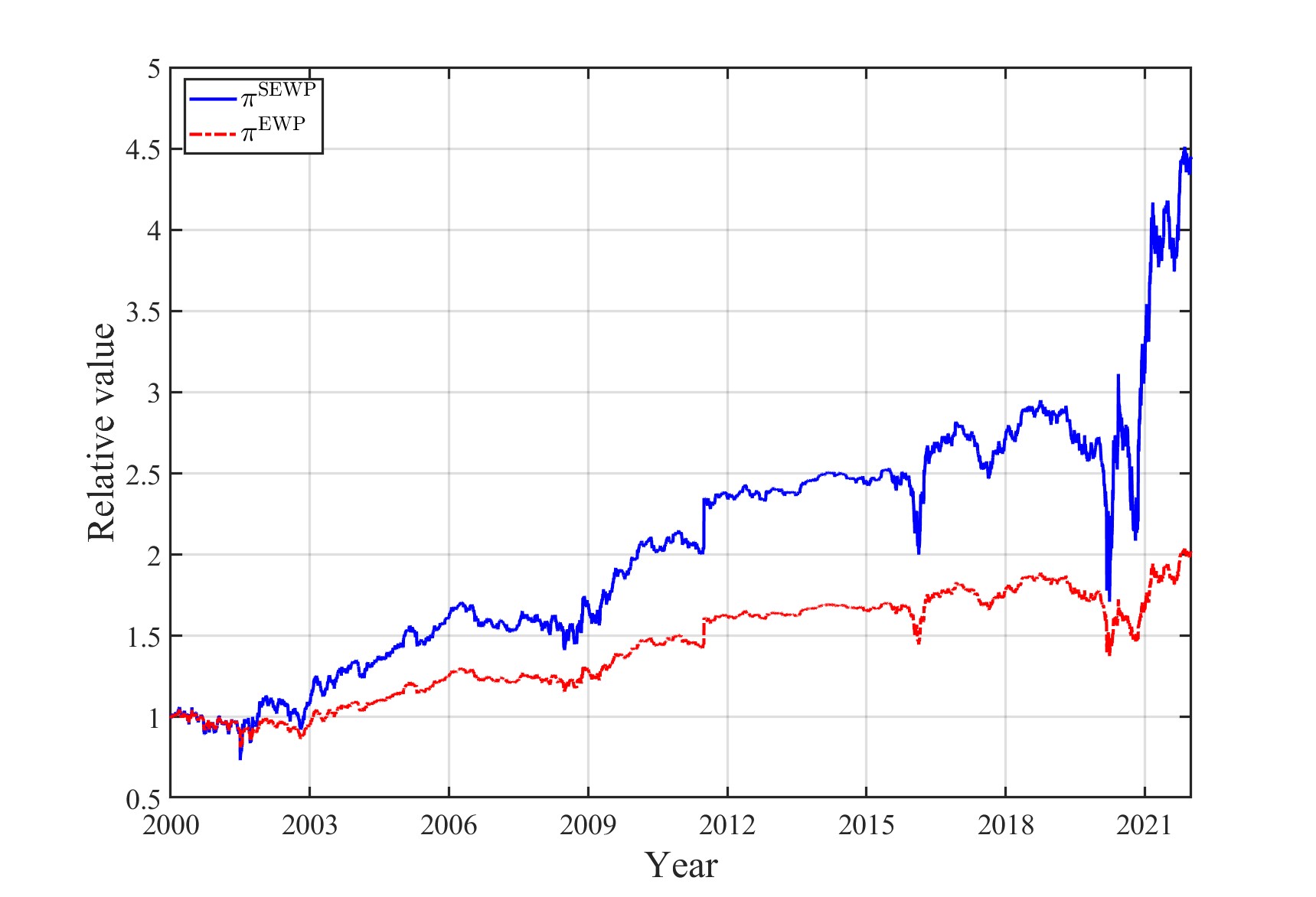}}
		\hspace{0in}
		\subfloat[SEP vs.\ EP in MRO, ALL and ROK]{
			\includegraphics[width=0.45\linewidth]{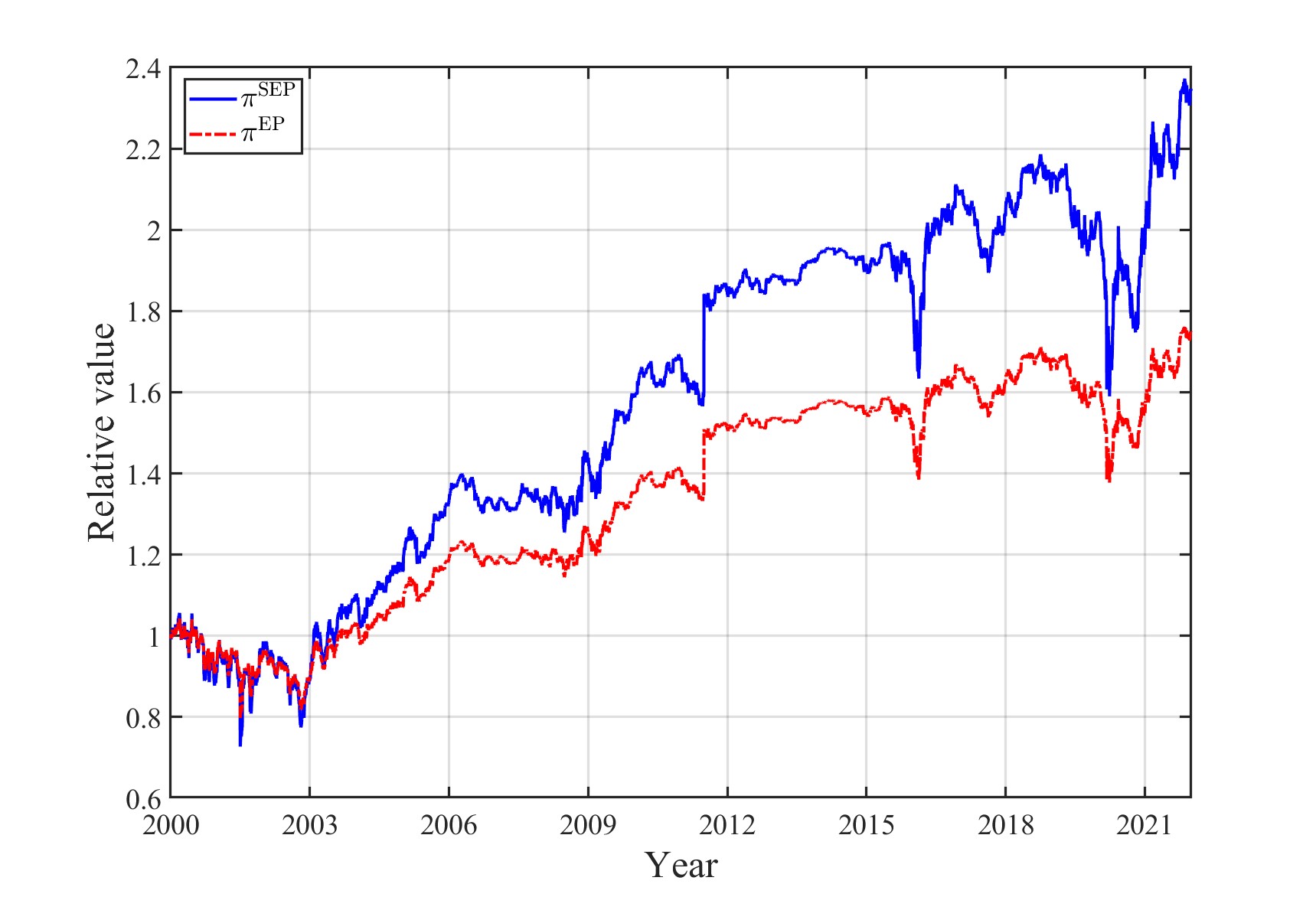}}
		\hspace{0in}
		\subfloat[SEWP vs.\ EWP in MRO, ALL and DVN]{
			\includegraphics[width=0.45\linewidth]{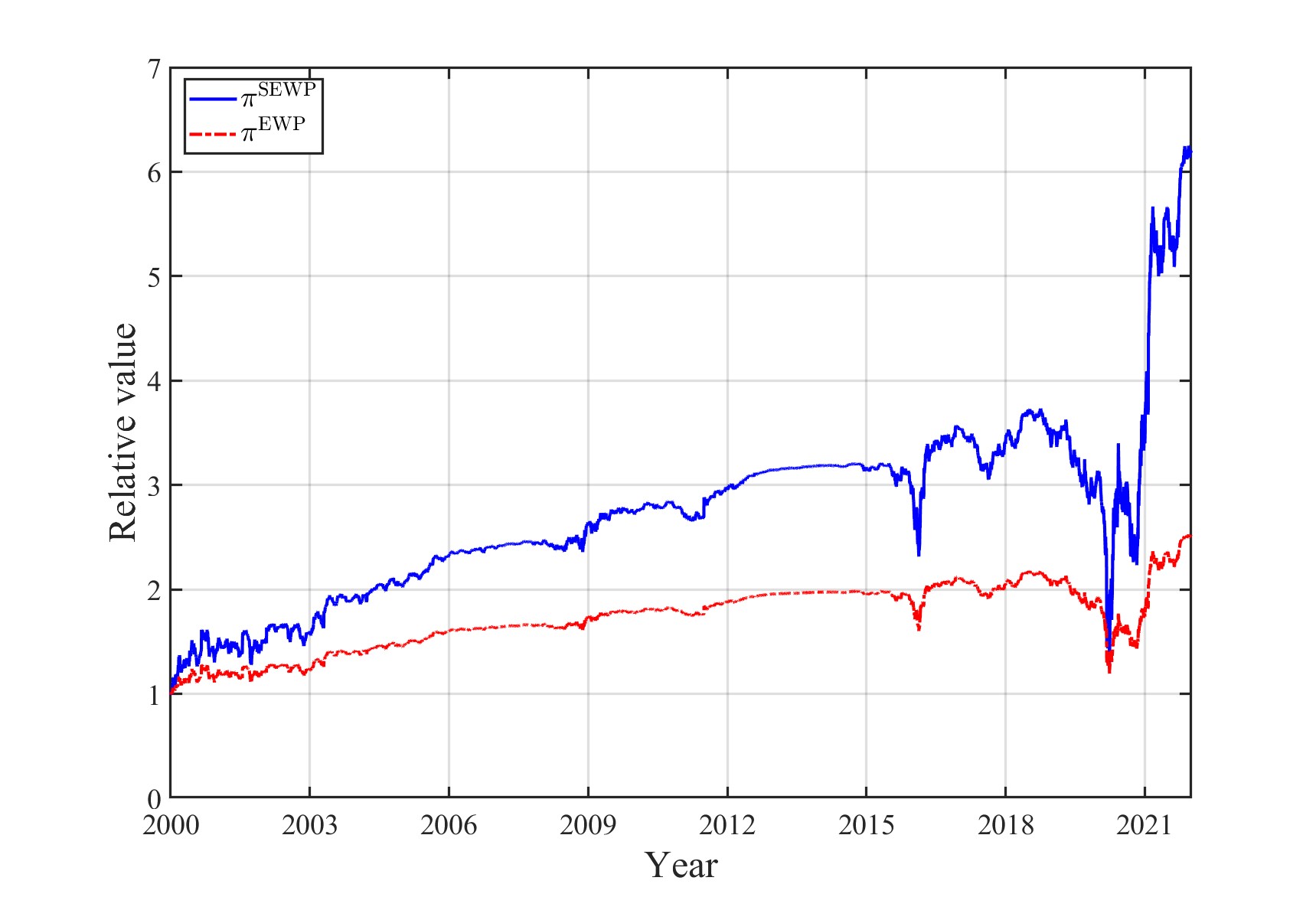}}
		\hspace{0in}
		\subfloat[SEP vs.\ EP in MRO, ALL and DVN]{
			\includegraphics[width=0.45\linewidth]{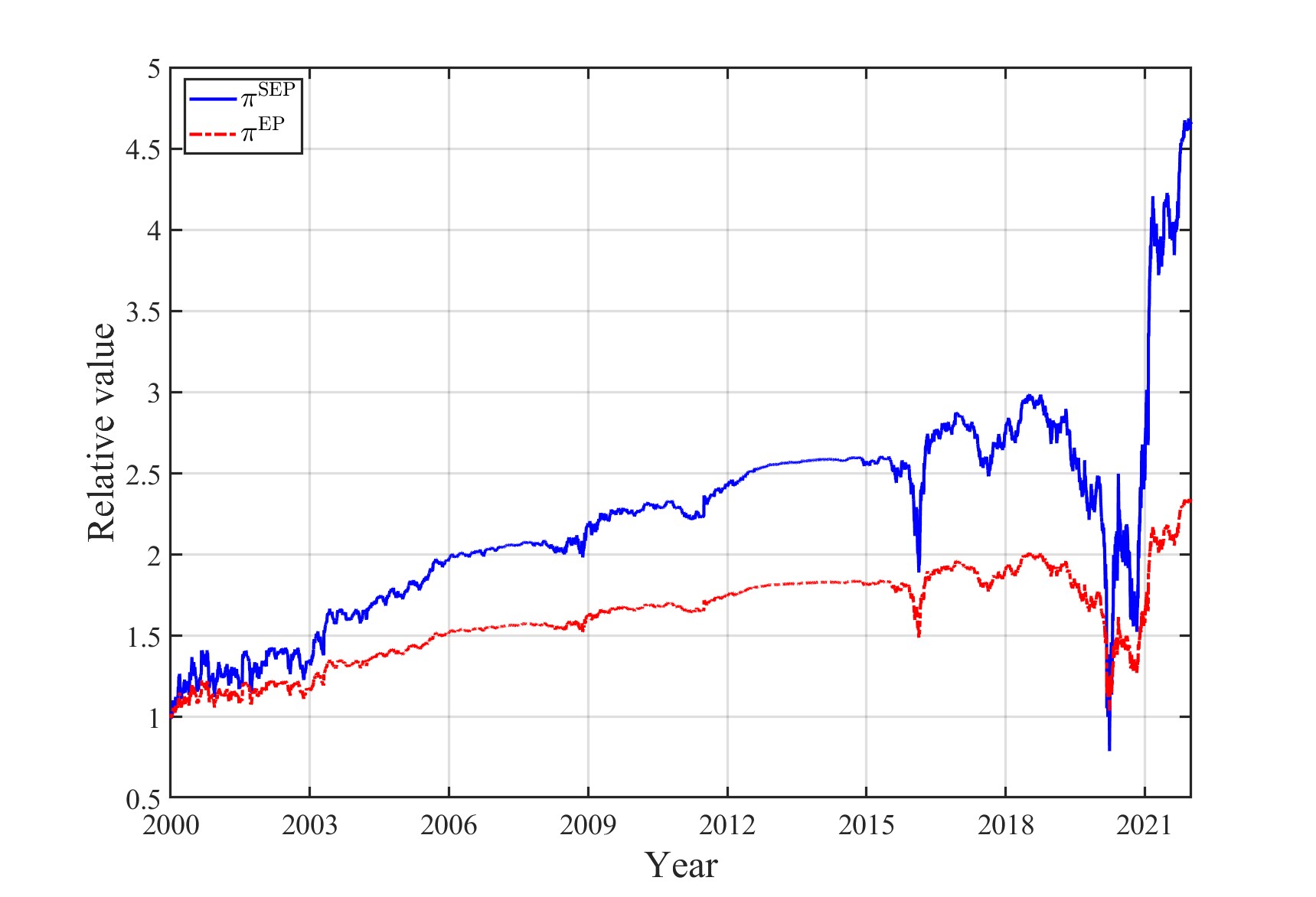}}
		\caption{Relative value trajectories in three-stock markets. Solid lines represent the admissible strategies; dashed lines represent the classical baselines.}
		\label{fig:RV_SEWP_SEP}
	\end{figure}
	
	In summary, the empirical evidence is consistent with our theoretical characterization: the SEWP is more sensitive to boundary-approaching market configurations and can deliver higher relative growth at the cost of larger drawdowns, whereas the SEP provides a more conservative alternative in volatile but effectively diversified markets.
	
	\subsubsection{Relative Value with Various Scaling Parameters}
	We now fix the three-stock market of MRO, ALL and DVN and vary the scaling parameter $\lambda\in\{0.842,0.896,0.95\}$ to illustrate how $\lambda$ controls the degree of short selling and the resulting relative performance. Figure \ref{fig:TimeSeriesDynamicLambda} visualizes the trajectories of the SEWP and SEP respectively.
	
	\begin{figure}[H]
		\centering
		\subfloat[SEWP trajectory in hyperplane]{
			\includegraphics[width=0.45\linewidth]{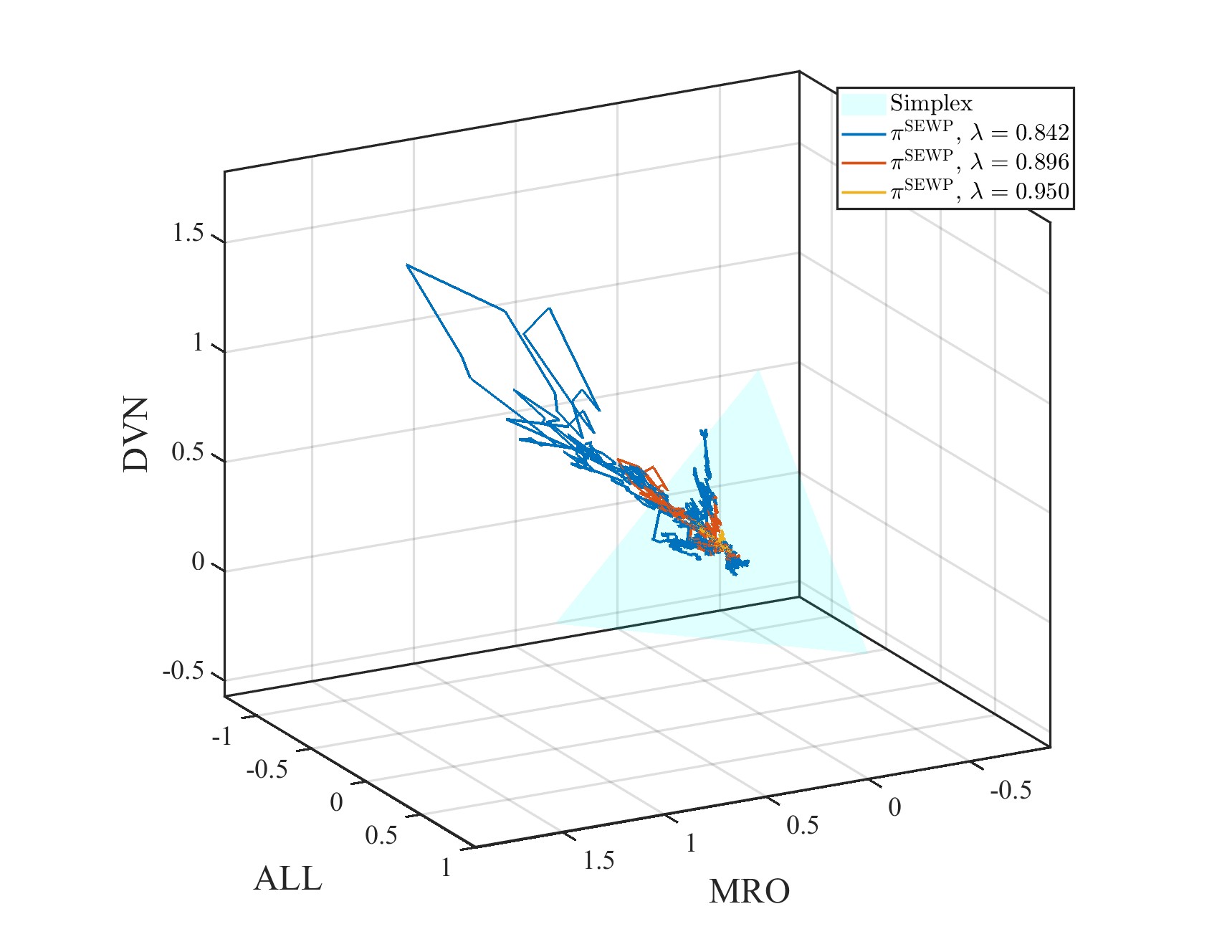}}
		\hspace{0in}
		\subfloat[SEP trajectory in hyperplane]{
			\includegraphics[width=0.45\linewidth]{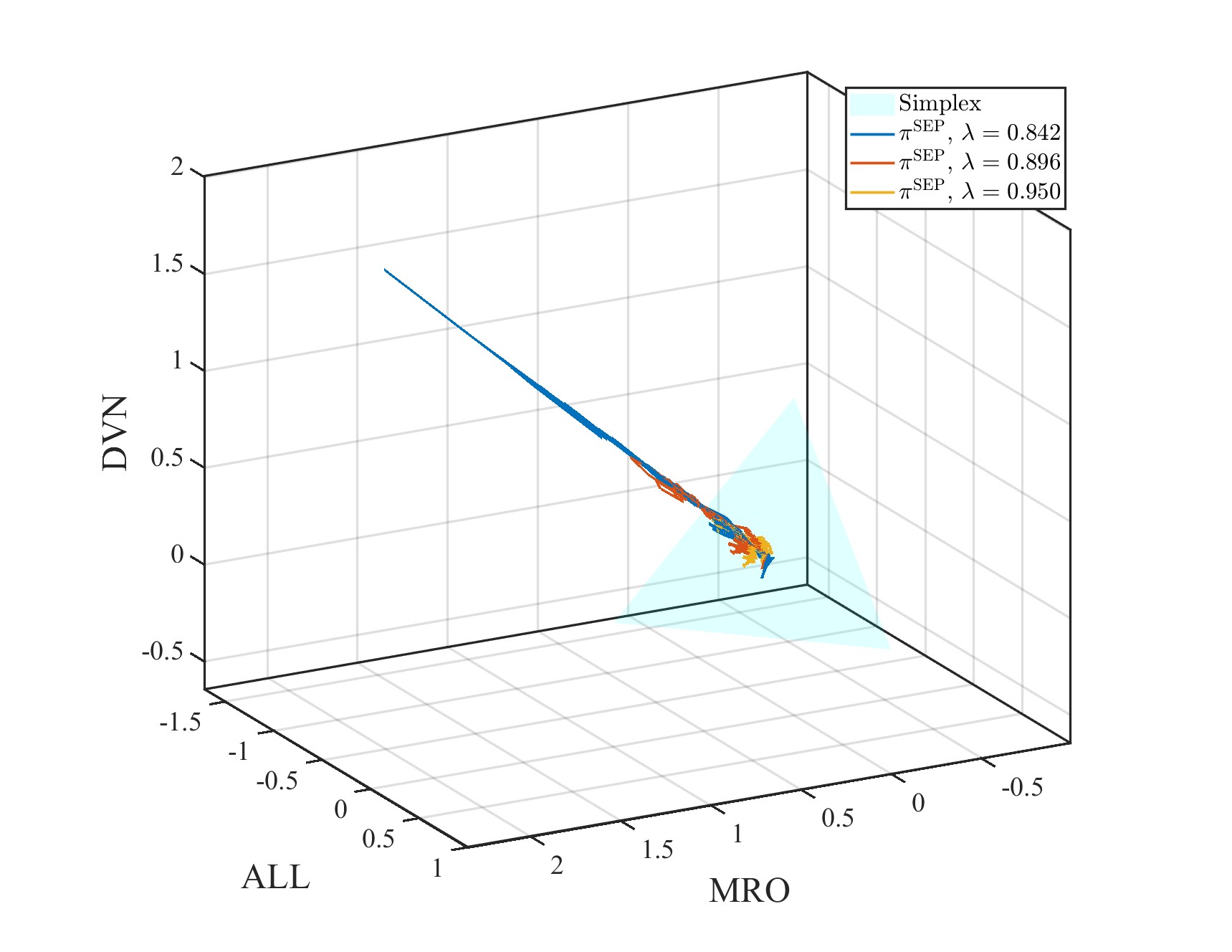}}
		\caption{Shrunken portfolio trajectories in the hyperplane with various scaling parameters $\lambda$.}
		\label{fig:TimeSeriesDynamicLambda}
	\end{figure}
	
	Figure \ref{fig:TimeSeriesDynamicLambda} shows directly that the scaling parameter $\lambda$ controls the degree of short selling. As $\lambda$ increases toward 1, the shrunken portfolios degenerate toward the long-only case because $S_\lambda(\Delta^{(3)})$ expands toward $\Delta^{(3)}$. Conversely, as $\lambda$ decreases, the generating domain $D_\lambda=S_\lambda(\Delta^{(3)})$ contracts toward the barycenter, which moves the zero boundary $\partial D_\lambda$ closer to the realized market trajectory and induces more short-selling positions in accordance with Corollary \ref{coro:ShortSellingNearBoundary}.
	
	The relative value trajectories in Figure \ref{fig:RVDynamicLambda} reflect this trade-off. On the one hand, the terminal values in both panels show that decreasing $\lambda$ amplifies relative returns. On the other hand, Figure \ref{fig:RVDynamicLambda}(b) indicates that the lower bound of $\lambda$ must be chosen with care: for the SEP with $\lambda=0.842$, the relative value fell below 1 during September to October 2020, meaning that the SEP underperformed the market portfolio during this period.
	
	\begin{figure}[H]
		\centering
		\subfloat[Relative value trajectories of SEWP]{
			\includegraphics[width=0.45\linewidth]{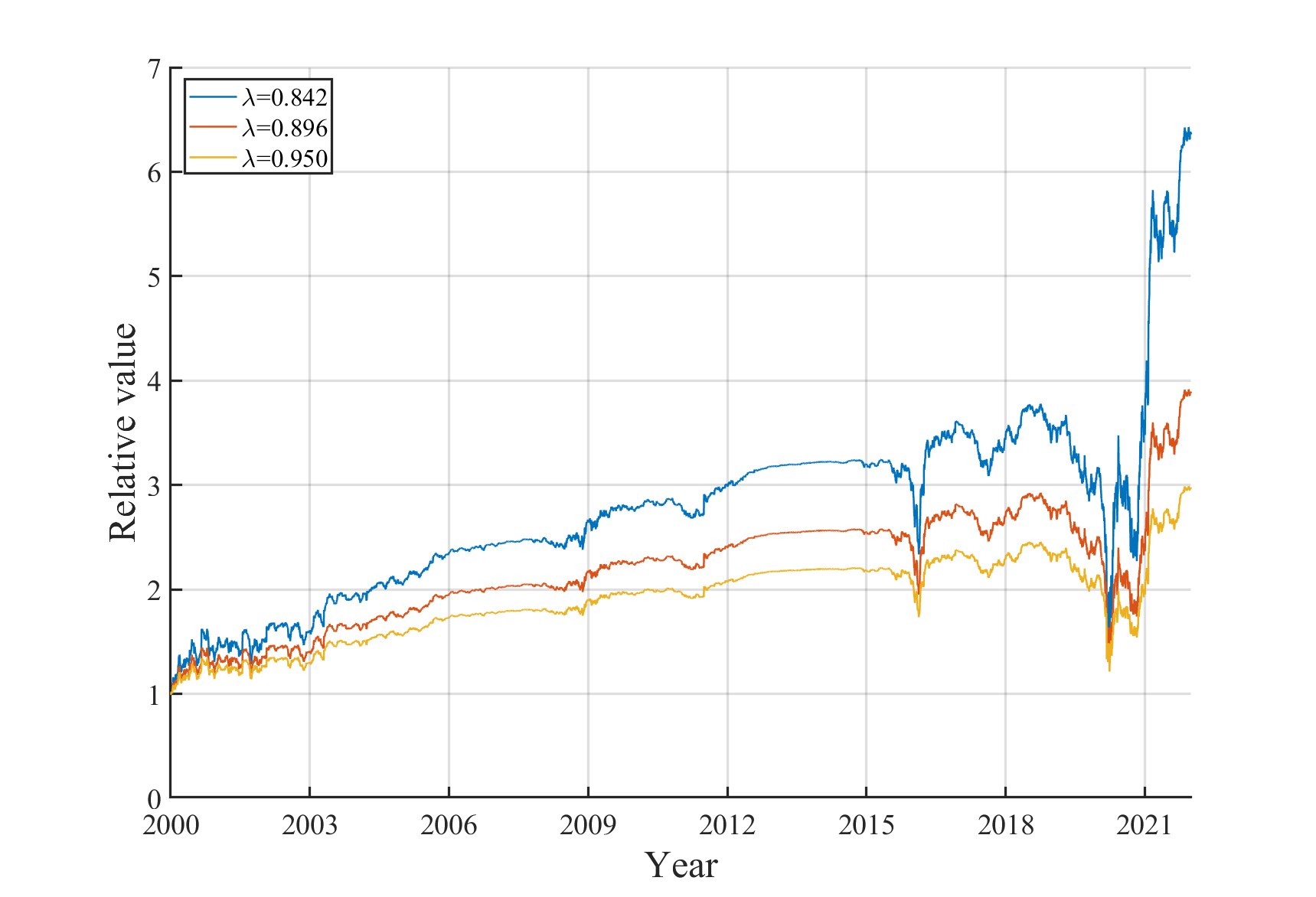}}
		\hspace{0in}
		\subfloat[Relative value trajectories of SEP]{
			\includegraphics[width=0.45\linewidth]{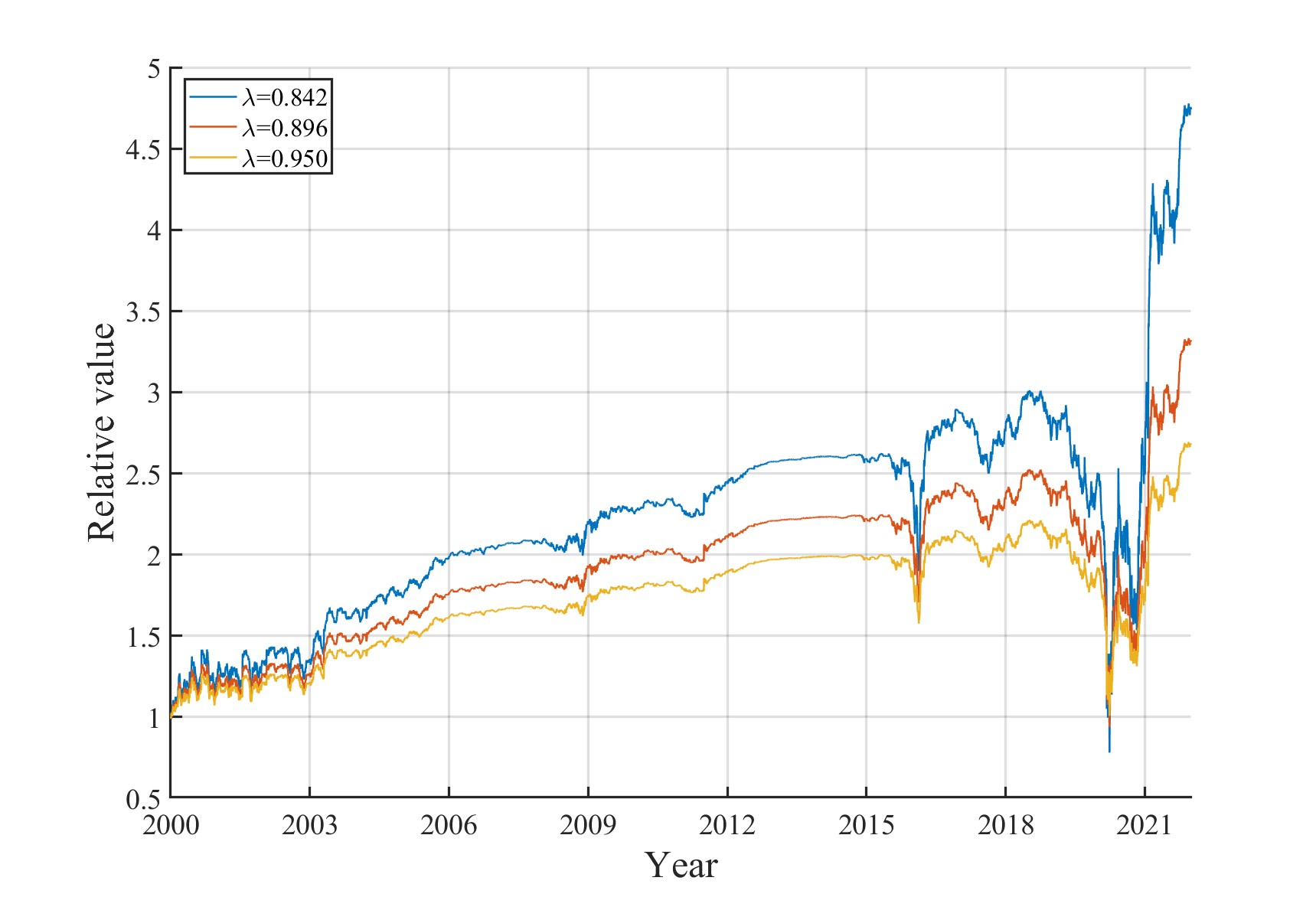}}
		\caption{Relative value trajectories of shrunken portfolios with various scaling parameter $\lambda$.}
		\label{fig:RVDynamicLambda}
	\end{figure}

	\section{Conclusion}
	This paper develops a geometric framework for constructing and analyzing admissible short-selling strategies within discrete-time Stochastic Portfolio Theory. By relaxing the domain of functional generation to a local market state space $K$, we extend the theory to a broader class of \emph{bankruptcy-proof} portfolios. This generalization demonstrates that the long-only restriction often imposed in the discrete-time literature is not intrinsic to the arbitrage mechanism, but rather a structural consequence of requiring the portfolio map to be well defined on the entire simplex.
	
	Our theoretical analysis yields three key insights.
	First, we established that pseudo-arbitrage is characterized by concavity of the generating function on $K$, which frees the construction from global constraints.
	Second, we characterized the presence of short selling by introducing the \emph{maximal concave extension}. We proved that short selling occurs if and only if this extension takes negative values, a phenomenon driven by the steepness of the logarithmic gradient near a zero boundary.
	Third, we identified a \emph{geometric phase transition}, where the interplay between symmetry of the generating potential and its boundary behavior gives rise to a long-only core and a short-selling region under suitable geometric conditions. These regions are described qualitatively and are not asserted to form an exact partition of the domain.
	
	To bridge the gap between theoretical existence and practical construction, we introduced the \emph{barycentric scaling transformation}. This methodology provides a systematic tool for engineering admissible strategies by contracting classical generating functions. Our examples illustrate a robust geometric principle: the structure of the zero set of the generating function on the boundary of the generating domain has direct implications for the geometry of short-selling exposure.
	Specifically, generating functions vanishing on the entire boundary induce a broad short-selling shell (as in the SEWP), whereas those vanishing only at vertices restrict short-selling exposure to localized regions (as in the SEP).
	
	These findings offer a unified geometric perspective on relative arbitrage, treating long-only and long-short strategies not as distinct categories, but as continuous variations controlled by the domain of definition. Future research may focus on dynamic calibration of the scaling parameter to adapt to changing volatility regimes, extending this local geometric framework to settings with transaction costs.

\newpage
\bibliographystyle{plain}
\bibliography{refs}
\end{document}